\pdfglyphtounicode{A}{0041}
\pdfglyphtounicode{AE}{00C6}
\pdfglyphtounicode{AEacute}{01FC}
\pdfglyphtounicode{AEmacron}{01E2}
\pdfglyphtounicode{AEsmall}{F7E6}
\pdfglyphtounicode{Aacute}{00C1}
\pdfglyphtounicode{Aacutesmall}{F7E1}
\pdfglyphtounicode{Abreve}{0102}
\pdfglyphtounicode{Abreveacute}{1EAE}
\pdfglyphtounicode{Abrevecyrillic}{04D0}
\pdfglyphtounicode{Abrevedotbelow}{1EB6}
\pdfglyphtounicode{Abrevegrave}{1EB0}
\pdfglyphtounicode{Abrevehookabove}{1EB2}
\pdfglyphtounicode{Abrevetilde}{1EB4}
\pdfglyphtounicode{Acaron}{01CD}
\pdfglyphtounicode{Acircle}{24B6}
\pdfglyphtounicode{Acircumflex}{00C2}
\pdfglyphtounicode{Acircumflexacute}{1EA4}
\pdfglyphtounicode{Acircumflexdotbelow}{1EAC}
\pdfglyphtounicode{Acircumflexgrave}{1EA6}
\pdfglyphtounicode{Acircumflexhookabove}{1EA8}
\pdfglyphtounicode{Acircumflexsmall}{F7E2}
\pdfglyphtounicode{Acircumflextilde}{1EAA}
\pdfglyphtounicode{Acute}{F6C9}
\pdfglyphtounicode{Acutesmall}{F7B4}
\pdfglyphtounicode{Acyrillic}{0410}
\pdfglyphtounicode{Adblgrave}{0200}
\pdfglyphtounicode{Adieresis}{00C4}
\pdfglyphtounicode{Adieresiscyrillic}{04D2}
\pdfglyphtounicode{Adieresismacron}{01DE}
\pdfglyphtounicode{Adieresissmall}{F7E4}
\pdfglyphtounicode{Adotbelow}{1EA0}
\pdfglyphtounicode{Adotmacron}{01E0}
\pdfglyphtounicode{Agrave}{00C0}
\pdfglyphtounicode{Agravesmall}{F7E0}
\pdfglyphtounicode{Ahookabove}{1EA2}
\pdfglyphtounicode{Aiecyrillic}{04D4}
\pdfglyphtounicode{Ainvertedbreve}{0202}
\pdfglyphtounicode{Alpha}{0391}
\pdfglyphtounicode{Alphatonos}{0386}
\pdfglyphtounicode{Amacron}{0100}
\pdfglyphtounicode{Amonospace}{FF21}
\pdfglyphtounicode{Aogonek}{0104}
\pdfglyphtounicode{Aring}{00C5}
\pdfglyphtounicode{Aringacute}{01FA}
\pdfglyphtounicode{Aringbelow}{1E00}
\pdfglyphtounicode{Aringsmall}{F7E5}
\pdfglyphtounicode{Asmall}{F761}
\pdfglyphtounicode{Atilde}{00C3}
\pdfglyphtounicode{Atildesmall}{F7E3}
\pdfglyphtounicode{Aybarmenian}{0531}
\pdfglyphtounicode{B}{0042}
\pdfglyphtounicode{Bcircle}{24B7}
\pdfglyphtounicode{Bdotaccent}{1E02}
\pdfglyphtounicode{Bdotbelow}{1E04}
\pdfglyphtounicode{Becyrillic}{0411}
\pdfglyphtounicode{Benarmenian}{0532}
\pdfglyphtounicode{Beta}{0392}
\pdfglyphtounicode{Bhook}{0181}
\pdfglyphtounicode{Blinebelow}{1E06}
\pdfglyphtounicode{Bmonospace}{FF22}
\pdfglyphtounicode{Brevesmall}{F6F4}
\pdfglyphtounicode{Bsmall}{F762}
\pdfglyphtounicode{Btopbar}{0182}
\pdfglyphtounicode{C}{0043}
\pdfglyphtounicode{Caarmenian}{053E}
\pdfglyphtounicode{Cacute}{0106}
\pdfglyphtounicode{Caron}{F6CA}
\pdfglyphtounicode{Caronsmall}{F6F5}
\pdfglyphtounicode{Ccaron}{010C}
\pdfglyphtounicode{Ccedilla}{00C7}
\pdfglyphtounicode{Ccedillaacute}{1E08}
\pdfglyphtounicode{Ccedillasmall}{F7E7}
\pdfglyphtounicode{Ccircle}{24B8}
\pdfglyphtounicode{Ccircumflex}{0108}
\pdfglyphtounicode{Cdot}{010A}
\pdfglyphtounicode{Cdotaccent}{010A}
\pdfglyphtounicode{Cedillasmall}{F7B8}
\pdfglyphtounicode{Chaarmenian}{0549}
\pdfglyphtounicode{Cheabkhasiancyrillic}{04BC}
\pdfglyphtounicode{Checyrillic}{0427}
\pdfglyphtounicode{Chedescenderabkhasiancyrillic}{04BE}
\pdfglyphtounicode{Chedescendercyrillic}{04B6}
\pdfglyphtounicode{Chedieresiscyrillic}{04F4}
\pdfglyphtounicode{Cheharmenian}{0543}
\pdfglyphtounicode{Chekhakassiancyrillic}{04CB}
\pdfglyphtounicode{Cheverticalstrokecyrillic}{04B8}
\pdfglyphtounicode{Chi}{03A7}
\pdfglyphtounicode{Chook}{0187}
\pdfglyphtounicode{Circumflexsmall}{F6F6}
\pdfglyphtounicode{Cmonospace}{FF23}
\pdfglyphtounicode{Coarmenian}{0551}
\pdfglyphtounicode{Csmall}{F763}
\pdfglyphtounicode{D}{0044}
\pdfglyphtounicode{DZ}{01F1}
\pdfglyphtounicode{DZcaron}{01C4}
\pdfglyphtounicode{Daarmenian}{0534}
\pdfglyphtounicode{Dafrican}{0189}
\pdfglyphtounicode{Dcaron}{010E}
\pdfglyphtounicode{Dcedilla}{1E10}
\pdfglyphtounicode{Dcircle}{24B9}
\pdfglyphtounicode{Dcircumflexbelow}{1E12}
\pdfglyphtounicode{Dcroat}{0110}
\pdfglyphtounicode{Ddotaccent}{1E0A}
\pdfglyphtounicode{Ddotbelow}{1E0C}
\pdfglyphtounicode{Decyrillic}{0414}
\pdfglyphtounicode{Deicoptic}{03EE}
\pdfglyphtounicode{Delta}{2206}
\pdfglyphtounicode{Deltagreek}{0394}
\pdfglyphtounicode{Dhook}{018A}
\pdfglyphtounicode{Dieresis}{F6CB}
\pdfglyphtounicode{DieresisAcute}{F6CC}
\pdfglyphtounicode{DieresisGrave}{F6CD}
\pdfglyphtounicode{Dieresissmall}{F7A8}
\pdfglyphtounicode{Digammagreek}{03DC}
\pdfglyphtounicode{Djecyrillic}{0402}
\pdfglyphtounicode{Dlinebelow}{1E0E}
\pdfglyphtounicode{Dmonospace}{FF24}
\pdfglyphtounicode{Dotaccentsmall}{F6F7}
\pdfglyphtounicode{Dslash}{0110}
\pdfglyphtounicode{Dsmall}{F764}
\pdfglyphtounicode{Dtopbar}{018B}
\pdfglyphtounicode{Dz}{01F2}
\pdfglyphtounicode{Dzcaron}{01C5}
\pdfglyphtounicode{Dzeabkhasiancyrillic}{04E0}
\pdfglyphtounicode{Dzecyrillic}{0405}
\pdfglyphtounicode{Dzhecyrillic}{040F}
\pdfglyphtounicode{E}{0045}
\pdfglyphtounicode{Eacute}{00C9}
\pdfglyphtounicode{Eacutesmall}{F7E9}
\pdfglyphtounicode{Ebreve}{0114}
\pdfglyphtounicode{Ecaron}{011A}
\pdfglyphtounicode{Ecedillabreve}{1E1C}
\pdfglyphtounicode{Echarmenian}{0535}
\pdfglyphtounicode{Ecircle}{24BA}
\pdfglyphtounicode{Ecircumflex}{00CA}
\pdfglyphtounicode{Ecircumflexacute}{1EBE}
\pdfglyphtounicode{Ecircumflexbelow}{1E18}
\pdfglyphtounicode{Ecircumflexdotbelow}{1EC6}
\pdfglyphtounicode{Ecircumflexgrave}{1EC0}
\pdfglyphtounicode{Ecircumflexhookabove}{1EC2}
\pdfglyphtounicode{Ecircumflexsmall}{F7EA}
\pdfglyphtounicode{Ecircumflextilde}{1EC4}
\pdfglyphtounicode{Ecyrillic}{0404}
\pdfglyphtounicode{Edblgrave}{0204}
\pdfglyphtounicode{Edieresis}{00CB}
\pdfglyphtounicode{Edieresissmall}{F7EB}
\pdfglyphtounicode{Edot}{0116}
\pdfglyphtounicode{Edotaccent}{0116}
\pdfglyphtounicode{Edotbelow}{1EB8}
\pdfglyphtounicode{Efcyrillic}{0424}
\pdfglyphtounicode{Egrave}{00C8}
\pdfglyphtounicode{Egravesmall}{F7E8}
\pdfglyphtounicode{Eharmenian}{0537}
\pdfglyphtounicode{Ehookabove}{1EBA}
\pdfglyphtounicode{Eightroman}{2167}
\pdfglyphtounicode{Einvertedbreve}{0206}
\pdfglyphtounicode{Eiotifiedcyrillic}{0464}
\pdfglyphtounicode{Elcyrillic}{041B}
\pdfglyphtounicode{Elevenroman}{216A}
\pdfglyphtounicode{Emacron}{0112}
\pdfglyphtounicode{Emacronacute}{1E16}
\pdfglyphtounicode{Emacrongrave}{1E14}
\pdfglyphtounicode{Emcyrillic}{041C}
\pdfglyphtounicode{Emonospace}{FF25}
\pdfglyphtounicode{Encyrillic}{041D}
\pdfglyphtounicode{Endescendercyrillic}{04A2}
\pdfglyphtounicode{Eng}{014A}
\pdfglyphtounicode{Enghecyrillic}{04A4}
\pdfglyphtounicode{Enhookcyrillic}{04C7}
\pdfglyphtounicode{Eogonek}{0118}
\pdfglyphtounicode{Eopen}{0190}
\pdfglyphtounicode{Epsilon}{0395}
\pdfglyphtounicode{Epsilontonos}{0388}
\pdfglyphtounicode{Ercyrillic}{0420}
\pdfglyphtounicode{Ereversed}{018E}
\pdfglyphtounicode{Ereversedcyrillic}{042D}
\pdfglyphtounicode{Escyrillic}{0421}
\pdfglyphtounicode{Esdescendercyrillic}{04AA}
\pdfglyphtounicode{Esh}{01A9}
\pdfglyphtounicode{Esmall}{F765}
\pdfglyphtounicode{Eta}{0397}
\pdfglyphtounicode{Etarmenian}{0538}
\pdfglyphtounicode{Etatonos}{0389}
\pdfglyphtounicode{Eth}{00D0}
\pdfglyphtounicode{Ethsmall}{F7F0}
\pdfglyphtounicode{Etilde}{1EBC}
\pdfglyphtounicode{Etildebelow}{1E1A}
\pdfglyphtounicode{Euro}{20AC}
\pdfglyphtounicode{Ezh}{01B7}
\pdfglyphtounicode{Ezhcaron}{01EE}
\pdfglyphtounicode{Ezhreversed}{01B8}
\pdfglyphtounicode{F}{0046}
\pdfglyphtounicode{Fcircle}{24BB}
\pdfglyphtounicode{Fdotaccent}{1E1E}
\pdfglyphtounicode{Feharmenian}{0556}
\pdfglyphtounicode{Feicoptic}{03E4}
\pdfglyphtounicode{Fhook}{0191}
\pdfglyphtounicode{Fitacyrillic}{0472}
\pdfglyphtounicode{Fiveroman}{2164}
\pdfglyphtounicode{Fmonospace}{FF26}
\pdfglyphtounicode{Fourroman}{2163}
\pdfglyphtounicode{Fsmall}{F766}
\pdfglyphtounicode{G}{0047}
\pdfglyphtounicode{GBsquare}{3387}
\pdfglyphtounicode{Gacute}{01F4}
\pdfglyphtounicode{Gamma}{0393}
\pdfglyphtounicode{Gammaafrican}{0194}
\pdfglyphtounicode{Gangiacoptic}{03EA}
\pdfglyphtounicode{Gbreve}{011E}
\pdfglyphtounicode{Gcaron}{01E6}
\pdfglyphtounicode{Gcedilla}{0122}
\pdfglyphtounicode{Gcircle}{24BC}
\pdfglyphtounicode{Gcircumflex}{011C}
\pdfglyphtounicode{Gcommaaccent}{0122}
\pdfglyphtounicode{Gdot}{0120}
\pdfglyphtounicode{Gdotaccent}{0120}
\pdfglyphtounicode{Gecyrillic}{0413}
\pdfglyphtounicode{Ghadarmenian}{0542}
\pdfglyphtounicode{Ghemiddlehookcyrillic}{0494}
\pdfglyphtounicode{Ghestrokecyrillic}{0492}
\pdfglyphtounicode{Gheupturncyrillic}{0490}
\pdfglyphtounicode{Ghook}{0193}
\pdfglyphtounicode{Gimarmenian}{0533}
\pdfglyphtounicode{Gjecyrillic}{0403}
\pdfglyphtounicode{Gmacron}{1E20}
\pdfglyphtounicode{Gmonospace}{FF27}
\pdfglyphtounicode{Grave}{F6CE}
\pdfglyphtounicode{Gravesmall}{F760}
\pdfglyphtounicode{Gsmall}{F767}
\pdfglyphtounicode{Gsmallhook}{029B}
\pdfglyphtounicode{Gstroke}{01E4}
\pdfglyphtounicode{H}{0048}
\pdfglyphtounicode{H18533}{25CF}
\pdfglyphtounicode{H18543}{25AA}
\pdfglyphtounicode{H18551}{25AB}
\pdfglyphtounicode{H22073}{25A1}
\pdfglyphtounicode{HPsquare}{33CB}
\pdfglyphtounicode{Haabkhasiancyrillic}{04A8}
\pdfglyphtounicode{Hadescendercyrillic}{04B2}
\pdfglyphtounicode{Hardsigncyrillic}{042A}
\pdfglyphtounicode{Hbar}{0126}
\pdfglyphtounicode{Hbrevebelow}{1E2A}
\pdfglyphtounicode{Hcedilla}{1E28}
\pdfglyphtounicode{Hcircle}{24BD}
\pdfglyphtounicode{Hcircumflex}{0124}
\pdfglyphtounicode{Hdieresis}{1E26}
\pdfglyphtounicode{Hdotaccent}{1E22}
\pdfglyphtounicode{Hdotbelow}{1E24}
\pdfglyphtounicode{Hmonospace}{FF28}
\pdfglyphtounicode{Hoarmenian}{0540}
\pdfglyphtounicode{Horicoptic}{03E8}
\pdfglyphtounicode{Hsmall}{F768}
\pdfglyphtounicode{Hungarumlaut}{F6CF}
\pdfglyphtounicode{Hungarumlautsmall}{F6F8}
\pdfglyphtounicode{Hzsquare}{3390}
\pdfglyphtounicode{I}{0049}
\pdfglyphtounicode{IAcyrillic}{042F}
\pdfglyphtounicode{IJ}{0132}
\pdfglyphtounicode{IUcyrillic}{042E}
\pdfglyphtounicode{Iacute}{00CD}
\pdfglyphtounicode{Iacutesmall}{F7ED}
\pdfglyphtounicode{Ibreve}{012C}
\pdfglyphtounicode{Icaron}{01CF}
\pdfglyphtounicode{Icircle}{24BE}
\pdfglyphtounicode{Icircumflex}{00CE}
\pdfglyphtounicode{Icircumflexsmall}{F7EE}
\pdfglyphtounicode{Icyrillic}{0406}
\pdfglyphtounicode{Idblgrave}{0208}
\pdfglyphtounicode{Idieresis}{00CF}
\pdfglyphtounicode{Idieresisacute}{1E2E}
\pdfglyphtounicode{Idieresiscyrillic}{04E4}
\pdfglyphtounicode{Idieresissmall}{F7EF}
\pdfglyphtounicode{Idot}{0130}
\pdfglyphtounicode{Idotaccent}{0130}
\pdfglyphtounicode{Idotbelow}{1ECA}
\pdfglyphtounicode{Iebrevecyrillic}{04D6}
\pdfglyphtounicode{Iecyrillic}{0415}
\pdfglyphtounicode{Ifraktur}{2111}
\pdfglyphtounicode{Igrave}{00CC}
\pdfglyphtounicode{Igravesmall}{F7EC}
\pdfglyphtounicode{Ihookabove}{1EC8}
\pdfglyphtounicode{Iicyrillic}{0418}
\pdfglyphtounicode{Iinvertedbreve}{020A}
\pdfglyphtounicode{Iishortcyrillic}{0419}
\pdfglyphtounicode{Imacron}{012A}
\pdfglyphtounicode{Imacroncyrillic}{04E2}
\pdfglyphtounicode{Imonospace}{FF29}
\pdfglyphtounicode{Iniarmenian}{053B}
\pdfglyphtounicode{Iocyrillic}{0401}
\pdfglyphtounicode{Iogonek}{012E}
\pdfglyphtounicode{Iota}{0399}
\pdfglyphtounicode{Iotaafrican}{0196}
\pdfglyphtounicode{Iotadieresis}{03AA}
\pdfglyphtounicode{Iotatonos}{038A}
\pdfglyphtounicode{Ismall}{F769}
\pdfglyphtounicode{Istroke}{0197}
\pdfglyphtounicode{Itilde}{0128}
\pdfglyphtounicode{Itildebelow}{1E2C}
\pdfglyphtounicode{Izhitsacyrillic}{0474}
\pdfglyphtounicode{Izhitsadblgravecyrillic}{0476}
\pdfglyphtounicode{J}{004A}
\pdfglyphtounicode{Jaarmenian}{0541}
\pdfglyphtounicode{Jcircle}{24BF}
\pdfglyphtounicode{Jcircumflex}{0134}
\pdfglyphtounicode{Jecyrillic}{0408}
\pdfglyphtounicode{Jheharmenian}{054B}
\pdfglyphtounicode{Jmonospace}{FF2A}
\pdfglyphtounicode{Jsmall}{F76A}
\pdfglyphtounicode{K}{004B}
\pdfglyphtounicode{KBsquare}{3385}
\pdfglyphtounicode{KKsquare}{33CD}
\pdfglyphtounicode{Kabashkircyrillic}{04A0}
\pdfglyphtounicode{Kacute}{1E30}
\pdfglyphtounicode{Kacyrillic}{041A}
\pdfglyphtounicode{Kadescendercyrillic}{049A}
\pdfglyphtounicode{Kahookcyrillic}{04C3}
\pdfglyphtounicode{Kappa}{039A}
\pdfglyphtounicode{Kastrokecyrillic}{049E}
\pdfglyphtounicode{Kaverticalstrokecyrillic}{049C}
\pdfglyphtounicode{Kcaron}{01E8}
\pdfglyphtounicode{Kcedilla}{0136}
\pdfglyphtounicode{Kcircle}{24C0}
\pdfglyphtounicode{Kcommaaccent}{0136}
\pdfglyphtounicode{Kdotbelow}{1E32}
\pdfglyphtounicode{Keharmenian}{0554}
\pdfglyphtounicode{Kenarmenian}{053F}
\pdfglyphtounicode{Khacyrillic}{0425}
\pdfglyphtounicode{Kheicoptic}{03E6}
\pdfglyphtounicode{Khook}{0198}
\pdfglyphtounicode{Kjecyrillic}{040C}
\pdfglyphtounicode{Klinebelow}{1E34}
\pdfglyphtounicode{Kmonospace}{FF2B}
\pdfglyphtounicode{Koppacyrillic}{0480}
\pdfglyphtounicode{Koppagreek}{03DE}
\pdfglyphtounicode{Ksicyrillic}{046E}
\pdfglyphtounicode{Ksmall}{F76B}
\pdfglyphtounicode{L}{004C}
\pdfglyphtounicode{LJ}{01C7}
\pdfglyphtounicode{LL}{F6BF}
\pdfglyphtounicode{Lacute}{0139}
\pdfglyphtounicode{Lambda}{039B}
\pdfglyphtounicode{Lcaron}{013D}
\pdfglyphtounicode{Lcedilla}{013B}
\pdfglyphtounicode{Lcircle}{24C1}
\pdfglyphtounicode{Lcircumflexbelow}{1E3C}
\pdfglyphtounicode{Lcommaaccent}{013B}
\pdfglyphtounicode{Ldot}{013F}
\pdfglyphtounicode{Ldotaccent}{013F}
\pdfglyphtounicode{Ldotbelow}{1E36}
\pdfglyphtounicode{Ldotbelowmacron}{1E38}
\pdfglyphtounicode{Liwnarmenian}{053C}
\pdfglyphtounicode{Lj}{01C8}
\pdfglyphtounicode{Ljecyrillic}{0409}
\pdfglyphtounicode{Llinebelow}{1E3A}
\pdfglyphtounicode{Lmonospace}{FF2C}
\pdfglyphtounicode{Lslash}{0141}
\pdfglyphtounicode{Lslashsmall}{F6F9}
\pdfglyphtounicode{Lsmall}{F76C}
\pdfglyphtounicode{M}{004D}
\pdfglyphtounicode{MBsquare}{3386}
\pdfglyphtounicode{Macron}{F6D0}
\pdfglyphtounicode{Macronsmall}{F7AF}
\pdfglyphtounicode{Macute}{1E3E}
\pdfglyphtounicode{Mcircle}{24C2}
\pdfglyphtounicode{Mdotaccent}{1E40}
\pdfglyphtounicode{Mdotbelow}{1E42}
\pdfglyphtounicode{Menarmenian}{0544}
\pdfglyphtounicode{Mmonospace}{FF2D}
\pdfglyphtounicode{Msmall}{F76D}
\pdfglyphtounicode{Mturned}{019C}
\pdfglyphtounicode{Mu}{039C}
\pdfglyphtounicode{N}{004E}
\pdfglyphtounicode{NJ}{01CA}
\pdfglyphtounicode{Nacute}{0143}
\pdfglyphtounicode{Ncaron}{0147}
\pdfglyphtounicode{Ncedilla}{0145}
\pdfglyphtounicode{Ncircle}{24C3}
\pdfglyphtounicode{Ncircumflexbelow}{1E4A}
\pdfglyphtounicode{Ncommaaccent}{0145}
\pdfglyphtounicode{Ndotaccent}{1E44}
\pdfglyphtounicode{Ndotbelow}{1E46}
\pdfglyphtounicode{Nhookleft}{019D}
\pdfglyphtounicode{Nineroman}{2168}
\pdfglyphtounicode{Nj}{01CB}
\pdfglyphtounicode{Njecyrillic}{040A}
\pdfglyphtounicode{Nlinebelow}{1E48}
\pdfglyphtounicode{Nmonospace}{FF2E}
\pdfglyphtounicode{Nowarmenian}{0546}
\pdfglyphtounicode{Nsmall}{F76E}
\pdfglyphtounicode{Ntilde}{00D1}
\pdfglyphtounicode{Ntildesmall}{F7F1}
\pdfglyphtounicode{Nu}{039D}
\pdfglyphtounicode{O}{004F}
\pdfglyphtounicode{OE}{0152}
\pdfglyphtounicode{OEsmall}{F6FA}
\pdfglyphtounicode{Oacute}{00D3}
\pdfglyphtounicode{Oacutesmall}{F7F3}
\pdfglyphtounicode{Obarredcyrillic}{04E8}
\pdfglyphtounicode{Obarreddieresiscyrillic}{04EA}
\pdfglyphtounicode{Obreve}{014E}
\pdfglyphtounicode{Ocaron}{01D1}
\pdfglyphtounicode{Ocenteredtilde}{019F}
\pdfglyphtounicode{Ocircle}{24C4}
\pdfglyphtounicode{Ocircumflex}{00D4}
\pdfglyphtounicode{Ocircumflexacute}{1ED0}
\pdfglyphtounicode{Ocircumflexdotbelow}{1ED8}
\pdfglyphtounicode{Ocircumflexgrave}{1ED2}
\pdfglyphtounicode{Ocircumflexhookabove}{1ED4}
\pdfglyphtounicode{Ocircumflexsmall}{F7F4}
\pdfglyphtounicode{Ocircumflextilde}{1ED6}
\pdfglyphtounicode{Ocyrillic}{041E}
\pdfglyphtounicode{Odblacute}{0150}
\pdfglyphtounicode{Odblgrave}{020C}
\pdfglyphtounicode{Odieresis}{00D6}
\pdfglyphtounicode{Odieresiscyrillic}{04E6}
\pdfglyphtounicode{Odieresissmall}{F7F6}
\pdfglyphtounicode{Odotbelow}{1ECC}
\pdfglyphtounicode{Ogoneksmall}{F6FB}
\pdfglyphtounicode{Ograve}{00D2}
\pdfglyphtounicode{Ogravesmall}{F7F2}
\pdfglyphtounicode{Oharmenian}{0555}
\pdfglyphtounicode{Ohm}{2126}
\pdfglyphtounicode{Ohookabove}{1ECE}
\pdfglyphtounicode{Ohorn}{01A0}
\pdfglyphtounicode{Ohornacute}{1EDA}
\pdfglyphtounicode{Ohorndotbelow}{1EE2}
\pdfglyphtounicode{Ohorngrave}{1EDC}
\pdfglyphtounicode{Ohornhookabove}{1EDE}
\pdfglyphtounicode{Ohorntilde}{1EE0}
\pdfglyphtounicode{Ohungarumlaut}{0150}
\pdfglyphtounicode{Oi}{01A2}
\pdfglyphtounicode{Oinvertedbreve}{020E}
\pdfglyphtounicode{Omacron}{014C}
\pdfglyphtounicode{Omacronacute}{1E52}
\pdfglyphtounicode{Omacrongrave}{1E50}
\pdfglyphtounicode{Omega}{2126}
\pdfglyphtounicode{Omegacyrillic}{0460}
\pdfglyphtounicode{Omegagreek}{03A9}
\pdfglyphtounicode{Omegaroundcyrillic}{047A}
\pdfglyphtounicode{Omegatitlocyrillic}{047C}
\pdfglyphtounicode{Omegatonos}{038F}
\pdfglyphtounicode{Omicron}{039F}
\pdfglyphtounicode{Omicrontonos}{038C}
\pdfglyphtounicode{Omonospace}{FF2F}
\pdfglyphtounicode{Oneroman}{2160}
\pdfglyphtounicode{Oogonek}{01EA}
\pdfglyphtounicode{Oogonekmacron}{01EC}
\pdfglyphtounicode{Oopen}{0186}
\pdfglyphtounicode{Oslash}{00D8}
\pdfglyphtounicode{Oslashacute}{01FE}
\pdfglyphtounicode{Oslashsmall}{F7F8}
\pdfglyphtounicode{Osmall}{F76F}
\pdfglyphtounicode{Ostrokeacute}{01FE}
\pdfglyphtounicode{Otcyrillic}{047E}
\pdfglyphtounicode{Otilde}{00D5}
\pdfglyphtounicode{Otildeacute}{1E4C}
\pdfglyphtounicode{Otildedieresis}{1E4E}
\pdfglyphtounicode{Otildesmall}{F7F5}
\pdfglyphtounicode{P}{0050}
\pdfglyphtounicode{Pacute}{1E54}
\pdfglyphtounicode{Pcircle}{24C5}
\pdfglyphtounicode{Pdotaccent}{1E56}
\pdfglyphtounicode{Pecyrillic}{041F}
\pdfglyphtounicode{Peharmenian}{054A}
\pdfglyphtounicode{Pemiddlehookcyrillic}{04A6}
\pdfglyphtounicode{Phi}{03A6}
\pdfglyphtounicode{Phook}{01A4}
\pdfglyphtounicode{Pi}{03A0}
\pdfglyphtounicode{Piwrarmenian}{0553}
\pdfglyphtounicode{Pmonospace}{FF30}
\pdfglyphtounicode{Psi}{03A8}
\pdfglyphtounicode{Psicyrillic}{0470}
\pdfglyphtounicode{Psmall}{F770}
\pdfglyphtounicode{Q}{0051}
\pdfglyphtounicode{Qcircle}{24C6}
\pdfglyphtounicode{Qmonospace}{FF31}
\pdfglyphtounicode{Qsmall}{F771}
\pdfglyphtounicode{R}{0052}
\pdfglyphtounicode{Raarmenian}{054C}
\pdfglyphtounicode{Racute}{0154}
\pdfglyphtounicode{Rcaron}{0158}
\pdfglyphtounicode{Rcedilla}{0156}
\pdfglyphtounicode{Rcircle}{24C7}
\pdfglyphtounicode{Rcommaaccent}{0156}
\pdfglyphtounicode{Rdblgrave}{0210}
\pdfglyphtounicode{Rdotaccent}{1E58}
\pdfglyphtounicode{Rdotbelow}{1E5A}
\pdfglyphtounicode{Rdotbelowmacron}{1E5C}
\pdfglyphtounicode{Reharmenian}{0550}
\pdfglyphtounicode{Rfraktur}{211C}
\pdfglyphtounicode{Rho}{03A1}
\pdfglyphtounicode{Ringsmall}{F6FC}
\pdfglyphtounicode{Rinvertedbreve}{0212}
\pdfglyphtounicode{Rlinebelow}{1E5E}
\pdfglyphtounicode{Rmonospace}{FF32}
\pdfglyphtounicode{Rsmall}{F772}
\pdfglyphtounicode{Rsmallinverted}{0281}
\pdfglyphtounicode{Rsmallinvertedsuperior}{02B6}
\pdfglyphtounicode{S}{0053}
\pdfglyphtounicode{SF010000}{250C}
\pdfglyphtounicode{SF020000}{2514}
\pdfglyphtounicode{SF030000}{2510}
\pdfglyphtounicode{SF040000}{2518}
\pdfglyphtounicode{SF050000}{253C}
\pdfglyphtounicode{SF060000}{252C}
\pdfglyphtounicode{SF070000}{2534}
\pdfglyphtounicode{SF080000}{251C}
\pdfglyphtounicode{SF090000}{2524}
\pdfglyphtounicode{SF100000}{2500}
\pdfglyphtounicode{SF110000}{2502}
\pdfglyphtounicode{SF190000}{2561}
\pdfglyphtounicode{SF200000}{2562}
\pdfglyphtounicode{SF210000}{2556}
\pdfglyphtounicode{SF220000}{2555}
\pdfglyphtounicode{SF230000}{2563}
\pdfglyphtounicode{SF240000}{2551}
\pdfglyphtounicode{SF250000}{2557}
\pdfglyphtounicode{SF260000}{255D}
\pdfglyphtounicode{SF270000}{255C}
\pdfglyphtounicode{SF280000}{255B}
\pdfglyphtounicode{SF360000}{255E}
\pdfglyphtounicode{SF370000}{255F}
\pdfglyphtounicode{SF380000}{255A}
\pdfglyphtounicode{SF390000}{2554}
\pdfglyphtounicode{SF400000}{2569}
\pdfglyphtounicode{SF410000}{2566}
\pdfglyphtounicode{SF420000}{2560}
\pdfglyphtounicode{SF430000}{2550}
\pdfglyphtounicode{SF440000}{256C}
\pdfglyphtounicode{SF450000}{2567}
\pdfglyphtounicode{SF460000}{2568}
\pdfglyphtounicode{SF470000}{2564}
\pdfglyphtounicode{SF480000}{2565}
\pdfglyphtounicode{SF490000}{2559}
\pdfglyphtounicode{SF500000}{2558}
\pdfglyphtounicode{SF510000}{2552}
\pdfglyphtounicode{SF520000}{2553}
\pdfglyphtounicode{SF530000}{256B}
\pdfglyphtounicode{SF540000}{256A}
\pdfglyphtounicode{Sacute}{015A}
\pdfglyphtounicode{Sacutedotaccent}{1E64}
\pdfglyphtounicode{Sampigreek}{03E0}
\pdfglyphtounicode{Scaron}{0160}
\pdfglyphtounicode{Scarondotaccent}{1E66}
\pdfglyphtounicode{Scaronsmall}{F6FD}
\pdfglyphtounicode{Scedilla}{015E}
\pdfglyphtounicode{Schwa}{018F}
\pdfglyphtounicode{Schwacyrillic}{04D8}
\pdfglyphtounicode{Schwadieresiscyrillic}{04DA}
\pdfglyphtounicode{Scircle}{24C8}
\pdfglyphtounicode{Scircumflex}{015C}
\pdfglyphtounicode{Scommaaccent}{0218}
\pdfglyphtounicode{Sdotaccent}{1E60}
\pdfglyphtounicode{Sdotbelow}{1E62}
\pdfglyphtounicode{Sdotbelowdotaccent}{1E68}
\pdfglyphtounicode{Seharmenian}{054D}
\pdfglyphtounicode{Sevenroman}{2166}
\pdfglyphtounicode{Shaarmenian}{0547}
\pdfglyphtounicode{Shacyrillic}{0428}
\pdfglyphtounicode{Shchacyrillic}{0429}
\pdfglyphtounicode{Sheicoptic}{03E2}
\pdfglyphtounicode{Shhacyrillic}{04BA}
\pdfglyphtounicode{Shimacoptic}{03EC}
\pdfglyphtounicode{Sigma}{03A3}
\pdfglyphtounicode{Sixroman}{2165}
\pdfglyphtounicode{Smonospace}{FF33}
\pdfglyphtounicode{Softsigncyrillic}{042C}
\pdfglyphtounicode{Ssmall}{F773}
\pdfglyphtounicode{Stigmagreek}{03DA}
\pdfglyphtounicode{T}{0054}
\pdfglyphtounicode{Tau}{03A4}
\pdfglyphtounicode{Tbar}{0166}
\pdfglyphtounicode{Tcaron}{0164}
\pdfglyphtounicode{Tcedilla}{0162}
\pdfglyphtounicode{Tcircle}{24C9}
\pdfglyphtounicode{Tcircumflexbelow}{1E70}
\pdfglyphtounicode{Tcommaaccent}{0162}
\pdfglyphtounicode{Tdotaccent}{1E6A}
\pdfglyphtounicode{Tdotbelow}{1E6C}
\pdfglyphtounicode{Tecyrillic}{0422}
\pdfglyphtounicode{Tedescendercyrillic}{04AC}
\pdfglyphtounicode{Tenroman}{2169}
\pdfglyphtounicode{Tetsecyrillic}{04B4}
\pdfglyphtounicode{Theta}{0398}
\pdfglyphtounicode{Thook}{01AC}
\pdfglyphtounicode{Thorn}{00DE}
\pdfglyphtounicode{Thornsmall}{F7FE}
\pdfglyphtounicode{Threeroman}{2162}
\pdfglyphtounicode{Tildesmall}{F6FE}
\pdfglyphtounicode{Tiwnarmenian}{054F}
\pdfglyphtounicode{Tlinebelow}{1E6E}
\pdfglyphtounicode{Tmonospace}{FF34}
\pdfglyphtounicode{Toarmenian}{0539}
\pdfglyphtounicode{Tonefive}{01BC}
\pdfglyphtounicode{Tonesix}{0184}
\pdfglyphtounicode{Tonetwo}{01A7}
\pdfglyphtounicode{Tretroflexhook}{01AE}
\pdfglyphtounicode{Tsecyrillic}{0426}
\pdfglyphtounicode{Tshecyrillic}{040B}
\pdfglyphtounicode{Tsmall}{F774}
\pdfglyphtounicode{Twelveroman}{216B}
\pdfglyphtounicode{Tworoman}{2161}
\pdfglyphtounicode{U}{0055}
\pdfglyphtounicode{Uacute}{00DA}
\pdfglyphtounicode{Uacutesmall}{F7FA}
\pdfglyphtounicode{Ubreve}{016C}
\pdfglyphtounicode{Ucaron}{01D3}
\pdfglyphtounicode{Ucircle}{24CA}
\pdfglyphtounicode{Ucircumflex}{00DB}
\pdfglyphtounicode{Ucircumflexbelow}{1E76}
\pdfglyphtounicode{Ucircumflexsmall}{F7FB}
\pdfglyphtounicode{Ucyrillic}{0423}
\pdfglyphtounicode{Udblacute}{0170}
\pdfglyphtounicode{Udblgrave}{0214}
\pdfglyphtounicode{Udieresis}{00DC}
\pdfglyphtounicode{Udieresisacute}{01D7}
\pdfglyphtounicode{Udieresisbelow}{1E72}
\pdfglyphtounicode{Udieresiscaron}{01D9}
\pdfglyphtounicode{Udieresiscyrillic}{04F0}
\pdfglyphtounicode{Udieresisgrave}{01DB}
\pdfglyphtounicode{Udieresismacron}{01D5}
\pdfglyphtounicode{Udieresissmall}{F7FC}
\pdfglyphtounicode{Udotbelow}{1EE4}
\pdfglyphtounicode{Ugrave}{00D9}
\pdfglyphtounicode{Ugravesmall}{F7F9}
\pdfglyphtounicode{Uhookabove}{1EE6}
\pdfglyphtounicode{Uhorn}{01AF}
\pdfglyphtounicode{Uhornacute}{1EE8}
\pdfglyphtounicode{Uhorndotbelow}{1EF0}
\pdfglyphtounicode{Uhorngrave}{1EEA}
\pdfglyphtounicode{Uhornhookabove}{1EEC}
\pdfglyphtounicode{Uhorntilde}{1EEE}
\pdfglyphtounicode{Uhungarumlaut}{0170}
\pdfglyphtounicode{Uhungarumlautcyrillic}{04F2}
\pdfglyphtounicode{Uinvertedbreve}{0216}
\pdfglyphtounicode{Ukcyrillic}{0478}
\pdfglyphtounicode{Umacron}{016A}
\pdfglyphtounicode{Umacroncyrillic}{04EE}
\pdfglyphtounicode{Umacrondieresis}{1E7A}
\pdfglyphtounicode{Umonospace}{FF35}
\pdfglyphtounicode{Uogonek}{0172}
\pdfglyphtounicode{Upsilon}{03A5}
\pdfglyphtounicode{Upsilon1}{03D2}
\pdfglyphtounicode{Upsilonacutehooksymbolgreek}{03D3}
\pdfglyphtounicode{Upsilonafrican}{01B1}
\pdfglyphtounicode{Upsilondieresis}{03AB}
\pdfglyphtounicode{Upsilondieresishooksymbolgreek}{03D4}
\pdfglyphtounicode{Upsilonhooksymbol}{03D2}
\pdfglyphtounicode{Upsilontonos}{038E}
\pdfglyphtounicode{Uring}{016E}
\pdfglyphtounicode{Ushortcyrillic}{040E}
\pdfglyphtounicode{Usmall}{F775}
\pdfglyphtounicode{Ustraightcyrillic}{04AE}
\pdfglyphtounicode{Ustraightstrokecyrillic}{04B0}
\pdfglyphtounicode{Utilde}{0168}
\pdfglyphtounicode{Utildeacute}{1E78}
\pdfglyphtounicode{Utildebelow}{1E74}
\pdfglyphtounicode{V}{0056}
\pdfglyphtounicode{Vcircle}{24CB}
\pdfglyphtounicode{Vdotbelow}{1E7E}
\pdfglyphtounicode{Vecyrillic}{0412}
\pdfglyphtounicode{Vewarmenian}{054E}
\pdfglyphtounicode{Vhook}{01B2}
\pdfglyphtounicode{Vmonospace}{FF36}
\pdfglyphtounicode{Voarmenian}{0548}
\pdfglyphtounicode{Vsmall}{F776}
\pdfglyphtounicode{Vtilde}{1E7C}
\pdfglyphtounicode{W}{0057}
\pdfglyphtounicode{Wacute}{1E82}
\pdfglyphtounicode{Wcircle}{24CC}
\pdfglyphtounicode{Wcircumflex}{0174}
\pdfglyphtounicode{Wdieresis}{1E84}
\pdfglyphtounicode{Wdotaccent}{1E86}
\pdfglyphtounicode{Wdotbelow}{1E88}
\pdfglyphtounicode{Wgrave}{1E80}
\pdfglyphtounicode{Wmonospace}{FF37}
\pdfglyphtounicode{Wsmall}{F777}
\pdfglyphtounicode{X}{0058}
\pdfglyphtounicode{Xcircle}{24CD}
\pdfglyphtounicode{Xdieresis}{1E8C}
\pdfglyphtounicode{Xdotaccent}{1E8A}
\pdfglyphtounicode{Xeharmenian}{053D}
\pdfglyphtounicode{Xi}{039E}
\pdfglyphtounicode{Xmonospace}{FF38}
\pdfglyphtounicode{Xsmall}{F778}
\pdfglyphtounicode{Y}{0059}
\pdfglyphtounicode{Yacute}{00DD}
\pdfglyphtounicode{Yacutesmall}{F7FD}
\pdfglyphtounicode{Yatcyrillic}{0462}
\pdfglyphtounicode{Ycircle}{24CE}
\pdfglyphtounicode{Ycircumflex}{0176}
\pdfglyphtounicode{Ydieresis}{0178}
\pdfglyphtounicode{Ydieresissmall}{F7FF}
\pdfglyphtounicode{Ydotaccent}{1E8E}
\pdfglyphtounicode{Ydotbelow}{1EF4}
\pdfglyphtounicode{Yericyrillic}{042B}
\pdfglyphtounicode{Yerudieresiscyrillic}{04F8}
\pdfglyphtounicode{Ygrave}{1EF2}
\pdfglyphtounicode{Yhook}{01B3}
\pdfglyphtounicode{Yhookabove}{1EF6}
\pdfglyphtounicode{Yiarmenian}{0545}
\pdfglyphtounicode{Yicyrillic}{0407}
\pdfglyphtounicode{Yiwnarmenian}{0552}
\pdfglyphtounicode{Ymonospace}{FF39}
\pdfglyphtounicode{Ysmall}{F779}
\pdfglyphtounicode{Ytilde}{1EF8}
\pdfglyphtounicode{Yusbigcyrillic}{046A}
\pdfglyphtounicode{Yusbigiotifiedcyrillic}{046C}
\pdfglyphtounicode{Yuslittlecyrillic}{0466}
\pdfglyphtounicode{Yuslittleiotifiedcyrillic}{0468}
\pdfglyphtounicode{Z}{005A}
\pdfglyphtounicode{Zaarmenian}{0536}
\pdfglyphtounicode{Zacute}{0179}
\pdfglyphtounicode{Zcaron}{017D}
\pdfglyphtounicode{Zcaronsmall}{F6FF}
\pdfglyphtounicode{Zcircle}{24CF}
\pdfglyphtounicode{Zcircumflex}{1E90}
\pdfglyphtounicode{Zdot}{017B}
\pdfglyphtounicode{Zdotaccent}{017B}
\pdfglyphtounicode{Zdotbelow}{1E92}
\pdfglyphtounicode{Zecyrillic}{0417}
\pdfglyphtounicode{Zedescendercyrillic}{0498}
\pdfglyphtounicode{Zedieresiscyrillic}{04DE}
\pdfglyphtounicode{Zeta}{0396}
\pdfglyphtounicode{Zhearmenian}{053A}
\pdfglyphtounicode{Zhebrevecyrillic}{04C1}
\pdfglyphtounicode{Zhecyrillic}{0416}
\pdfglyphtounicode{Zhedescendercyrillic}{0496}
\pdfglyphtounicode{Zhedieresiscyrillic}{04DC}
\pdfglyphtounicode{Zlinebelow}{1E94}
\pdfglyphtounicode{Zmonospace}{FF3A}
\pdfglyphtounicode{Zsmall}{F77A}
\pdfglyphtounicode{Zstroke}{01B5}
\pdfglyphtounicode{a}{0061}
\pdfglyphtounicode{aabengali}{0986}
\pdfglyphtounicode{aacute}{00E1}
\pdfglyphtounicode{aadeva}{0906}
\pdfglyphtounicode{aagujarati}{0A86}
\pdfglyphtounicode{aagurmukhi}{0A06}
\pdfglyphtounicode{aamatragurmukhi}{0A3E}
\pdfglyphtounicode{aarusquare}{3303}
\pdfglyphtounicode{aavowelsignbengali}{09BE}
\pdfglyphtounicode{aavowelsigndeva}{093E}
\pdfglyphtounicode{aavowelsigngujarati}{0ABE}
\pdfglyphtounicode{abbreviationmarkarmenian}{055F}
\pdfglyphtounicode{abbreviationsigndeva}{0970}
\pdfglyphtounicode{abengali}{0985}
\pdfglyphtounicode{abopomofo}{311A}
\pdfglyphtounicode{abreve}{0103}
\pdfglyphtounicode{abreveacute}{1EAF}
\pdfglyphtounicode{abrevecyrillic}{04D1}
\pdfglyphtounicode{abrevedotbelow}{1EB7}
\pdfglyphtounicode{abrevegrave}{1EB1}
\pdfglyphtounicode{abrevehookabove}{1EB3}
\pdfglyphtounicode{abrevetilde}{1EB5}
\pdfglyphtounicode{acaron}{01CE}
\pdfglyphtounicode{acircle}{24D0}
\pdfglyphtounicode{acircumflex}{00E2}
\pdfglyphtounicode{acircumflexacute}{1EA5}
\pdfglyphtounicode{acircumflexdotbelow}{1EAD}
\pdfglyphtounicode{acircumflexgrave}{1EA7}
\pdfglyphtounicode{acircumflexhookabove}{1EA9}
\pdfglyphtounicode{acircumflextilde}{1EAB}
\pdfglyphtounicode{acute}{00B4}
\pdfglyphtounicode{acutebelowcmb}{0317}
\pdfglyphtounicode{acutecmb}{0301}
\pdfglyphtounicode{acutecomb}{0301}
\pdfglyphtounicode{acutedeva}{0954}
\pdfglyphtounicode{acutelowmod}{02CF}
\pdfglyphtounicode{acutetonecmb}{0341}
\pdfglyphtounicode{acyrillic}{0430}
\pdfglyphtounicode{adblgrave}{0201}
\pdfglyphtounicode{addakgurmukhi}{0A71}
\pdfglyphtounicode{adeva}{0905}
\pdfglyphtounicode{adieresis}{00E4}
\pdfglyphtounicode{adieresiscyrillic}{04D3}
\pdfglyphtounicode{adieresismacron}{01DF}
\pdfglyphtounicode{adotbelow}{1EA1}
\pdfglyphtounicode{adotmacron}{01E1}
\pdfglyphtounicode{ae}{00E6}
\pdfglyphtounicode{aeacute}{01FD}
\pdfglyphtounicode{aekorean}{3150}
\pdfglyphtounicode{aemacron}{01E3}
\pdfglyphtounicode{afii00208}{2015}
\pdfglyphtounicode{afii08941}{20A4}
\pdfglyphtounicode{afii10017}{0410}
\pdfglyphtounicode{afii10018}{0411}
\pdfglyphtounicode{afii10019}{0412}
\pdfglyphtounicode{afii10020}{0413}
\pdfglyphtounicode{afii10021}{0414}
\pdfglyphtounicode{afii10022}{0415}
\pdfglyphtounicode{afii10023}{0401}
\pdfglyphtounicode{afii10024}{0416}
\pdfglyphtounicode{afii10025}{0417}
\pdfglyphtounicode{afii10026}{0418}
\pdfglyphtounicode{afii10027}{0419}
\pdfglyphtounicode{afii10028}{041A}
\pdfglyphtounicode{afii10029}{041B}
\pdfglyphtounicode{afii10030}{041C}
\pdfglyphtounicode{afii10031}{041D}
\pdfglyphtounicode{afii10032}{041E}
\pdfglyphtounicode{afii10033}{041F}
\pdfglyphtounicode{afii10034}{0420}
\pdfglyphtounicode{afii10035}{0421}
\pdfglyphtounicode{afii10036}{0422}
\pdfglyphtounicode{afii10037}{0423}
\pdfglyphtounicode{afii10038}{0424}
\pdfglyphtounicode{afii10039}{0425}
\pdfglyphtounicode{afii10040}{0426}
\pdfglyphtounicode{afii10041}{0427}
\pdfglyphtounicode{afii10042}{0428}
\pdfglyphtounicode{afii10043}{0429}
\pdfglyphtounicode{afii10044}{042A}
\pdfglyphtounicode{afii10045}{042B}
\pdfglyphtounicode{afii10046}{042C}
\pdfglyphtounicode{afii10047}{042D}
\pdfglyphtounicode{afii10048}{042E}
\pdfglyphtounicode{afii10049}{042F}
\pdfglyphtounicode{afii10050}{0490}
\pdfglyphtounicode{afii10051}{0402}
\pdfglyphtounicode{afii10052}{0403}
\pdfglyphtounicode{afii10053}{0404}
\pdfglyphtounicode{afii10054}{0405}
\pdfglyphtounicode{afii10055}{0406}
\pdfglyphtounicode{afii10056}{0407}
\pdfglyphtounicode{afii10057}{0408}
\pdfglyphtounicode{afii10058}{0409}
\pdfglyphtounicode{afii10059}{040A}
\pdfglyphtounicode{afii10060}{040B}
\pdfglyphtounicode{afii10061}{040C}
\pdfglyphtounicode{afii10062}{040E}
\pdfglyphtounicode{afii10063}{F6C4}
\pdfglyphtounicode{afii10064}{F6C5}
\pdfglyphtounicode{afii10065}{0430}
\pdfglyphtounicode{afii10066}{0431}
\pdfglyphtounicode{afii10067}{0432}
\pdfglyphtounicode{afii10068}{0433}
\pdfglyphtounicode{afii10069}{0434}
\pdfglyphtounicode{afii10070}{0435}
\pdfglyphtounicode{afii10071}{0451}
\pdfglyphtounicode{afii10072}{0436}
\pdfglyphtounicode{afii10073}{0437}
\pdfglyphtounicode{afii10074}{0438}
\pdfglyphtounicode{afii10075}{0439}
\pdfglyphtounicode{afii10076}{043A}
\pdfglyphtounicode{afii10077}{043B}
\pdfglyphtounicode{afii10078}{043C}
\pdfglyphtounicode{afii10079}{043D}
\pdfglyphtounicode{afii10080}{043E}
\pdfglyphtounicode{afii10081}{043F}
\pdfglyphtounicode{afii10082}{0440}
\pdfglyphtounicode{afii10083}{0441}
\pdfglyphtounicode{afii10084}{0442}
\pdfglyphtounicode{afii10085}{0443}
\pdfglyphtounicode{afii10086}{0444}
\pdfglyphtounicode{afii10087}{0445}
\pdfglyphtounicode{afii10088}{0446}
\pdfglyphtounicode{afii10089}{0447}
\pdfglyphtounicode{afii10090}{0448}
\pdfglyphtounicode{afii10091}{0449}
\pdfglyphtounicode{afii10092}{044A}
\pdfglyphtounicode{afii10093}{044B}
\pdfglyphtounicode{afii10094}{044C}
\pdfglyphtounicode{afii10095}{044D}
\pdfglyphtounicode{afii10096}{044E}
\pdfglyphtounicode{afii10097}{044F}
\pdfglyphtounicode{afii10098}{0491}
\pdfglyphtounicode{afii10099}{0452}
\pdfglyphtounicode{afii10100}{0453}
\pdfglyphtounicode{afii10101}{0454}
\pdfglyphtounicode{afii10102}{0455}
\pdfglyphtounicode{afii10103}{0456}
\pdfglyphtounicode{afii10104}{0457}
\pdfglyphtounicode{afii10105}{0458}
\pdfglyphtounicode{afii10106}{0459}
\pdfglyphtounicode{afii10107}{045A}
\pdfglyphtounicode{afii10108}{045B}
\pdfglyphtounicode{afii10109}{045C}
\pdfglyphtounicode{afii10110}{045E}
\pdfglyphtounicode{afii10145}{040F}
\pdfglyphtounicode{afii10146}{0462}
\pdfglyphtounicode{afii10147}{0472}
\pdfglyphtounicode{afii10148}{0474}
\pdfglyphtounicode{afii10192}{F6C6}
\pdfglyphtounicode{afii10193}{045F}
\pdfglyphtounicode{afii10194}{0463}
\pdfglyphtounicode{afii10195}{0473}
\pdfglyphtounicode{afii10196}{0475}
\pdfglyphtounicode{afii10831}{F6C7}
\pdfglyphtounicode{afii10832}{F6C8}
\pdfglyphtounicode{afii10846}{04D9}
\pdfglyphtounicode{afii299}{200E}
\pdfglyphtounicode{afii300}{200F}
\pdfglyphtounicode{afii301}{200D}
\pdfglyphtounicode{afii57381}{066A}
\pdfglyphtounicode{afii57388}{060C}
\pdfglyphtounicode{afii57392}{0660}
\pdfglyphtounicode{afii57393}{0661}
\pdfglyphtounicode{afii57394}{0662}
\pdfglyphtounicode{afii57395}{0663}
\pdfglyphtounicode{afii57396}{0664}
\pdfglyphtounicode{afii57397}{0665}
\pdfglyphtounicode{afii57398}{0666}
\pdfglyphtounicode{afii57399}{0667}
\pdfglyphtounicode{afii57400}{0668}
\pdfglyphtounicode{afii57401}{0669}
\pdfglyphtounicode{afii57403}{061B}
\pdfglyphtounicode{afii57407}{061F}
\pdfglyphtounicode{afii57409}{0621}
\pdfglyphtounicode{afii57410}{0622}
\pdfglyphtounicode{afii57411}{0623}
\pdfglyphtounicode{afii57412}{0624}
\pdfglyphtounicode{afii57413}{0625}
\pdfglyphtounicode{afii57414}{0626}
\pdfglyphtounicode{afii57415}{0627}
\pdfglyphtounicode{afii57416}{0628}
\pdfglyphtounicode{afii57417}{0629}
\pdfglyphtounicode{afii57418}{062A}
\pdfglyphtounicode{afii57419}{062B}
\pdfglyphtounicode{afii57420}{062C}
\pdfglyphtounicode{afii57421}{062D}
\pdfglyphtounicode{afii57422}{062E}
\pdfglyphtounicode{afii57423}{062F}
\pdfglyphtounicode{afii57424}{0630}
\pdfglyphtounicode{afii57425}{0631}
\pdfglyphtounicode{afii57426}{0632}
\pdfglyphtounicode{afii57427}{0633}
\pdfglyphtounicode{afii57428}{0634}
\pdfglyphtounicode{afii57429}{0635}
\pdfglyphtounicode{afii57430}{0636}
\pdfglyphtounicode{afii57431}{0637}
\pdfglyphtounicode{afii57432}{0638}
\pdfglyphtounicode{afii57433}{0639}
\pdfglyphtounicode{afii57434}{063A}
\pdfglyphtounicode{afii57440}{0640}
\pdfglyphtounicode{afii57441}{0641}
\pdfglyphtounicode{afii57442}{0642}
\pdfglyphtounicode{afii57443}{0643}
\pdfglyphtounicode{afii57444}{0644}
\pdfglyphtounicode{afii57445}{0645}
\pdfglyphtounicode{afii57446}{0646}
\pdfglyphtounicode{afii57448}{0648}
\pdfglyphtounicode{afii57449}{0649}
\pdfglyphtounicode{afii57450}{064A}
\pdfglyphtounicode{afii57451}{064B}
\pdfglyphtounicode{afii57452}{064C}
\pdfglyphtounicode{afii57453}{064D}
\pdfglyphtounicode{afii57454}{064E}
\pdfglyphtounicode{afii57455}{064F}
\pdfglyphtounicode{afii57456}{0650}
\pdfglyphtounicode{afii57457}{0651}
\pdfglyphtounicode{afii57458}{0652}
\pdfglyphtounicode{afii57470}{0647}
\pdfglyphtounicode{afii57505}{06A4}
\pdfglyphtounicode{afii57506}{067E}
\pdfglyphtounicode{afii57507}{0686}
\pdfglyphtounicode{afii57508}{0698}
\pdfglyphtounicode{afii57509}{06AF}
\pdfglyphtounicode{afii57511}{0679}
\pdfglyphtounicode{afii57512}{0688}
\pdfglyphtounicode{afii57513}{0691}
\pdfglyphtounicode{afii57514}{06BA}
\pdfglyphtounicode{afii57519}{06D2}
\pdfglyphtounicode{afii57534}{06D5}
\pdfglyphtounicode{afii57636}{20AA}
\pdfglyphtounicode{afii57645}{05BE}
\pdfglyphtounicode{afii57658}{05C3}
\pdfglyphtounicode{afii57664}{05D0}
\pdfglyphtounicode{afii57665}{05D1}
\pdfglyphtounicode{afii57666}{05D2}
\pdfglyphtounicode{afii57667}{05D3}
\pdfglyphtounicode{afii57668}{05D4}
\pdfglyphtounicode{afii57669}{05D5}
\pdfglyphtounicode{afii57670}{05D6}
\pdfglyphtounicode{afii57671}{05D7}
\pdfglyphtounicode{afii57672}{05D8}
\pdfglyphtounicode{afii57673}{05D9}
\pdfglyphtounicode{afii57674}{05DA}
\pdfglyphtounicode{afii57675}{05DB}
\pdfglyphtounicode{afii57676}{05DC}
\pdfglyphtounicode{afii57677}{05DD}
\pdfglyphtounicode{afii57678}{05DE}
\pdfglyphtounicode{afii57679}{05DF}
\pdfglyphtounicode{afii57680}{05E0}
\pdfglyphtounicode{afii57681}{05E1}
\pdfglyphtounicode{afii57682}{05E2}
\pdfglyphtounicode{afii57683}{05E3}
\pdfglyphtounicode{afii57684}{05E4}
\pdfglyphtounicode{afii57685}{05E5}
\pdfglyphtounicode{afii57686}{05E6}
\pdfglyphtounicode{afii57687}{05E7}
\pdfglyphtounicode{afii57688}{05E8}
\pdfglyphtounicode{afii57689}{05E9}
\pdfglyphtounicode{afii57690}{05EA}
\pdfglyphtounicode{afii57694}{FB2A}
\pdfglyphtounicode{afii57695}{FB2B}
\pdfglyphtounicode{afii57700}{FB4B}
\pdfglyphtounicode{afii57705}{FB1F}
\pdfglyphtounicode{afii57716}{05F0}
\pdfglyphtounicode{afii57717}{05F1}
\pdfglyphtounicode{afii57718}{05F2}
\pdfglyphtounicode{afii57723}{FB35}
\pdfglyphtounicode{afii57793}{05B4}
\pdfglyphtounicode{afii57794}{05B5}
\pdfglyphtounicode{afii57795}{05B6}
\pdfglyphtounicode{afii57796}{05BB}
\pdfglyphtounicode{afii57797}{05B8}
\pdfglyphtounicode{afii57798}{05B7}
\pdfglyphtounicode{afii57799}{05B0}
\pdfglyphtounicode{afii57800}{05B2}
\pdfglyphtounicode{afii57801}{05B1}
\pdfglyphtounicode{afii57802}{05B3}
\pdfglyphtounicode{afii57803}{05C2}
\pdfglyphtounicode{afii57804}{05C1}
\pdfglyphtounicode{afii57806}{05B9}
\pdfglyphtounicode{afii57807}{05BC}
\pdfglyphtounicode{afii57839}{05BD}
\pdfglyphtounicode{afii57841}{05BF}
\pdfglyphtounicode{afii57842}{05C0}
\pdfglyphtounicode{afii57929}{02BC}
\pdfglyphtounicode{afii61248}{2105}
\pdfglyphtounicode{afii61289}{2113}
\pdfglyphtounicode{afii61352}{2116}
\pdfglyphtounicode{afii61573}{202C}
\pdfglyphtounicode{afii61574}{202D}
\pdfglyphtounicode{afii61575}{202E}
\pdfglyphtounicode{afii61664}{200C}
\pdfglyphtounicode{afii63167}{066D}
\pdfglyphtounicode{afii64937}{02BD}
\pdfglyphtounicode{agrave}{00E0}
\pdfglyphtounicode{agujarati}{0A85}
\pdfglyphtounicode{agurmukhi}{0A05}
\pdfglyphtounicode{ahiragana}{3042}
\pdfglyphtounicode{ahookabove}{1EA3}
\pdfglyphtounicode{aibengali}{0990}
\pdfglyphtounicode{aibopomofo}{311E}
\pdfglyphtounicode{aideva}{0910}
\pdfglyphtounicode{aiecyrillic}{04D5}
\pdfglyphtounicode{aigujarati}{0A90}
\pdfglyphtounicode{aigurmukhi}{0A10}
\pdfglyphtounicode{aimatragurmukhi}{0A48}
\pdfglyphtounicode{ainarabic}{0639}
\pdfglyphtounicode{ainfinalarabic}{FECA}
\pdfglyphtounicode{aininitialarabic}{FECB}
\pdfglyphtounicode{ainmedialarabic}{FECC}
\pdfglyphtounicode{ainvertedbreve}{0203}
\pdfglyphtounicode{aivowelsignbengali}{09C8}
\pdfglyphtounicode{aivowelsigndeva}{0948}
\pdfglyphtounicode{aivowelsigngujarati}{0AC8}
\pdfglyphtounicode{akatakana}{30A2}
\pdfglyphtounicode{akatakanahalfwidth}{FF71}
\pdfglyphtounicode{akorean}{314F}
\pdfglyphtounicode{alef}{05D0}
\pdfglyphtounicode{alefarabic}{0627}
\pdfglyphtounicode{alefdageshhebrew}{FB30}
\pdfglyphtounicode{aleffinalarabic}{FE8E}
\pdfglyphtounicode{alefhamzaabovearabic}{0623}
\pdfglyphtounicode{alefhamzaabovefinalarabic}{FE84}
\pdfglyphtounicode{alefhamzabelowarabic}{0625}
\pdfglyphtounicode{alefhamzabelowfinalarabic}{FE88}
\pdfglyphtounicode{alefhebrew}{05D0}
\pdfglyphtounicode{aleflamedhebrew}{FB4F}
\pdfglyphtounicode{alefmaddaabovearabic}{0622}
\pdfglyphtounicode{alefmaddaabovefinalarabic}{FE82}
\pdfglyphtounicode{alefmaksuraarabic}{0649}
\pdfglyphtounicode{alefmaksurafinalarabic}{FEF0}
\pdfglyphtounicode{alefmaksurainitialarabic}{FEF3}
\pdfglyphtounicode{alefmaksuramedialarabic}{FEF4}
\pdfglyphtounicode{alefpatahhebrew}{FB2E}
\pdfglyphtounicode{alefqamatshebrew}{FB2F}
\pdfglyphtounicode{aleph}{2135}
\pdfglyphtounicode{allequal}{224C}
\pdfglyphtounicode{alpha}{03B1}
\pdfglyphtounicode{alphatonos}{03AC}
\pdfglyphtounicode{amacron}{0101}
\pdfglyphtounicode{amonospace}{FF41}
\pdfglyphtounicode{ampersand}{0026}
\pdfglyphtounicode{ampersandmonospace}{FF06}
\pdfglyphtounicode{ampersandsmall}{F726}
\pdfglyphtounicode{amsquare}{33C2}
\pdfglyphtounicode{anbopomofo}{3122}
\pdfglyphtounicode{angbopomofo}{3124}
\pdfglyphtounicode{angkhankhuthai}{0E5A}
\pdfglyphtounicode{angle}{2220}
\pdfglyphtounicode{anglebracketleft}{3008}
\pdfglyphtounicode{anglebracketleftvertical}{FE3F}
\pdfglyphtounicode{anglebracketright}{3009}
\pdfglyphtounicode{anglebracketrightvertical}{FE40}
\pdfglyphtounicode{angleleft}{2329}
\pdfglyphtounicode{angleright}{232A}
\pdfglyphtounicode{angstrom}{212B}
\pdfglyphtounicode{anoteleia}{0387}
\pdfglyphtounicode{anudattadeva}{0952}
\pdfglyphtounicode{anusvarabengali}{0982}
\pdfglyphtounicode{anusvaradeva}{0902}
\pdfglyphtounicode{anusvaragujarati}{0A82}
\pdfglyphtounicode{aogonek}{0105}
\pdfglyphtounicode{apaatosquare}{3300}
\pdfglyphtounicode{aparen}{249C}
\pdfglyphtounicode{apostrophearmenian}{055A}
\pdfglyphtounicode{apostrophemod}{02BC}
\pdfglyphtounicode{apple}{F8FF}
\pdfglyphtounicode{approaches}{2250}
\pdfglyphtounicode{approxequal}{2248}
\pdfglyphtounicode{approxequalorimage}{2252}
\pdfglyphtounicode{approximatelyequal}{2245}
\pdfglyphtounicode{araeaekorean}{318E}
\pdfglyphtounicode{araeakorean}{318D}
\pdfglyphtounicode{arc}{2312}
\pdfglyphtounicode{arighthalfring}{1E9A}
\pdfglyphtounicode{aring}{00E5}
\pdfglyphtounicode{aringacute}{01FB}
\pdfglyphtounicode{aringbelow}{1E01}
\pdfglyphtounicode{arrowboth}{2194}
\pdfglyphtounicode{arrowdashdown}{21E3}
\pdfglyphtounicode{arrowdashleft}{21E0}
\pdfglyphtounicode{arrowdashright}{21E2}
\pdfglyphtounicode{arrowdashup}{21E1}
\pdfglyphtounicode{arrowdblboth}{21D4}
\pdfglyphtounicode{arrowdbldown}{21D3}
\pdfglyphtounicode{arrowdblleft}{21D0}
\pdfglyphtounicode{arrowdblright}{21D2}
\pdfglyphtounicode{arrowdblup}{21D1}
\pdfglyphtounicode{arrowdown}{2193}
\pdfglyphtounicode{arrowdownleft}{2199}
\pdfglyphtounicode{arrowdownright}{2198}
\pdfglyphtounicode{arrowdownwhite}{21E9}
\pdfglyphtounicode{arrowheaddownmod}{02C5}
\pdfglyphtounicode{arrowheadleftmod}{02C2}
\pdfglyphtounicode{arrowheadrightmod}{02C3}
\pdfglyphtounicode{arrowheadupmod}{02C4}
\pdfglyphtounicode{arrowhorizex}{F8E7}
\pdfglyphtounicode{arrowleft}{2190}
\pdfglyphtounicode{arrowleftdbl}{21D0}
\pdfglyphtounicode{arrowleftdblstroke}{21CD}
\pdfglyphtounicode{arrowleftoverright}{21C6}
\pdfglyphtounicode{arrowleftwhite}{21E6}
\pdfglyphtounicode{arrowright}{2192}
\pdfglyphtounicode{arrowrightdblstroke}{21CF}
\pdfglyphtounicode{arrowrightheavy}{279E}
\pdfglyphtounicode{arrowrightoverleft}{21C4}
\pdfglyphtounicode{arrowrightwhite}{21E8}
\pdfglyphtounicode{arrowtableft}{21E4}
\pdfglyphtounicode{arrowtabright}{21E5}
\pdfglyphtounicode{arrowup}{2191}
\pdfglyphtounicode{arrowupdn}{2195}
\pdfglyphtounicode{arrowupdnbse}{21A8}
\pdfglyphtounicode{arrowupdownbase}{21A8}
\pdfglyphtounicode{arrowupleft}{2196}
\pdfglyphtounicode{arrowupleftofdown}{21C5}
\pdfglyphtounicode{arrowupright}{2197}
\pdfglyphtounicode{arrowupwhite}{21E7}
\pdfglyphtounicode{arrowvertex}{F8E6}
\pdfglyphtounicode{asciicircum}{005E}
\pdfglyphtounicode{asciicircummonospace}{FF3E}
\pdfglyphtounicode{asciitilde}{007E}
\pdfglyphtounicode{asciitildemonospace}{FF5E}
\pdfglyphtounicode{ascript}{0251}
\pdfglyphtounicode{ascriptturned}{0252}
\pdfglyphtounicode{asmallhiragana}{3041}
\pdfglyphtounicode{asmallkatakana}{30A1}
\pdfglyphtounicode{asmallkatakanahalfwidth}{FF67}
\pdfglyphtounicode{asterisk}{002A}
\pdfglyphtounicode{asteriskaltonearabic}{066D}
\pdfglyphtounicode{asteriskarabic}{066D}
\pdfglyphtounicode{asteriskmath}{2217}
\pdfglyphtounicode{asteriskmonospace}{FF0A}
\pdfglyphtounicode{asterisksmall}{FE61}
\pdfglyphtounicode{asterism}{2042}
\pdfglyphtounicode{asuperior}{F6E9}
\pdfglyphtounicode{asymptoticallyequal}{2243}
\pdfglyphtounicode{at}{0040}
\pdfglyphtounicode{atilde}{00E3}
\pdfglyphtounicode{atmonospace}{FF20}
\pdfglyphtounicode{atsmall}{FE6B}
\pdfglyphtounicode{aturned}{0250}
\pdfglyphtounicode{aubengali}{0994}
\pdfglyphtounicode{aubopomofo}{3120}
\pdfglyphtounicode{audeva}{0914}
\pdfglyphtounicode{augujarati}{0A94}
\pdfglyphtounicode{augurmukhi}{0A14}
\pdfglyphtounicode{aulengthmarkbengali}{09D7}
\pdfglyphtounicode{aumatragurmukhi}{0A4C}
\pdfglyphtounicode{auvowelsignbengali}{09CC}
\pdfglyphtounicode{auvowelsigndeva}{094C}
\pdfglyphtounicode{auvowelsigngujarati}{0ACC}
\pdfglyphtounicode{avagrahadeva}{093D}
\pdfglyphtounicode{aybarmenian}{0561}
\pdfglyphtounicode{ayin}{05E2}
\pdfglyphtounicode{ayinaltonehebrew}{FB20}
\pdfglyphtounicode{ayinhebrew}{05E2}
\pdfglyphtounicode{b}{0062}
\pdfglyphtounicode{babengali}{09AC}
\pdfglyphtounicode{backslash}{005C}
\pdfglyphtounicode{backslashmonospace}{FF3C}
\pdfglyphtounicode{badeva}{092C}
\pdfglyphtounicode{bagujarati}{0AAC}
\pdfglyphtounicode{bagurmukhi}{0A2C}
\pdfglyphtounicode{bahiragana}{3070}
\pdfglyphtounicode{bahtthai}{0E3F}
\pdfglyphtounicode{bakatakana}{30D0}
\pdfglyphtounicode{bar}{007C}
\pdfglyphtounicode{barmonospace}{FF5C}
\pdfglyphtounicode{bbopomofo}{3105}
\pdfglyphtounicode{bcircle}{24D1}
\pdfglyphtounicode{bdotaccent}{1E03}
\pdfglyphtounicode{bdotbelow}{1E05}
\pdfglyphtounicode{beamedsixteenthnotes}{266C}
\pdfglyphtounicode{because}{2235}
\pdfglyphtounicode{becyrillic}{0431}
\pdfglyphtounicode{beharabic}{0628}
\pdfglyphtounicode{behfinalarabic}{FE90}
\pdfglyphtounicode{behinitialarabic}{FE91}
\pdfglyphtounicode{behiragana}{3079}
\pdfglyphtounicode{behmedialarabic}{FE92}
\pdfglyphtounicode{behmeeminitialarabic}{FC9F}
\pdfglyphtounicode{behmeemisolatedarabic}{FC08}
\pdfglyphtounicode{behnoonfinalarabic}{FC6D}
\pdfglyphtounicode{bekatakana}{30D9}
\pdfglyphtounicode{benarmenian}{0562}
\pdfglyphtounicode{bet}{05D1}
\pdfglyphtounicode{beta}{03B2}
\pdfglyphtounicode{betasymbolgreek}{03D0}
\pdfglyphtounicode{betdagesh}{FB31}
\pdfglyphtounicode{betdageshhebrew}{FB31}
\pdfglyphtounicode{bethebrew}{05D1}
\pdfglyphtounicode{betrafehebrew}{FB4C}
\pdfglyphtounicode{bhabengali}{09AD}
\pdfglyphtounicode{bhadeva}{092D}
\pdfglyphtounicode{bhagujarati}{0AAD}
\pdfglyphtounicode{bhagurmukhi}{0A2D}
\pdfglyphtounicode{bhook}{0253}
\pdfglyphtounicode{bihiragana}{3073}
\pdfglyphtounicode{bikatakana}{30D3}
\pdfglyphtounicode{bilabialclick}{0298}
\pdfglyphtounicode{bindigurmukhi}{0A02}
\pdfglyphtounicode{birusquare}{3331}
\pdfglyphtounicode{blackcircle}{25CF}
\pdfglyphtounicode{blackdiamond}{25C6}
\pdfglyphtounicode{blackdownpointingtriangle}{25BC}
\pdfglyphtounicode{blackleftpointingpointer}{25C4}
\pdfglyphtounicode{blackleftpointingtriangle}{25C0}
\pdfglyphtounicode{blacklenticularbracketleft}{3010}
\pdfglyphtounicode{blacklenticularbracketleftvertical}{FE3B}
\pdfglyphtounicode{blacklenticularbracketright}{3011}
\pdfglyphtounicode{blacklenticularbracketrightvertical}{FE3C}
\pdfglyphtounicode{blacklowerlefttriangle}{25E3}
\pdfglyphtounicode{blacklowerrighttriangle}{25E2}
\pdfglyphtounicode{blackrectangle}{25AC}
\pdfglyphtounicode{blackrightpointingpointer}{25BA}
\pdfglyphtounicode{blackrightpointingtriangle}{25B6}
\pdfglyphtounicode{blacksmallsquare}{25AA}
\pdfglyphtounicode{blacksmilingface}{263B}
\pdfglyphtounicode{blacksquare}{25A0}
\pdfglyphtounicode{blackstar}{2605}
\pdfglyphtounicode{blackupperlefttriangle}{25E4}
\pdfglyphtounicode{blackupperrighttriangle}{25E5}
\pdfglyphtounicode{blackuppointingsmalltriangle}{25B4}
\pdfglyphtounicode{blackuppointingtriangle}{25B2}
\pdfglyphtounicode{blank}{2423}
\pdfglyphtounicode{blinebelow}{1E07}
\pdfglyphtounicode{block}{2588}
\pdfglyphtounicode{bmonospace}{FF42}
\pdfglyphtounicode{bobaimaithai}{0E1A}
\pdfglyphtounicode{bohiragana}{307C}
\pdfglyphtounicode{bokatakana}{30DC}
\pdfglyphtounicode{bparen}{249D}
\pdfglyphtounicode{bqsquare}{33C3}
\pdfglyphtounicode{braceex}{F8F4}
\pdfglyphtounicode{braceleft}{007B}
\pdfglyphtounicode{braceleftbt}{F8F3}
\pdfglyphtounicode{braceleftmid}{F8F2}
\pdfglyphtounicode{braceleftmonospace}{FF5B}
\pdfglyphtounicode{braceleftsmall}{FE5B}
\pdfglyphtounicode{bracelefttp}{F8F1}
\pdfglyphtounicode{braceleftvertical}{FE37}
\pdfglyphtounicode{braceright}{007D}
\pdfglyphtounicode{bracerightbt}{F8FE}
\pdfglyphtounicode{bracerightmid}{F8FD}
\pdfglyphtounicode{bracerightmonospace}{FF5D}
\pdfglyphtounicode{bracerightsmall}{FE5C}
\pdfglyphtounicode{bracerighttp}{F8FC}
\pdfglyphtounicode{bracerightvertical}{FE38}
\pdfglyphtounicode{bracketleft}{005B}
\pdfglyphtounicode{bracketleftbt}{F8F0}
\pdfglyphtounicode{bracketleftex}{F8EF}
\pdfglyphtounicode{bracketleftmonospace}{FF3B}
\pdfglyphtounicode{bracketlefttp}{F8EE}
\pdfglyphtounicode{bracketright}{005D}
\pdfglyphtounicode{bracketrightbt}{F8FB}
\pdfglyphtounicode{bracketrightex}{F8FA}
\pdfglyphtounicode{bracketrightmonospace}{FF3D}
\pdfglyphtounicode{bracketrighttp}{F8F9}
\pdfglyphtounicode{breve}{02D8}
\pdfglyphtounicode{brevebelowcmb}{032E}
\pdfglyphtounicode{brevecmb}{0306}
\pdfglyphtounicode{breveinvertedbelowcmb}{032F}
\pdfglyphtounicode{breveinvertedcmb}{0311}
\pdfglyphtounicode{breveinverteddoublecmb}{0361}
\pdfglyphtounicode{bridgebelowcmb}{032A}
\pdfglyphtounicode{bridgeinvertedbelowcmb}{033A}
\pdfglyphtounicode{brokenbar}{00A6}
\pdfglyphtounicode{bstroke}{0180}
\pdfglyphtounicode{bsuperior}{F6EA}
\pdfglyphtounicode{btopbar}{0183}
\pdfglyphtounicode{buhiragana}{3076}
\pdfglyphtounicode{bukatakana}{30D6}
\pdfglyphtounicode{bullet}{2022}
\pdfglyphtounicode{bulletinverse}{25D8}
\pdfglyphtounicode{bulletoperator}{2219}
\pdfglyphtounicode{bullseye}{25CE}
\pdfglyphtounicode{c}{0063}
\pdfglyphtounicode{caarmenian}{056E}
\pdfglyphtounicode{cabengali}{099A}
\pdfglyphtounicode{cacute}{0107}
\pdfglyphtounicode{cadeva}{091A}
\pdfglyphtounicode{cagujarati}{0A9A}
\pdfglyphtounicode{cagurmukhi}{0A1A}
\pdfglyphtounicode{calsquare}{3388}
\pdfglyphtounicode{candrabindubengali}{0981}
\pdfglyphtounicode{candrabinducmb}{0310}
\pdfglyphtounicode{candrabindudeva}{0901}
\pdfglyphtounicode{candrabindugujarati}{0A81}
\pdfglyphtounicode{capslock}{21EA}
\pdfglyphtounicode{careof}{2105}
\pdfglyphtounicode{caron}{02C7}
\pdfglyphtounicode{caronbelowcmb}{032C}
\pdfglyphtounicode{caroncmb}{030C}
\pdfglyphtounicode{carriagereturn}{21B5}
\pdfglyphtounicode{cbopomofo}{3118}
\pdfglyphtounicode{ccaron}{010D}
\pdfglyphtounicode{ccedilla}{00E7}
\pdfglyphtounicode{ccedillaacute}{1E09}
\pdfglyphtounicode{ccircle}{24D2}
\pdfglyphtounicode{ccircumflex}{0109}
\pdfglyphtounicode{ccurl}{0255}
\pdfglyphtounicode{cdot}{010B}
\pdfglyphtounicode{cdotaccent}{010B}
\pdfglyphtounicode{cdsquare}{33C5}
\pdfglyphtounicode{cedilla}{00B8}
\pdfglyphtounicode{cedillacmb}{0327}
\pdfglyphtounicode{cent}{00A2}
\pdfglyphtounicode{centigrade}{2103}
\pdfglyphtounicode{centinferior}{F6DF}
\pdfglyphtounicode{centmonospace}{FFE0}
\pdfglyphtounicode{centoldstyle}{F7A2}
\pdfglyphtounicode{centsuperior}{F6E0}
\pdfglyphtounicode{chaarmenian}{0579}
\pdfglyphtounicode{chabengali}{099B}
\pdfglyphtounicode{chadeva}{091B}
\pdfglyphtounicode{chagujarati}{0A9B}
\pdfglyphtounicode{chagurmukhi}{0A1B}
\pdfglyphtounicode{chbopomofo}{3114}
\pdfglyphtounicode{cheabkhasiancyrillic}{04BD}
\pdfglyphtounicode{checkmark}{2713}
\pdfglyphtounicode{checyrillic}{0447}
\pdfglyphtounicode{chedescenderabkhasiancyrillic}{04BF}
\pdfglyphtounicode{chedescendercyrillic}{04B7}
\pdfglyphtounicode{chedieresiscyrillic}{04F5}
\pdfglyphtounicode{cheharmenian}{0573}
\pdfglyphtounicode{chekhakassiancyrillic}{04CC}
\pdfglyphtounicode{cheverticalstrokecyrillic}{04B9}
\pdfglyphtounicode{chi}{03C7}
\pdfglyphtounicode{chieuchacirclekorean}{3277}
\pdfglyphtounicode{chieuchaparenkorean}{3217}
\pdfglyphtounicode{chieuchcirclekorean}{3269}
\pdfglyphtounicode{chieuchkorean}{314A}
\pdfglyphtounicode{chieuchparenkorean}{3209}
\pdfglyphtounicode{chochangthai}{0E0A}
\pdfglyphtounicode{chochanthai}{0E08}
\pdfglyphtounicode{chochingthai}{0E09}
\pdfglyphtounicode{chochoethai}{0E0C}
\pdfglyphtounicode{chook}{0188}
\pdfglyphtounicode{cieucacirclekorean}{3276}
\pdfglyphtounicode{cieucaparenkorean}{3216}
\pdfglyphtounicode{cieuccirclekorean}{3268}
\pdfglyphtounicode{cieuckorean}{3148}
\pdfglyphtounicode{cieucparenkorean}{3208}
\pdfglyphtounicode{cieucuparenkorean}{321C}
\pdfglyphtounicode{circle}{25CB}
\pdfglyphtounicode{circlemultiply}{2297}
\pdfglyphtounicode{circleot}{2299}
\pdfglyphtounicode{circleplus}{2295}
\pdfglyphtounicode{circlepostalmark}{3036}
\pdfglyphtounicode{circlewithlefthalfblack}{25D0}
\pdfglyphtounicode{circlewithrighthalfblack}{25D1}
\pdfglyphtounicode{circumflex}{02C6}
\pdfglyphtounicode{circumflexbelowcmb}{032D}
\pdfglyphtounicode{circumflexcmb}{0302}
\pdfglyphtounicode{clear}{2327}
\pdfglyphtounicode{clickalveolar}{01C2}
\pdfglyphtounicode{clickdental}{01C0}
\pdfglyphtounicode{clicklateral}{01C1}
\pdfglyphtounicode{clickretroflex}{01C3}
\pdfglyphtounicode{club}{2663}
\pdfglyphtounicode{clubsuitblack}{2663}
\pdfglyphtounicode{clubsuitwhite}{2667}
\pdfglyphtounicode{cmcubedsquare}{33A4}
\pdfglyphtounicode{cmonospace}{FF43}
\pdfglyphtounicode{cmsquaredsquare}{33A0}
\pdfglyphtounicode{coarmenian}{0581}
\pdfglyphtounicode{colon}{003A}
\pdfglyphtounicode{colonmonetary}{20A1}
\pdfglyphtounicode{colonmonospace}{FF1A}
\pdfglyphtounicode{colonsign}{20A1}
\pdfglyphtounicode{colonsmall}{FE55}
\pdfglyphtounicode{colontriangularhalfmod}{02D1}
\pdfglyphtounicode{colontriangularmod}{02D0}
\pdfglyphtounicode{comma}{002C}
\pdfglyphtounicode{commaabovecmb}{0313}
\pdfglyphtounicode{commaaboverightcmb}{0315}
\pdfglyphtounicode{commaaccent}{F6C3}
\pdfglyphtounicode{commaarabic}{060C}
\pdfglyphtounicode{commaarmenian}{055D}
\pdfglyphtounicode{commainferior}{F6E1}
\pdfglyphtounicode{commamonospace}{FF0C}
\pdfglyphtounicode{commareversedabovecmb}{0314}
\pdfglyphtounicode{commareversedmod}{02BD}
\pdfglyphtounicode{commasmall}{FE50}
\pdfglyphtounicode{commasuperior}{F6E2}
\pdfglyphtounicode{commaturnedabovecmb}{0312}
\pdfglyphtounicode{commaturnedmod}{02BB}
\pdfglyphtounicode{compass}{263C}
\pdfglyphtounicode{congruent}{2245}
\pdfglyphtounicode{contourintegral}{222E}
\pdfglyphtounicode{control}{2303}
\pdfglyphtounicode{controlACK}{0006}
\pdfglyphtounicode{controlBEL}{0007}
\pdfglyphtounicode{controlBS}{0008}
\pdfglyphtounicode{controlCAN}{0018}
\pdfglyphtounicode{controlCR}{000D}
\pdfglyphtounicode{controlDC1}{0011}
\pdfglyphtounicode{controlDC2}{0012}
\pdfglyphtounicode{controlDC3}{0013}
\pdfglyphtounicode{controlDC4}{0014}
\pdfglyphtounicode{controlDEL}{007F}
\pdfglyphtounicode{controlDLE}{0010}
\pdfglyphtounicode{controlEM}{0019}
\pdfglyphtounicode{controlENQ}{0005}
\pdfglyphtounicode{controlEOT}{0004}
\pdfglyphtounicode{controlESC}{001B}
\pdfglyphtounicode{controlETB}{0017}
\pdfglyphtounicode{controlETX}{0003}
\pdfglyphtounicode{controlFF}{000C}
\pdfglyphtounicode{controlFS}{001C}
\pdfglyphtounicode{controlGS}{001D}
\pdfglyphtounicode{controlHT}{0009}
\pdfglyphtounicode{controlLF}{000A}
\pdfglyphtounicode{controlNAK}{0015}
\pdfglyphtounicode{controlRS}{001E}
\pdfglyphtounicode{controlSI}{000F}
\pdfglyphtounicode{controlSO}{000E}
\pdfglyphtounicode{controlSOT}{0002}
\pdfglyphtounicode{controlSTX}{0001}
\pdfglyphtounicode{controlSUB}{001A}
\pdfglyphtounicode{controlSYN}{0016}
\pdfglyphtounicode{controlUS}{001F}
\pdfglyphtounicode{controlVT}{000B}
\pdfglyphtounicode{copyright}{00A9}
\pdfglyphtounicode{copyrightsans}{F8E9}
\pdfglyphtounicode{copyrightserif}{F6D9}
\pdfglyphtounicode{cornerbracketleft}{300C}
\pdfglyphtounicode{cornerbracketlefthalfwidth}{FF62}
\pdfglyphtounicode{cornerbracketleftvertical}{FE41}
\pdfglyphtounicode{cornerbracketright}{300D}
\pdfglyphtounicode{cornerbracketrighthalfwidth}{FF63}
\pdfglyphtounicode{cornerbracketrightvertical}{FE42}
\pdfglyphtounicode{corporationsquare}{337F}
\pdfglyphtounicode{cosquare}{33C7}
\pdfglyphtounicode{coverkgsquare}{33C6}
\pdfglyphtounicode{cparen}{249E}
\pdfglyphtounicode{cruzeiro}{20A2}
\pdfglyphtounicode{cstretched}{0297}
\pdfglyphtounicode{curlyand}{22CF}
\pdfglyphtounicode{curlyor}{22CE}
\pdfglyphtounicode{currency}{00A4}
\pdfglyphtounicode{cyrBreve}{F6D1}
\pdfglyphtounicode{cyrFlex}{F6D2}
\pdfglyphtounicode{cyrbreve}{F6D4}
\pdfglyphtounicode{cyrflex}{F6D5}
\pdfglyphtounicode{d}{0064}
\pdfglyphtounicode{daarmenian}{0564}
\pdfglyphtounicode{dabengali}{09A6}
\pdfglyphtounicode{dadarabic}{0636}
\pdfglyphtounicode{dadeva}{0926}
\pdfglyphtounicode{dadfinalarabic}{FEBE}
\pdfglyphtounicode{dadinitialarabic}{FEBF}
\pdfglyphtounicode{dadmedialarabic}{FEC0}
\pdfglyphtounicode{dagesh}{05BC}
\pdfglyphtounicode{dageshhebrew}{05BC}
\pdfglyphtounicode{dagger}{2020}
\pdfglyphtounicode{daggerdbl}{2021}
\pdfglyphtounicode{dagujarati}{0AA6}
\pdfglyphtounicode{dagurmukhi}{0A26}
\pdfglyphtounicode{dahiragana}{3060}
\pdfglyphtounicode{dakatakana}{30C0}
\pdfglyphtounicode{dalarabic}{062F}
\pdfglyphtounicode{dalet}{05D3}
\pdfglyphtounicode{daletdagesh}{FB33}
\pdfglyphtounicode{daletdageshhebrew}{FB33}
\pdfglyphtounicode{dalethebrew}{05D3}
\pdfglyphtounicode{dalfinalarabic}{FEAA}
\pdfglyphtounicode{dammaarabic}{064F}
\pdfglyphtounicode{dammalowarabic}{064F}
\pdfglyphtounicode{dammatanaltonearabic}{064C}
\pdfglyphtounicode{dammatanarabic}{064C}
\pdfglyphtounicode{danda}{0964}
\pdfglyphtounicode{dargahebrew}{05A7}
\pdfglyphtounicode{dargalefthebrew}{05A7}
\pdfglyphtounicode{dasiapneumatacyrilliccmb}{0485}
\pdfglyphtounicode{dblGrave}{F6D3}
\pdfglyphtounicode{dblanglebracketleft}{300A}
\pdfglyphtounicode{dblanglebracketleftvertical}{FE3D}
\pdfglyphtounicode{dblanglebracketright}{300B}
\pdfglyphtounicode{dblanglebracketrightvertical}{FE3E}
\pdfglyphtounicode{dblarchinvertedbelowcmb}{032B}
\pdfglyphtounicode{dblarrowleft}{21D4}
\pdfglyphtounicode{dblarrowright}{21D2}
\pdfglyphtounicode{dbldanda}{0965}
\pdfglyphtounicode{dblgrave}{F6D6}
\pdfglyphtounicode{dblgravecmb}{030F}
\pdfglyphtounicode{dblintegral}{222C}
\pdfglyphtounicode{dbllowline}{2017}
\pdfglyphtounicode{dbllowlinecmb}{0333}
\pdfglyphtounicode{dbloverlinecmb}{033F}
\pdfglyphtounicode{dblprimemod}{02BA}
\pdfglyphtounicode{dblverticalbar}{2016}
\pdfglyphtounicode{dblverticallineabovecmb}{030E}
\pdfglyphtounicode{dbopomofo}{3109}
\pdfglyphtounicode{dbsquare}{33C8}
\pdfglyphtounicode{dcaron}{010F}
\pdfglyphtounicode{dcedilla}{1E11}
\pdfglyphtounicode{dcircle}{24D3}
\pdfglyphtounicode{dcircumflexbelow}{1E13}
\pdfglyphtounicode{dcroat}{0111}
\pdfglyphtounicode{ddabengali}{09A1}
\pdfglyphtounicode{ddadeva}{0921}
\pdfglyphtounicode{ddagujarati}{0AA1}
\pdfglyphtounicode{ddagurmukhi}{0A21}
\pdfglyphtounicode{ddalarabic}{0688}
\pdfglyphtounicode{ddalfinalarabic}{FB89}
\pdfglyphtounicode{dddhadeva}{095C}
\pdfglyphtounicode{ddhabengali}{09A2}
\pdfglyphtounicode{ddhadeva}{0922}
\pdfglyphtounicode{ddhagujarati}{0AA2}
\pdfglyphtounicode{ddhagurmukhi}{0A22}
\pdfglyphtounicode{ddotaccent}{1E0B}
\pdfglyphtounicode{ddotbelow}{1E0D}
\pdfglyphtounicode{decimalseparatorarabic}{066B}
\pdfglyphtounicode{decimalseparatorpersian}{066B}
\pdfglyphtounicode{decyrillic}{0434}
\pdfglyphtounicode{degree}{00B0}
\pdfglyphtounicode{dehihebrew}{05AD}
\pdfglyphtounicode{dehiragana}{3067}
\pdfglyphtounicode{deicoptic}{03EF}
\pdfglyphtounicode{dekatakana}{30C7}
\pdfglyphtounicode{deleteleft}{232B}
\pdfglyphtounicode{deleteright}{2326}
\pdfglyphtounicode{delta}{03B4}
\pdfglyphtounicode{deltaturned}{018D}
\pdfglyphtounicode{denominatorminusonenumeratorbengali}{09F8}
\pdfglyphtounicode{dezh}{02A4}
\pdfglyphtounicode{dhabengali}{09A7}
\pdfglyphtounicode{dhadeva}{0927}
\pdfglyphtounicode{dhagujarati}{0AA7}
\pdfglyphtounicode{dhagurmukhi}{0A27}
\pdfglyphtounicode{dhook}{0257}
\pdfglyphtounicode{dialytikatonos}{0385}
\pdfglyphtounicode{dialytikatonoscmb}{0344}
\pdfglyphtounicode{diamond}{2666}
\pdfglyphtounicode{diamondsuitwhite}{2662}
\pdfglyphtounicode{dieresis}{00A8}
\pdfglyphtounicode{dieresisacute}{F6D7}
\pdfglyphtounicode{dieresisbelowcmb}{0324}
\pdfglyphtounicode{dieresiscmb}{0308}
\pdfglyphtounicode{dieresisgrave}{F6D8}
\pdfglyphtounicode{dieresistonos}{0385}
\pdfglyphtounicode{dihiragana}{3062}
\pdfglyphtounicode{dikatakana}{30C2}
\pdfglyphtounicode{dittomark}{3003}
\pdfglyphtounicode{divide}{00F7}
\pdfglyphtounicode{divides}{2223}
\pdfglyphtounicode{divisionslash}{2215}
\pdfglyphtounicode{djecyrillic}{0452}
\pdfglyphtounicode{dkshade}{2593}
\pdfglyphtounicode{dlinebelow}{1E0F}
\pdfglyphtounicode{dlsquare}{3397}
\pdfglyphtounicode{dmacron}{0111}
\pdfglyphtounicode{dmonospace}{FF44}
\pdfglyphtounicode{dnblock}{2584}
\pdfglyphtounicode{dochadathai}{0E0E}
\pdfglyphtounicode{dodekthai}{0E14}
\pdfglyphtounicode{dohiragana}{3069}
\pdfglyphtounicode{dokatakana}{30C9}
\pdfglyphtounicode{dollar}{0024}
\pdfglyphtounicode{dollarinferior}{F6E3}
\pdfglyphtounicode{dollarmonospace}{FF04}
\pdfglyphtounicode{dollaroldstyle}{F724}
\pdfglyphtounicode{dollarsmall}{FE69}
\pdfglyphtounicode{dollarsuperior}{F6E4}
\pdfglyphtounicode{dong}{20AB}
\pdfglyphtounicode{dorusquare}{3326}
\pdfglyphtounicode{dotaccent}{02D9}
\pdfglyphtounicode{dotaccentcmb}{0307}
\pdfglyphtounicode{dotbelowcmb}{0323}
\pdfglyphtounicode{dotbelowcomb}{0323}
\pdfglyphtounicode{dotkatakana}{30FB}
\pdfglyphtounicode{dotlessi}{0131}
\pdfglyphtounicode{dotlessj}{F6BE}
\pdfglyphtounicode{dotlessjstrokehook}{0284}
\pdfglyphtounicode{dotmath}{22C5}
\pdfglyphtounicode{dottedcircle}{25CC}
\pdfglyphtounicode{doubleyodpatah}{FB1F}
\pdfglyphtounicode{doubleyodpatahhebrew}{FB1F}
\pdfglyphtounicode{downtackbelowcmb}{031E}
\pdfglyphtounicode{downtackmod}{02D5}
\pdfglyphtounicode{dparen}{249F}
\pdfglyphtounicode{dsuperior}{F6EB}
\pdfglyphtounicode{dtail}{0256}
\pdfglyphtounicode{dtopbar}{018C}
\pdfglyphtounicode{duhiragana}{3065}
\pdfglyphtounicode{dukatakana}{30C5}
\pdfglyphtounicode{dz}{01F3}
\pdfglyphtounicode{dzaltone}{02A3}
\pdfglyphtounicode{dzcaron}{01C6}
\pdfglyphtounicode{dzcurl}{02A5}
\pdfglyphtounicode{dzeabkhasiancyrillic}{04E1}
\pdfglyphtounicode{dzecyrillic}{0455}
\pdfglyphtounicode{dzhecyrillic}{045F}
\pdfglyphtounicode{e}{0065}
\pdfglyphtounicode{eacute}{00E9}
\pdfglyphtounicode{earth}{2641}
\pdfglyphtounicode{ebengali}{098F}
\pdfglyphtounicode{ebopomofo}{311C}
\pdfglyphtounicode{ebreve}{0115}
\pdfglyphtounicode{ecandradeva}{090D}
\pdfglyphtounicode{ecandragujarati}{0A8D}
\pdfglyphtounicode{ecandravowelsigndeva}{0945}
\pdfglyphtounicode{ecandravowelsigngujarati}{0AC5}
\pdfglyphtounicode{ecaron}{011B}
\pdfglyphtounicode{ecedillabreve}{1E1D}
\pdfglyphtounicode{echarmenian}{0565}
\pdfglyphtounicode{echyiwnarmenian}{0587}
\pdfglyphtounicode{ecircle}{24D4}
\pdfglyphtounicode{ecircumflex}{00EA}
\pdfglyphtounicode{ecircumflexacute}{1EBF}
\pdfglyphtounicode{ecircumflexbelow}{1E19}
\pdfglyphtounicode{ecircumflexdotbelow}{1EC7}
\pdfglyphtounicode{ecircumflexgrave}{1EC1}
\pdfglyphtounicode{ecircumflexhookabove}{1EC3}
\pdfglyphtounicode{ecircumflextilde}{1EC5}
\pdfglyphtounicode{ecyrillic}{0454}
\pdfglyphtounicode{edblgrave}{0205}
\pdfglyphtounicode{edeva}{090F}
\pdfglyphtounicode{edieresis}{00EB}
\pdfglyphtounicode{edot}{0117}
\pdfglyphtounicode{edotaccent}{0117}
\pdfglyphtounicode{edotbelow}{1EB9}
\pdfglyphtounicode{eegurmukhi}{0A0F}
\pdfglyphtounicode{eematragurmukhi}{0A47}
\pdfglyphtounicode{efcyrillic}{0444}
\pdfglyphtounicode{egrave}{00E8}
\pdfglyphtounicode{egujarati}{0A8F}
\pdfglyphtounicode{eharmenian}{0567}
\pdfglyphtounicode{ehbopomofo}{311D}
\pdfglyphtounicode{ehiragana}{3048}
\pdfglyphtounicode{ehookabove}{1EBB}
\pdfglyphtounicode{eibopomofo}{311F}
\pdfglyphtounicode{eight}{0038}
\pdfglyphtounicode{eightarabic}{0668}
\pdfglyphtounicode{eightbengali}{09EE}
\pdfglyphtounicode{eightcircle}{2467}
\pdfglyphtounicode{eightcircleinversesansserif}{2791}
\pdfglyphtounicode{eightdeva}{096E}
\pdfglyphtounicode{eighteencircle}{2471}
\pdfglyphtounicode{eighteenparen}{2485}
\pdfglyphtounicode{eighteenperiod}{2499}
\pdfglyphtounicode{eightgujarati}{0AEE}
\pdfglyphtounicode{eightgurmukhi}{0A6E}
\pdfglyphtounicode{eighthackarabic}{0668}
\pdfglyphtounicode{eighthangzhou}{3028}
\pdfglyphtounicode{eighthnotebeamed}{266B}
\pdfglyphtounicode{eightideographicparen}{3227}
\pdfglyphtounicode{eightinferior}{2088}
\pdfglyphtounicode{eightmonospace}{FF18}
\pdfglyphtounicode{eightoldstyle}{F738}
\pdfglyphtounicode{eightparen}{247B}
\pdfglyphtounicode{eightperiod}{248F}
\pdfglyphtounicode{eightpersian}{06F8}
\pdfglyphtounicode{eightroman}{2177}
\pdfglyphtounicode{eightsuperior}{2078}
\pdfglyphtounicode{eightthai}{0E58}
\pdfglyphtounicode{einvertedbreve}{0207}
\pdfglyphtounicode{eiotifiedcyrillic}{0465}
\pdfglyphtounicode{ekatakana}{30A8}
\pdfglyphtounicode{ekatakanahalfwidth}{FF74}
\pdfglyphtounicode{ekonkargurmukhi}{0A74}
\pdfglyphtounicode{ekorean}{3154}
\pdfglyphtounicode{elcyrillic}{043B}
\pdfglyphtounicode{element}{2208}
\pdfglyphtounicode{elevencircle}{246A}
\pdfglyphtounicode{elevenparen}{247E}
\pdfglyphtounicode{elevenperiod}{2492}
\pdfglyphtounicode{elevenroman}{217A}
\pdfglyphtounicode{ellipsis}{2026}
\pdfglyphtounicode{ellipsisvertical}{22EE}
\pdfglyphtounicode{emacron}{0113}
\pdfglyphtounicode{emacronacute}{1E17}
\pdfglyphtounicode{emacrongrave}{1E15}
\pdfglyphtounicode{emcyrillic}{043C}
\pdfglyphtounicode{emdash}{2014}
\pdfglyphtounicode{emdashvertical}{FE31}
\pdfglyphtounicode{emonospace}{FF45}
\pdfglyphtounicode{emphasismarkarmenian}{055B}
\pdfglyphtounicode{emptyset}{2205}
\pdfglyphtounicode{enbopomofo}{3123}
\pdfglyphtounicode{encyrillic}{043D}
\pdfglyphtounicode{endash}{2013}
\pdfglyphtounicode{endashvertical}{FE32}
\pdfglyphtounicode{endescendercyrillic}{04A3}
\pdfglyphtounicode{eng}{014B}
\pdfglyphtounicode{engbopomofo}{3125}
\pdfglyphtounicode{enghecyrillic}{04A5}
\pdfglyphtounicode{enhookcyrillic}{04C8}
\pdfglyphtounicode{enspace}{2002}
\pdfglyphtounicode{eogonek}{0119}
\pdfglyphtounicode{eokorean}{3153}
\pdfglyphtounicode{eopen}{025B}
\pdfglyphtounicode{eopenclosed}{029A}
\pdfglyphtounicode{eopenreversed}{025C}
\pdfglyphtounicode{eopenreversedclosed}{025E}
\pdfglyphtounicode{eopenreversedhook}{025D}
\pdfglyphtounicode{eparen}{24A0}
\pdfglyphtounicode{epsilon}{03B5}
\pdfglyphtounicode{epsilontonos}{03AD}
\pdfglyphtounicode{equal}{003D}
\pdfglyphtounicode{equalmonospace}{FF1D}
\pdfglyphtounicode{equalsmall}{FE66}
\pdfglyphtounicode{equalsuperior}{207C}
\pdfglyphtounicode{equivalence}{2261}
\pdfglyphtounicode{erbopomofo}{3126}
\pdfglyphtounicode{ercyrillic}{0440}
\pdfglyphtounicode{ereversed}{0258}
\pdfglyphtounicode{ereversedcyrillic}{044D}
\pdfglyphtounicode{escyrillic}{0441}
\pdfglyphtounicode{esdescendercyrillic}{04AB}
\pdfglyphtounicode{esh}{0283}
\pdfglyphtounicode{eshcurl}{0286}
\pdfglyphtounicode{eshortdeva}{090E}
\pdfglyphtounicode{eshortvowelsigndeva}{0946}
\pdfglyphtounicode{eshreversedloop}{01AA}
\pdfglyphtounicode{eshsquatreversed}{0285}
\pdfglyphtounicode{esmallhiragana}{3047}
\pdfglyphtounicode{esmallkatakana}{30A7}
\pdfglyphtounicode{esmallkatakanahalfwidth}{FF6A}
\pdfglyphtounicode{estimated}{212E}
\pdfglyphtounicode{esuperior}{F6EC}
\pdfglyphtounicode{eta}{03B7}
\pdfglyphtounicode{etarmenian}{0568}
\pdfglyphtounicode{etatonos}{03AE}
\pdfglyphtounicode{eth}{00F0}
\pdfglyphtounicode{etilde}{1EBD}
\pdfglyphtounicode{etildebelow}{1E1B}
\pdfglyphtounicode{etnahtafoukhhebrew}{0591}
\pdfglyphtounicode{etnahtafoukhlefthebrew}{0591}
\pdfglyphtounicode{etnahtahebrew}{0591}
\pdfglyphtounicode{etnahtalefthebrew}{0591}
\pdfglyphtounicode{eturned}{01DD}
\pdfglyphtounicode{eukorean}{3161}
\pdfglyphtounicode{euro}{20AC}
\pdfglyphtounicode{evowelsignbengali}{09C7}
\pdfglyphtounicode{evowelsigndeva}{0947}
\pdfglyphtounicode{evowelsigngujarati}{0AC7}
\pdfglyphtounicode{exclam}{0021}
\pdfglyphtounicode{exclamarmenian}{055C}
\pdfglyphtounicode{exclamdbl}{203C}
\pdfglyphtounicode{exclamdown}{00A1}
\pdfglyphtounicode{exclamdownsmall}{F7A1}
\pdfglyphtounicode{exclammonospace}{FF01}
\pdfglyphtounicode{exclamsmall}{F721}
\pdfglyphtounicode{existential}{2203}
\pdfglyphtounicode{ezh}{0292}
\pdfglyphtounicode{ezhcaron}{01EF}
\pdfglyphtounicode{ezhcurl}{0293}
\pdfglyphtounicode{ezhreversed}{01B9}
\pdfglyphtounicode{ezhtail}{01BA}
\pdfglyphtounicode{f}{0066}
\pdfglyphtounicode{fadeva}{095E}
\pdfglyphtounicode{fagurmukhi}{0A5E}
\pdfglyphtounicode{fahrenheit}{2109}
\pdfglyphtounicode{fathaarabic}{064E}
\pdfglyphtounicode{fathalowarabic}{064E}
\pdfglyphtounicode{fathatanarabic}{064B}
\pdfglyphtounicode{fbopomofo}{3108}
\pdfglyphtounicode{fcircle}{24D5}
\pdfglyphtounicode{fdotaccent}{1E1F}
\pdfglyphtounicode{feharabic}{0641}
\pdfglyphtounicode{feharmenian}{0586}
\pdfglyphtounicode{fehfinalarabic}{FED2}
\pdfglyphtounicode{fehinitialarabic}{FED3}
\pdfglyphtounicode{fehmedialarabic}{FED4}
\pdfglyphtounicode{feicoptic}{03E5}
\pdfglyphtounicode{female}{2640}
\pdfglyphtounicode{ff}{FB00}
\pdfglyphtounicode{ffi}{FB03}
\pdfglyphtounicode{ffl}{FB04}
\pdfglyphtounicode{fi}{FB01}
\pdfglyphtounicode{fifteencircle}{246E}
\pdfglyphtounicode{fifteenparen}{2482}
\pdfglyphtounicode{fifteenperiod}{2496}
\pdfglyphtounicode{figuredash}{2012}
\pdfglyphtounicode{filledbox}{25A0}
\pdfglyphtounicode{filledrect}{25AC}
\pdfglyphtounicode{finalkaf}{05DA}
\pdfglyphtounicode{finalkafdagesh}{FB3A}
\pdfglyphtounicode{finalkafdageshhebrew}{FB3A}
\pdfglyphtounicode{finalkafhebrew}{05DA}
\pdfglyphtounicode{finalmem}{05DD}
\pdfglyphtounicode{finalmemhebrew}{05DD}
\pdfglyphtounicode{finalnun}{05DF}
\pdfglyphtounicode{finalnunhebrew}{05DF}
\pdfglyphtounicode{finalpe}{05E3}
\pdfglyphtounicode{finalpehebrew}{05E3}
\pdfglyphtounicode{finaltsadi}{05E5}
\pdfglyphtounicode{finaltsadihebrew}{05E5}
\pdfglyphtounicode{firsttonechinese}{02C9}
\pdfglyphtounicode{fisheye}{25C9}
\pdfglyphtounicode{fitacyrillic}{0473}
\pdfglyphtounicode{five}{0035}
\pdfglyphtounicode{fivearabic}{0665}
\pdfglyphtounicode{fivebengali}{09EB}
\pdfglyphtounicode{fivecircle}{2464}
\pdfglyphtounicode{fivecircleinversesansserif}{278E}
\pdfglyphtounicode{fivedeva}{096B}
\pdfglyphtounicode{fiveeighths}{215D}
\pdfglyphtounicode{fivegujarati}{0AEB}
\pdfglyphtounicode{fivegurmukhi}{0A6B}
\pdfglyphtounicode{fivehackarabic}{0665}
\pdfglyphtounicode{fivehangzhou}{3025}
\pdfglyphtounicode{fiveideographicparen}{3224}
\pdfglyphtounicode{fiveinferior}{2085}
\pdfglyphtounicode{fivemonospace}{FF15}
\pdfglyphtounicode{fiveoldstyle}{F735}
\pdfglyphtounicode{fiveparen}{2478}
\pdfglyphtounicode{fiveperiod}{248C}
\pdfglyphtounicode{fivepersian}{06F5}
\pdfglyphtounicode{fiveroman}{2174}
\pdfglyphtounicode{fivesuperior}{2075}
\pdfglyphtounicode{fivethai}{0E55}
\pdfglyphtounicode{fl}{FB02}
\pdfglyphtounicode{florin}{0192}
\pdfglyphtounicode{fmonospace}{FF46}
\pdfglyphtounicode{fmsquare}{3399}
\pdfglyphtounicode{fofanthai}{0E1F}
\pdfglyphtounicode{fofathai}{0E1D}
\pdfglyphtounicode{fongmanthai}{0E4F}
\pdfglyphtounicode{forall}{2200}
\pdfglyphtounicode{four}{0034}
\pdfglyphtounicode{fourarabic}{0664}
\pdfglyphtounicode{fourbengali}{09EA}
\pdfglyphtounicode{fourcircle}{2463}
\pdfglyphtounicode{fourcircleinversesansserif}{278D}
\pdfglyphtounicode{fourdeva}{096A}
\pdfglyphtounicode{fourgujarati}{0AEA}
\pdfglyphtounicode{fourgurmukhi}{0A6A}
\pdfglyphtounicode{fourhackarabic}{0664}
\pdfglyphtounicode{fourhangzhou}{3024}
\pdfglyphtounicode{fourideographicparen}{3223}
\pdfglyphtounicode{fourinferior}{2084}
\pdfglyphtounicode{fourmonospace}{FF14}
\pdfglyphtounicode{fournumeratorbengali}{09F7}
\pdfglyphtounicode{fouroldstyle}{F734}
\pdfglyphtounicode{fourparen}{2477}
\pdfglyphtounicode{fourperiod}{248B}
\pdfglyphtounicode{fourpersian}{06F4}
\pdfglyphtounicode{fourroman}{2173}
\pdfglyphtounicode{foursuperior}{2074}
\pdfglyphtounicode{fourteencircle}{246D}
\pdfglyphtounicode{fourteenparen}{2481}
\pdfglyphtounicode{fourteenperiod}{2495}
\pdfglyphtounicode{fourthai}{0E54}
\pdfglyphtounicode{fourthtonechinese}{02CB}
\pdfglyphtounicode{fparen}{24A1}
\pdfglyphtounicode{fraction}{2044}
\pdfglyphtounicode{franc}{20A3}
\pdfglyphtounicode{g}{0067}
\pdfglyphtounicode{gabengali}{0997}
\pdfglyphtounicode{gacute}{01F5}
\pdfglyphtounicode{gadeva}{0917}
\pdfglyphtounicode{gafarabic}{06AF}
\pdfglyphtounicode{gaffinalarabic}{FB93}
\pdfglyphtounicode{gafinitialarabic}{FB94}
\pdfglyphtounicode{gafmedialarabic}{FB95}
\pdfglyphtounicode{gagujarati}{0A97}
\pdfglyphtounicode{gagurmukhi}{0A17}
\pdfglyphtounicode{gahiragana}{304C}
\pdfglyphtounicode{gakatakana}{30AC}
\pdfglyphtounicode{gamma}{03B3}
\pdfglyphtounicode{gammalatinsmall}{0263}
\pdfglyphtounicode{gammasuperior}{02E0}
\pdfglyphtounicode{gangiacoptic}{03EB}
\pdfglyphtounicode{gbopomofo}{310D}
\pdfglyphtounicode{gbreve}{011F}
\pdfglyphtounicode{gcaron}{01E7}
\pdfglyphtounicode{gcedilla}{0123}
\pdfglyphtounicode{gcircle}{24D6}
\pdfglyphtounicode{gcircumflex}{011D}
\pdfglyphtounicode{gcommaaccent}{0123}
\pdfglyphtounicode{gdot}{0121}
\pdfglyphtounicode{gdotaccent}{0121}
\pdfglyphtounicode{gecyrillic}{0433}
\pdfglyphtounicode{gehiragana}{3052}
\pdfglyphtounicode{gekatakana}{30B2}
\pdfglyphtounicode{geometricallyequal}{2251}
\pdfglyphtounicode{gereshaccenthebrew}{059C}
\pdfglyphtounicode{gereshhebrew}{05F3}
\pdfglyphtounicode{gereshmuqdamhebrew}{059D}
\pdfglyphtounicode{germandbls}{00DF}
\pdfglyphtounicode{gershayimaccenthebrew}{059E}
\pdfglyphtounicode{gershayimhebrew}{05F4}
\pdfglyphtounicode{getamark}{3013}
\pdfglyphtounicode{ghabengali}{0998}
\pdfglyphtounicode{ghadarmenian}{0572}
\pdfglyphtounicode{ghadeva}{0918}
\pdfglyphtounicode{ghagujarati}{0A98}
\pdfglyphtounicode{ghagurmukhi}{0A18}
\pdfglyphtounicode{ghainarabic}{063A}
\pdfglyphtounicode{ghainfinalarabic}{FECE}
\pdfglyphtounicode{ghaininitialarabic}{FECF}
\pdfglyphtounicode{ghainmedialarabic}{FED0}
\pdfglyphtounicode{ghemiddlehookcyrillic}{0495}
\pdfglyphtounicode{ghestrokecyrillic}{0493}
\pdfglyphtounicode{gheupturncyrillic}{0491}
\pdfglyphtounicode{ghhadeva}{095A}
\pdfglyphtounicode{ghhagurmukhi}{0A5A}
\pdfglyphtounicode{ghook}{0260}
\pdfglyphtounicode{ghzsquare}{3393}
\pdfglyphtounicode{gihiragana}{304E}
\pdfglyphtounicode{gikatakana}{30AE}
\pdfglyphtounicode{gimarmenian}{0563}
\pdfglyphtounicode{gimel}{05D2}
\pdfglyphtounicode{gimeldagesh}{FB32}
\pdfglyphtounicode{gimeldageshhebrew}{FB32}
\pdfglyphtounicode{gimelhebrew}{05D2}
\pdfglyphtounicode{gjecyrillic}{0453}
\pdfglyphtounicode{glottalinvertedstroke}{01BE}
\pdfglyphtounicode{glottalstop}{0294}
\pdfglyphtounicode{glottalstopinverted}{0296}
\pdfglyphtounicode{glottalstopmod}{02C0}
\pdfglyphtounicode{glottalstopreversed}{0295}
\pdfglyphtounicode{glottalstopreversedmod}{02C1}
\pdfglyphtounicode{glottalstopreversedsuperior}{02E4}
\pdfglyphtounicode{glottalstopstroke}{02A1}
\pdfglyphtounicode{glottalstopstrokereversed}{02A2}
\pdfglyphtounicode{gmacron}{1E21}
\pdfglyphtounicode{gmonospace}{FF47}
\pdfglyphtounicode{gohiragana}{3054}
\pdfglyphtounicode{gokatakana}{30B4}
\pdfglyphtounicode{gparen}{24A2}
\pdfglyphtounicode{gpasquare}{33AC}
\pdfglyphtounicode{gradient}{2207}
\pdfglyphtounicode{grave}{0060}
\pdfglyphtounicode{gravebelowcmb}{0316}
\pdfglyphtounicode{gravecmb}{0300}
\pdfglyphtounicode{gravecomb}{0300}
\pdfglyphtounicode{gravedeva}{0953}
\pdfglyphtounicode{gravelowmod}{02CE}
\pdfglyphtounicode{gravemonospace}{FF40}
\pdfglyphtounicode{gravetonecmb}{0340}
\pdfglyphtounicode{greater}{003E}
\pdfglyphtounicode{greaterequal}{2265}
\pdfglyphtounicode{greaterequalorless}{22DB}
\pdfglyphtounicode{greatermonospace}{FF1E}
\pdfglyphtounicode{greaterorequivalent}{2273}
\pdfglyphtounicode{greaterorless}{2277}
\pdfglyphtounicode{greateroverequal}{2267}
\pdfglyphtounicode{greatersmall}{FE65}
\pdfglyphtounicode{gscript}{0261}
\pdfglyphtounicode{gstroke}{01E5}
\pdfglyphtounicode{guhiragana}{3050}
\pdfglyphtounicode{guillemotleft}{00AB}
\pdfglyphtounicode{guillemotright}{00BB}
\pdfglyphtounicode{guilsinglleft}{2039}
\pdfglyphtounicode{guilsinglright}{203A}
\pdfglyphtounicode{gukatakana}{30B0}
\pdfglyphtounicode{guramusquare}{3318}
\pdfglyphtounicode{gysquare}{33C9}
\pdfglyphtounicode{h}{0068}
\pdfglyphtounicode{haabkhasiancyrillic}{04A9}
\pdfglyphtounicode{haaltonearabic}{06C1}
\pdfglyphtounicode{habengali}{09B9}
\pdfglyphtounicode{hadescendercyrillic}{04B3}
\pdfglyphtounicode{hadeva}{0939}
\pdfglyphtounicode{hagujarati}{0AB9}
\pdfglyphtounicode{hagurmukhi}{0A39}
\pdfglyphtounicode{haharabic}{062D}
\pdfglyphtounicode{hahfinalarabic}{FEA2}
\pdfglyphtounicode{hahinitialarabic}{FEA3}
\pdfglyphtounicode{hahiragana}{306F}
\pdfglyphtounicode{hahmedialarabic}{FEA4}
\pdfglyphtounicode{haitusquare}{332A}
\pdfglyphtounicode{hakatakana}{30CF}
\pdfglyphtounicode{hakatakanahalfwidth}{FF8A}
\pdfglyphtounicode{halantgurmukhi}{0A4D}
\pdfglyphtounicode{hamzaarabic}{0621}
\pdfglyphtounicode{hamzalowarabic}{0621}
\pdfglyphtounicode{hangulfiller}{3164}
\pdfglyphtounicode{hardsigncyrillic}{044A}
\pdfglyphtounicode{harpoonleftbarbup}{21BC}
\pdfglyphtounicode{harpoonrightbarbup}{21C0}
\pdfglyphtounicode{hasquare}{33CA}
\pdfglyphtounicode{hatafpatah}{05B2}
\pdfglyphtounicode{hatafpatah16}{05B2}
\pdfglyphtounicode{hatafpatah23}{05B2}
\pdfglyphtounicode{hatafpatah2f}{05B2}
\pdfglyphtounicode{hatafpatahhebrew}{05B2}
\pdfglyphtounicode{hatafpatahnarrowhebrew}{05B2}
\pdfglyphtounicode{hatafpatahquarterhebrew}{05B2}
\pdfglyphtounicode{hatafpatahwidehebrew}{05B2}
\pdfglyphtounicode{hatafqamats}{05B3}
\pdfglyphtounicode{hatafqamats1b}{05B3}
\pdfglyphtounicode{hatafqamats28}{05B3}
\pdfglyphtounicode{hatafqamats34}{05B3}
\pdfglyphtounicode{hatafqamatshebrew}{05B3}
\pdfglyphtounicode{hatafqamatsnarrowhebrew}{05B3}
\pdfglyphtounicode{hatafqamatsquarterhebrew}{05B3}
\pdfglyphtounicode{hatafqamatswidehebrew}{05B3}
\pdfglyphtounicode{hatafsegol}{05B1}
\pdfglyphtounicode{hatafsegol17}{05B1}
\pdfglyphtounicode{hatafsegol24}{05B1}
\pdfglyphtounicode{hatafsegol30}{05B1}
\pdfglyphtounicode{hatafsegolhebrew}{05B1}
\pdfglyphtounicode{hatafsegolnarrowhebrew}{05B1}
\pdfglyphtounicode{hatafsegolquarterhebrew}{05B1}
\pdfglyphtounicode{hatafsegolwidehebrew}{05B1}
\pdfglyphtounicode{hbar}{0127}
\pdfglyphtounicode{hbopomofo}{310F}
\pdfglyphtounicode{hbrevebelow}{1E2B}
\pdfglyphtounicode{hcedilla}{1E29}
\pdfglyphtounicode{hcircle}{24D7}
\pdfglyphtounicode{hcircumflex}{0125}
\pdfglyphtounicode{hdieresis}{1E27}
\pdfglyphtounicode{hdotaccent}{1E23}
\pdfglyphtounicode{hdotbelow}{1E25}
\pdfglyphtounicode{he}{05D4}
\pdfglyphtounicode{heart}{2665}
\pdfglyphtounicode{heartsuitblack}{2665}
\pdfglyphtounicode{heartsuitwhite}{2661}
\pdfglyphtounicode{hedagesh}{FB34}
\pdfglyphtounicode{hedageshhebrew}{FB34}
\pdfglyphtounicode{hehaltonearabic}{06C1}
\pdfglyphtounicode{heharabic}{0647}
\pdfglyphtounicode{hehebrew}{05D4}
\pdfglyphtounicode{hehfinalaltonearabic}{FBA7}
\pdfglyphtounicode{hehfinalalttwoarabic}{FEEA}
\pdfglyphtounicode{hehfinalarabic}{FEEA}
\pdfglyphtounicode{hehhamzaabovefinalarabic}{FBA5}
\pdfglyphtounicode{hehhamzaaboveisolatedarabic}{FBA4}
\pdfglyphtounicode{hehinitialaltonearabic}{FBA8}
\pdfglyphtounicode{hehinitialarabic}{FEEB}
\pdfglyphtounicode{hehiragana}{3078}
\pdfglyphtounicode{hehmedialaltonearabic}{FBA9}
\pdfglyphtounicode{hehmedialarabic}{FEEC}
\pdfglyphtounicode{heiseierasquare}{337B}
\pdfglyphtounicode{hekatakana}{30D8}
\pdfglyphtounicode{hekatakanahalfwidth}{FF8D}
\pdfglyphtounicode{hekutaarusquare}{3336}
\pdfglyphtounicode{henghook}{0267}
\pdfglyphtounicode{herutusquare}{3339}
\pdfglyphtounicode{het}{05D7}
\pdfglyphtounicode{hethebrew}{05D7}
\pdfglyphtounicode{hhook}{0266}
\pdfglyphtounicode{hhooksuperior}{02B1}
\pdfglyphtounicode{hieuhacirclekorean}{327B}
\pdfglyphtounicode{hieuhaparenkorean}{321B}
\pdfglyphtounicode{hieuhcirclekorean}{326D}
\pdfglyphtounicode{hieuhkorean}{314E}
\pdfglyphtounicode{hieuhparenkorean}{320D}
\pdfglyphtounicode{hihiragana}{3072}
\pdfglyphtounicode{hikatakana}{30D2}
\pdfglyphtounicode{hikatakanahalfwidth}{FF8B}
\pdfglyphtounicode{hiriq}{05B4}
\pdfglyphtounicode{hiriq14}{05B4}
\pdfglyphtounicode{hiriq21}{05B4}
\pdfglyphtounicode{hiriq2d}{05B4}
\pdfglyphtounicode{hiriqhebrew}{05B4}
\pdfglyphtounicode{hiriqnarrowhebrew}{05B4}
\pdfglyphtounicode{hiriqquarterhebrew}{05B4}
\pdfglyphtounicode{hiriqwidehebrew}{05B4}
\pdfglyphtounicode{hlinebelow}{1E96}
\pdfglyphtounicode{hmonospace}{FF48}
\pdfglyphtounicode{hoarmenian}{0570}
\pdfglyphtounicode{hohipthai}{0E2B}
\pdfglyphtounicode{hohiragana}{307B}
\pdfglyphtounicode{hokatakana}{30DB}
\pdfglyphtounicode{hokatakanahalfwidth}{FF8E}
\pdfglyphtounicode{holam}{05B9}
\pdfglyphtounicode{holam19}{05B9}
\pdfglyphtounicode{holam26}{05B9}
\pdfglyphtounicode{holam32}{05B9}
\pdfglyphtounicode{holamhebrew}{05B9}
\pdfglyphtounicode{holamnarrowhebrew}{05B9}
\pdfglyphtounicode{holamquarterhebrew}{05B9}
\pdfglyphtounicode{holamwidehebrew}{05B9}
\pdfglyphtounicode{honokhukthai}{0E2E}
\pdfglyphtounicode{hookabovecomb}{0309}
\pdfglyphtounicode{hookcmb}{0309}
\pdfglyphtounicode{hookpalatalizedbelowcmb}{0321}
\pdfglyphtounicode{hookretroflexbelowcmb}{0322}
\pdfglyphtounicode{hoonsquare}{3342}
\pdfglyphtounicode{horicoptic}{03E9}
\pdfglyphtounicode{horizontalbar}{2015}
\pdfglyphtounicode{horncmb}{031B}
\pdfglyphtounicode{hotsprings}{2668}
\pdfglyphtounicode{house}{2302}
\pdfglyphtounicode{hparen}{24A3}
\pdfglyphtounicode{hsuperior}{02B0}
\pdfglyphtounicode{hturned}{0265}
\pdfglyphtounicode{huhiragana}{3075}
\pdfglyphtounicode{huiitosquare}{3333}
\pdfglyphtounicode{hukatakana}{30D5}
\pdfglyphtounicode{hukatakanahalfwidth}{FF8C}
\pdfglyphtounicode{hungarumlaut}{02DD}
\pdfglyphtounicode{hungarumlautcmb}{030B}
\pdfglyphtounicode{hv}{0195}
\pdfglyphtounicode{hyphen}{002D}
\pdfglyphtounicode{hypheninferior}{F6E5}
\pdfglyphtounicode{hyphenmonospace}{FF0D}
\pdfglyphtounicode{hyphensmall}{FE63}
\pdfglyphtounicode{hyphensuperior}{F6E6}
\pdfglyphtounicode{hyphentwo}{2010}
\pdfglyphtounicode{i}{0069}
\pdfglyphtounicode{iacute}{00ED}
\pdfglyphtounicode{iacyrillic}{044F}
\pdfglyphtounicode{ibengali}{0987}
\pdfglyphtounicode{ibopomofo}{3127}
\pdfglyphtounicode{ibreve}{012D}
\pdfglyphtounicode{icaron}{01D0}
\pdfglyphtounicode{icircle}{24D8}
\pdfglyphtounicode{icircumflex}{00EE}
\pdfglyphtounicode{icyrillic}{0456}
\pdfglyphtounicode{idblgrave}{0209}
\pdfglyphtounicode{ideographearthcircle}{328F}
\pdfglyphtounicode{ideographfirecircle}{328B}
\pdfglyphtounicode{ideographicallianceparen}{323F}
\pdfglyphtounicode{ideographiccallparen}{323A}
\pdfglyphtounicode{ideographiccentrecircle}{32A5}
\pdfglyphtounicode{ideographicclose}{3006}
\pdfglyphtounicode{ideographiccomma}{3001}
\pdfglyphtounicode{ideographiccommaleft}{FF64}
\pdfglyphtounicode{ideographiccongratulationparen}{3237}
\pdfglyphtounicode{ideographiccorrectcircle}{32A3}
\pdfglyphtounicode{ideographicearthparen}{322F}
\pdfglyphtounicode{ideographicenterpriseparen}{323D}
\pdfglyphtounicode{ideographicexcellentcircle}{329D}
\pdfglyphtounicode{ideographicfestivalparen}{3240}
\pdfglyphtounicode{ideographicfinancialcircle}{3296}
\pdfglyphtounicode{ideographicfinancialparen}{3236}
\pdfglyphtounicode{ideographicfireparen}{322B}
\pdfglyphtounicode{ideographichaveparen}{3232}
\pdfglyphtounicode{ideographichighcircle}{32A4}
\pdfglyphtounicode{ideographiciterationmark}{3005}
\pdfglyphtounicode{ideographiclaborcircle}{3298}
\pdfglyphtounicode{ideographiclaborparen}{3238}
\pdfglyphtounicode{ideographicleftcircle}{32A7}
\pdfglyphtounicode{ideographiclowcircle}{32A6}
\pdfglyphtounicode{ideographicmedicinecircle}{32A9}
\pdfglyphtounicode{ideographicmetalparen}{322E}
\pdfglyphtounicode{ideographicmoonparen}{322A}
\pdfglyphtounicode{ideographicnameparen}{3234}
\pdfglyphtounicode{ideographicperiod}{3002}
\pdfglyphtounicode{ideographicprintcircle}{329E}
\pdfglyphtounicode{ideographicreachparen}{3243}
\pdfglyphtounicode{ideographicrepresentparen}{3239}
\pdfglyphtounicode{ideographicresourceparen}{323E}
\pdfglyphtounicode{ideographicrightcircle}{32A8}
\pdfglyphtounicode{ideographicsecretcircle}{3299}
\pdfglyphtounicode{ideographicselfparen}{3242}
\pdfglyphtounicode{ideographicsocietyparen}{3233}
\pdfglyphtounicode{ideographicspace}{3000}
\pdfglyphtounicode{ideographicspecialparen}{3235}
\pdfglyphtounicode{ideographicstockparen}{3231}
\pdfglyphtounicode{ideographicstudyparen}{323B}
\pdfglyphtounicode{ideographicsunparen}{3230}
\pdfglyphtounicode{ideographicsuperviseparen}{323C}
\pdfglyphtounicode{ideographicwaterparen}{322C}
\pdfglyphtounicode{ideographicwoodparen}{322D}
\pdfglyphtounicode{ideographiczero}{3007}
\pdfglyphtounicode{ideographmetalcircle}{328E}
\pdfglyphtounicode{ideographmooncircle}{328A}
\pdfglyphtounicode{ideographnamecircle}{3294}
\pdfglyphtounicode{ideographsuncircle}{3290}
\pdfglyphtounicode{ideographwatercircle}{328C}
\pdfglyphtounicode{ideographwoodcircle}{328D}
\pdfglyphtounicode{ideva}{0907}
\pdfglyphtounicode{idieresis}{00EF}
\pdfglyphtounicode{idieresisacute}{1E2F}
\pdfglyphtounicode{idieresiscyrillic}{04E5}
\pdfglyphtounicode{idotbelow}{1ECB}
\pdfglyphtounicode{iebrevecyrillic}{04D7}
\pdfglyphtounicode{iecyrillic}{0435}
\pdfglyphtounicode{ieungacirclekorean}{3275}
\pdfglyphtounicode{ieungaparenkorean}{3215}
\pdfglyphtounicode{ieungcirclekorean}{3267}
\pdfglyphtounicode{ieungkorean}{3147}
\pdfglyphtounicode{ieungparenkorean}{3207}
\pdfglyphtounicode{igrave}{00EC}
\pdfglyphtounicode{igujarati}{0A87}
\pdfglyphtounicode{igurmukhi}{0A07}
\pdfglyphtounicode{ihiragana}{3044}
\pdfglyphtounicode{ihookabove}{1EC9}
\pdfglyphtounicode{iibengali}{0988}
\pdfglyphtounicode{iicyrillic}{0438}
\pdfglyphtounicode{iideva}{0908}
\pdfglyphtounicode{iigujarati}{0A88}
\pdfglyphtounicode{iigurmukhi}{0A08}
\pdfglyphtounicode{iimatragurmukhi}{0A40}
\pdfglyphtounicode{iinvertedbreve}{020B}
\pdfglyphtounicode{iishortcyrillic}{0439}
\pdfglyphtounicode{iivowelsignbengali}{09C0}
\pdfglyphtounicode{iivowelsigndeva}{0940}
\pdfglyphtounicode{iivowelsigngujarati}{0AC0}
\pdfglyphtounicode{ij}{0133}
\pdfglyphtounicode{ikatakana}{30A4}
\pdfglyphtounicode{ikatakanahalfwidth}{FF72}
\pdfglyphtounicode{ikorean}{3163}
\pdfglyphtounicode{ilde}{02DC}
\pdfglyphtounicode{iluyhebrew}{05AC}
\pdfglyphtounicode{imacron}{012B}
\pdfglyphtounicode{imacroncyrillic}{04E3}
\pdfglyphtounicode{imageorapproximatelyequal}{2253}
\pdfglyphtounicode{imatragurmukhi}{0A3F}
\pdfglyphtounicode{imonospace}{FF49}
\pdfglyphtounicode{increment}{2206}
\pdfglyphtounicode{infinity}{221E}
\pdfglyphtounicode{iniarmenian}{056B}
\pdfglyphtounicode{integral}{222B}
\pdfglyphtounicode{integralbottom}{2321}
\pdfglyphtounicode{integralbt}{2321}
\pdfglyphtounicode{integralex}{F8F5}
\pdfglyphtounicode{integraltop}{2320}
\pdfglyphtounicode{integraltp}{2320}
\pdfglyphtounicode{intersection}{2229}
\pdfglyphtounicode{intisquare}{3305}
\pdfglyphtounicode{invbullet}{25D8}
\pdfglyphtounicode{invcircle}{25D9}
\pdfglyphtounicode{invsmileface}{263B}
\pdfglyphtounicode{iocyrillic}{0451}
\pdfglyphtounicode{iogonek}{012F}
\pdfglyphtounicode{iota}{03B9}
\pdfglyphtounicode{iotadieresis}{03CA}
\pdfglyphtounicode{iotadieresistonos}{0390}
\pdfglyphtounicode{iotalatin}{0269}
\pdfglyphtounicode{iotatonos}{03AF}
\pdfglyphtounicode{iparen}{24A4}
\pdfglyphtounicode{irigurmukhi}{0A72}
\pdfglyphtounicode{ismallhiragana}{3043}
\pdfglyphtounicode{ismallkatakana}{30A3}
\pdfglyphtounicode{ismallkatakanahalfwidth}{FF68}
\pdfglyphtounicode{issharbengali}{09FA}
\pdfglyphtounicode{istroke}{0268}
\pdfglyphtounicode{isuperior}{F6ED}
\pdfglyphtounicode{iterationhiragana}{309D}
\pdfglyphtounicode{iterationkatakana}{30FD}
\pdfglyphtounicode{itilde}{0129}
\pdfglyphtounicode{itildebelow}{1E2D}
\pdfglyphtounicode{iubopomofo}{3129}
\pdfglyphtounicode{iucyrillic}{044E}
\pdfglyphtounicode{ivowelsignbengali}{09BF}
\pdfglyphtounicode{ivowelsigndeva}{093F}
\pdfglyphtounicode{ivowelsigngujarati}{0ABF}
\pdfglyphtounicode{izhitsacyrillic}{0475}
\pdfglyphtounicode{izhitsadblgravecyrillic}{0477}
\pdfglyphtounicode{j}{006A}
\pdfglyphtounicode{jaarmenian}{0571}
\pdfglyphtounicode{jabengali}{099C}
\pdfglyphtounicode{jadeva}{091C}
\pdfglyphtounicode{jagujarati}{0A9C}
\pdfglyphtounicode{jagurmukhi}{0A1C}
\pdfglyphtounicode{jbopomofo}{3110}
\pdfglyphtounicode{jcaron}{01F0}
\pdfglyphtounicode{jcircle}{24D9}
\pdfglyphtounicode{jcircumflex}{0135}
\pdfglyphtounicode{jcrossedtail}{029D}
\pdfglyphtounicode{jdotlessstroke}{025F}
\pdfglyphtounicode{jecyrillic}{0458}
\pdfglyphtounicode{jeemarabic}{062C}
\pdfglyphtounicode{jeemfinalarabic}{FE9E}
\pdfglyphtounicode{jeeminitialarabic}{FE9F}
\pdfglyphtounicode{jeemmedialarabic}{FEA0}
\pdfglyphtounicode{jeharabic}{0698}
\pdfglyphtounicode{jehfinalarabic}{FB8B}
\pdfglyphtounicode{jhabengali}{099D}
\pdfglyphtounicode{jhadeva}{091D}
\pdfglyphtounicode{jhagujarati}{0A9D}
\pdfglyphtounicode{jhagurmukhi}{0A1D}
\pdfglyphtounicode{jheharmenian}{057B}
\pdfglyphtounicode{jis}{3004}
\pdfglyphtounicode{jmonospace}{FF4A}
\pdfglyphtounicode{jparen}{24A5}
\pdfglyphtounicode{jsuperior}{02B2}
\pdfglyphtounicode{k}{006B}
\pdfglyphtounicode{kabashkircyrillic}{04A1}
\pdfglyphtounicode{kabengali}{0995}
\pdfglyphtounicode{kacute}{1E31}
\pdfglyphtounicode{kacyrillic}{043A}
\pdfglyphtounicode{kadescendercyrillic}{049B}
\pdfglyphtounicode{kadeva}{0915}
\pdfglyphtounicode{kaf}{05DB}
\pdfglyphtounicode{kafarabic}{0643}
\pdfglyphtounicode{kafdagesh}{FB3B}
\pdfglyphtounicode{kafdageshhebrew}{FB3B}
\pdfglyphtounicode{kaffinalarabic}{FEDA}
\pdfglyphtounicode{kafhebrew}{05DB}
\pdfglyphtounicode{kafinitialarabic}{FEDB}
\pdfglyphtounicode{kafmedialarabic}{FEDC}
\pdfglyphtounicode{kafrafehebrew}{FB4D}
\pdfglyphtounicode{kagujarati}{0A95}
\pdfglyphtounicode{kagurmukhi}{0A15}
\pdfglyphtounicode{kahiragana}{304B}
\pdfglyphtounicode{kahookcyrillic}{04C4}
\pdfglyphtounicode{kakatakana}{30AB}
\pdfglyphtounicode{kakatakanahalfwidth}{FF76}
\pdfglyphtounicode{kappa}{03BA}
\pdfglyphtounicode{kappasymbolgreek}{03F0}
\pdfglyphtounicode{kapyeounmieumkorean}{3171}
\pdfglyphtounicode{kapyeounphieuphkorean}{3184}
\pdfglyphtounicode{kapyeounpieupkorean}{3178}
\pdfglyphtounicode{kapyeounssangpieupkorean}{3179}
\pdfglyphtounicode{karoriisquare}{330D}
\pdfglyphtounicode{kashidaautoarabic}{0640}
\pdfglyphtounicode{kashidaautonosidebearingarabic}{0640}
\pdfglyphtounicode{kasmallkatakana}{30F5}
\pdfglyphtounicode{kasquare}{3384}
\pdfglyphtounicode{kasraarabic}{0650}
\pdfglyphtounicode{kasratanarabic}{064D}
\pdfglyphtounicode{kastrokecyrillic}{049F}
\pdfglyphtounicode{katahiraprolongmarkhalfwidth}{FF70}
\pdfglyphtounicode{kaverticalstrokecyrillic}{049D}
\pdfglyphtounicode{kbopomofo}{310E}
\pdfglyphtounicode{kcalsquare}{3389}
\pdfglyphtounicode{kcaron}{01E9}
\pdfglyphtounicode{kcedilla}{0137}
\pdfglyphtounicode{kcircle}{24DA}
\pdfglyphtounicode{kcommaaccent}{0137}
\pdfglyphtounicode{kdotbelow}{1E33}
\pdfglyphtounicode{keharmenian}{0584}
\pdfglyphtounicode{kehiragana}{3051}
\pdfglyphtounicode{kekatakana}{30B1}
\pdfglyphtounicode{kekatakanahalfwidth}{FF79}
\pdfglyphtounicode{kenarmenian}{056F}
\pdfglyphtounicode{kesmallkatakana}{30F6}
\pdfglyphtounicode{kgreenlandic}{0138}
\pdfglyphtounicode{khabengali}{0996}
\pdfglyphtounicode{khacyrillic}{0445}
\pdfglyphtounicode{khadeva}{0916}
\pdfglyphtounicode{khagujarati}{0A96}
\pdfglyphtounicode{khagurmukhi}{0A16}
\pdfglyphtounicode{khaharabic}{062E}
\pdfglyphtounicode{khahfinalarabic}{FEA6}
\pdfglyphtounicode{khahinitialarabic}{FEA7}
\pdfglyphtounicode{khahmedialarabic}{FEA8}
\pdfglyphtounicode{kheicoptic}{03E7}
\pdfglyphtounicode{khhadeva}{0959}
\pdfglyphtounicode{khhagurmukhi}{0A59}
\pdfglyphtounicode{khieukhacirclekorean}{3278}
\pdfglyphtounicode{khieukhaparenkorean}{3218}
\pdfglyphtounicode{khieukhcirclekorean}{326A}
\pdfglyphtounicode{khieukhkorean}{314B}
\pdfglyphtounicode{khieukhparenkorean}{320A}
\pdfglyphtounicode{khokhaithai}{0E02}
\pdfglyphtounicode{khokhonthai}{0E05}
\pdfglyphtounicode{khokhuatthai}{0E03}
\pdfglyphtounicode{khokhwaithai}{0E04}
\pdfglyphtounicode{khomutthai}{0E5B}
\pdfglyphtounicode{khook}{0199}
\pdfglyphtounicode{khorakhangthai}{0E06}
\pdfglyphtounicode{khzsquare}{3391}
\pdfglyphtounicode{kihiragana}{304D}
\pdfglyphtounicode{kikatakana}{30AD}
\pdfglyphtounicode{kikatakanahalfwidth}{FF77}
\pdfglyphtounicode{kiroguramusquare}{3315}
\pdfglyphtounicode{kiromeetorusquare}{3316}
\pdfglyphtounicode{kirosquare}{3314}
\pdfglyphtounicode{kiyeokacirclekorean}{326E}
\pdfglyphtounicode{kiyeokaparenkorean}{320E}
\pdfglyphtounicode{kiyeokcirclekorean}{3260}
\pdfglyphtounicode{kiyeokkorean}{3131}
\pdfglyphtounicode{kiyeokparenkorean}{3200}
\pdfglyphtounicode{kiyeoksioskorean}{3133}
\pdfglyphtounicode{kjecyrillic}{045C}
\pdfglyphtounicode{klinebelow}{1E35}
\pdfglyphtounicode{klsquare}{3398}
\pdfglyphtounicode{kmcubedsquare}{33A6}
\pdfglyphtounicode{kmonospace}{FF4B}
\pdfglyphtounicode{kmsquaredsquare}{33A2}
\pdfglyphtounicode{kohiragana}{3053}
\pdfglyphtounicode{kohmsquare}{33C0}
\pdfglyphtounicode{kokaithai}{0E01}
\pdfglyphtounicode{kokatakana}{30B3}
\pdfglyphtounicode{kokatakanahalfwidth}{FF7A}
\pdfglyphtounicode{kooposquare}{331E}
\pdfglyphtounicode{koppacyrillic}{0481}
\pdfglyphtounicode{koreanstandardsymbol}{327F}
\pdfglyphtounicode{koroniscmb}{0343}
\pdfglyphtounicode{kparen}{24A6}
\pdfglyphtounicode{kpasquare}{33AA}
\pdfglyphtounicode{ksicyrillic}{046F}
\pdfglyphtounicode{ktsquare}{33CF}
\pdfglyphtounicode{kturned}{029E}
\pdfglyphtounicode{kuhiragana}{304F}
\pdfglyphtounicode{kukatakana}{30AF}
\pdfglyphtounicode{kukatakanahalfwidth}{FF78}
\pdfglyphtounicode{kvsquare}{33B8}
\pdfglyphtounicode{kwsquare}{33BE}
\pdfglyphtounicode{l}{006C}
\pdfglyphtounicode{labengali}{09B2}
\pdfglyphtounicode{lacute}{013A}
\pdfglyphtounicode{ladeva}{0932}
\pdfglyphtounicode{lagujarati}{0AB2}
\pdfglyphtounicode{lagurmukhi}{0A32}
\pdfglyphtounicode{lakkhangyaothai}{0E45}
\pdfglyphtounicode{lamaleffinalarabic}{FEFC}
\pdfglyphtounicode{lamalefhamzaabovefinalarabic}{FEF8}
\pdfglyphtounicode{lamalefhamzaaboveisolatedarabic}{FEF7}
\pdfglyphtounicode{lamalefhamzabelowfinalarabic}{FEFA}
\pdfglyphtounicode{lamalefhamzabelowisolatedarabic}{FEF9}
\pdfglyphtounicode{lamalefisolatedarabic}{FEFB}
\pdfglyphtounicode{lamalefmaddaabovefinalarabic}{FEF6}
\pdfglyphtounicode{lamalefmaddaaboveisolatedarabic}{FEF5}
\pdfglyphtounicode{lamarabic}{0644}
\pdfglyphtounicode{lambda}{03BB}
\pdfglyphtounicode{lambdastroke}{019B}
\pdfglyphtounicode{lamed}{05DC}
\pdfglyphtounicode{lameddagesh}{FB3C}
\pdfglyphtounicode{lameddageshhebrew}{FB3C}
\pdfglyphtounicode{lamedhebrew}{05DC}
\pdfglyphtounicode{lamfinalarabic}{FEDE}
\pdfglyphtounicode{lamhahinitialarabic}{FCCA}
\pdfglyphtounicode{laminitialarabic}{FEDF}
\pdfglyphtounicode{lamjeeminitialarabic}{FCC9}
\pdfglyphtounicode{lamkhahinitialarabic}{FCCB}
\pdfglyphtounicode{lamlamhehisolatedarabic}{FDF2}
\pdfglyphtounicode{lammedialarabic}{FEE0}
\pdfglyphtounicode{lammeemhahinitialarabic}{FD88}
\pdfglyphtounicode{lammeeminitialarabic}{FCCC}
\pdfglyphtounicode{largecircle}{25EF}
\pdfglyphtounicode{lbar}{019A}
\pdfglyphtounicode{lbelt}{026C}
\pdfglyphtounicode{lbopomofo}{310C}
\pdfglyphtounicode{lcaron}{013E}
\pdfglyphtounicode{lcedilla}{013C}
\pdfglyphtounicode{lcircle}{24DB}
\pdfglyphtounicode{lcircumflexbelow}{1E3D}
\pdfglyphtounicode{lcommaaccent}{013C}
\pdfglyphtounicode{ldot}{0140}
\pdfglyphtounicode{ldotaccent}{0140}
\pdfglyphtounicode{ldotbelow}{1E37}
\pdfglyphtounicode{ldotbelowmacron}{1E39}
\pdfglyphtounicode{leftangleabovecmb}{031A}
\pdfglyphtounicode{lefttackbelowcmb}{0318}
\pdfglyphtounicode{less}{003C}
\pdfglyphtounicode{lessequal}{2264}
\pdfglyphtounicode{lessequalorgreater}{22DA}
\pdfglyphtounicode{lessmonospace}{FF1C}
\pdfglyphtounicode{lessorequivalent}{2272}
\pdfglyphtounicode{lessorgreater}{2276}
\pdfglyphtounicode{lessoverequal}{2266}
\pdfglyphtounicode{lesssmall}{FE64}
\pdfglyphtounicode{lezh}{026E}
\pdfglyphtounicode{lfblock}{258C}
\pdfglyphtounicode{lhookretroflex}{026D}
\pdfglyphtounicode{lira}{20A4}
\pdfglyphtounicode{liwnarmenian}{056C}
\pdfglyphtounicode{lj}{01C9}
\pdfglyphtounicode{ljecyrillic}{0459}
\pdfglyphtounicode{ll}{F6C0}
\pdfglyphtounicode{lladeva}{0933}
\pdfglyphtounicode{llagujarati}{0AB3}
\pdfglyphtounicode{llinebelow}{1E3B}
\pdfglyphtounicode{llladeva}{0934}
\pdfglyphtounicode{llvocalicbengali}{09E1}
\pdfglyphtounicode{llvocalicdeva}{0961}
\pdfglyphtounicode{llvocalicvowelsignbengali}{09E3}
\pdfglyphtounicode{llvocalicvowelsigndeva}{0963}
\pdfglyphtounicode{lmiddletilde}{026B}
\pdfglyphtounicode{lmonospace}{FF4C}
\pdfglyphtounicode{lmsquare}{33D0}
\pdfglyphtounicode{lochulathai}{0E2C}
\pdfglyphtounicode{logicaland}{2227}
\pdfglyphtounicode{logicalnot}{00AC}
\pdfglyphtounicode{logicalnotreversed}{2310}
\pdfglyphtounicode{logicalor}{2228}
\pdfglyphtounicode{lolingthai}{0E25}
\pdfglyphtounicode{longs}{017F}
\pdfglyphtounicode{lowlinecenterline}{FE4E}
\pdfglyphtounicode{lowlinecmb}{0332}
\pdfglyphtounicode{lowlinedashed}{FE4D}
\pdfglyphtounicode{lozenge}{25CA}
\pdfglyphtounicode{lparen}{24A7}
\pdfglyphtounicode{lslash}{0142}
\pdfglyphtounicode{lsquare}{2113}
\pdfglyphtounicode{lsuperior}{F6EE}
\pdfglyphtounicode{ltshade}{2591}
\pdfglyphtounicode{luthai}{0E26}
\pdfglyphtounicode{lvocalicbengali}{098C}
\pdfglyphtounicode{lvocalicdeva}{090C}
\pdfglyphtounicode{lvocalicvowelsignbengali}{09E2}
\pdfglyphtounicode{lvocalicvowelsigndeva}{0962}
\pdfglyphtounicode{lxsquare}{33D3}
\pdfglyphtounicode{m}{006D}
\pdfglyphtounicode{mabengali}{09AE}
\pdfglyphtounicode{macron}{00AF}
\pdfglyphtounicode{macronbelowcmb}{0331}
\pdfglyphtounicode{macroncmb}{0304}
\pdfglyphtounicode{macronlowmod}{02CD}
\pdfglyphtounicode{macronmonospace}{FFE3}
\pdfglyphtounicode{macute}{1E3F}
\pdfglyphtounicode{madeva}{092E}
\pdfglyphtounicode{magujarati}{0AAE}
\pdfglyphtounicode{magurmukhi}{0A2E}
\pdfglyphtounicode{mahapakhhebrew}{05A4}
\pdfglyphtounicode{mahapakhlefthebrew}{05A4}
\pdfglyphtounicode{mahiragana}{307E}
\pdfglyphtounicode{maichattawalowleftthai}{F895}
\pdfglyphtounicode{maichattawalowrightthai}{F894}
\pdfglyphtounicode{maichattawathai}{0E4B}
\pdfglyphtounicode{maichattawaupperleftthai}{F893}
\pdfglyphtounicode{maieklowleftthai}{F88C}
\pdfglyphtounicode{maieklowrightthai}{F88B}
\pdfglyphtounicode{maiekthai}{0E48}
\pdfglyphtounicode{maiekupperleftthai}{F88A}
\pdfglyphtounicode{maihanakatleftthai}{F884}
\pdfglyphtounicode{maihanakatthai}{0E31}
\pdfglyphtounicode{maitaikhuleftthai}{F889}
\pdfglyphtounicode{maitaikhuthai}{0E47}
\pdfglyphtounicode{maitholowleftthai}{F88F}
\pdfglyphtounicode{maitholowrightthai}{F88E}
\pdfglyphtounicode{maithothai}{0E49}
\pdfglyphtounicode{maithoupperleftthai}{F88D}
\pdfglyphtounicode{maitrilowleftthai}{F892}
\pdfglyphtounicode{maitrilowrightthai}{F891}
\pdfglyphtounicode{maitrithai}{0E4A}
\pdfglyphtounicode{maitriupperleftthai}{F890}
\pdfglyphtounicode{maiyamokthai}{0E46}
\pdfglyphtounicode{makatakana}{30DE}
\pdfglyphtounicode{makatakanahalfwidth}{FF8F}
\pdfglyphtounicode{male}{2642}
\pdfglyphtounicode{mansyonsquare}{3347}
\pdfglyphtounicode{maqafhebrew}{05BE}
\pdfglyphtounicode{mars}{2642}
\pdfglyphtounicode{masoracirclehebrew}{05AF}
\pdfglyphtounicode{masquare}{3383}
\pdfglyphtounicode{mbopomofo}{3107}
\pdfglyphtounicode{mbsquare}{33D4}
\pdfglyphtounicode{mcircle}{24DC}
\pdfglyphtounicode{mcubedsquare}{33A5}
\pdfglyphtounicode{mdotaccent}{1E41}
\pdfglyphtounicode{mdotbelow}{1E43}
\pdfglyphtounicode{meemarabic}{0645}
\pdfglyphtounicode{meemfinalarabic}{FEE2}
\pdfglyphtounicode{meeminitialarabic}{FEE3}
\pdfglyphtounicode{meemmedialarabic}{FEE4}
\pdfglyphtounicode{meemmeeminitialarabic}{FCD1}
\pdfglyphtounicode{meemmeemisolatedarabic}{FC48}
\pdfglyphtounicode{meetorusquare}{334D}
\pdfglyphtounicode{mehiragana}{3081}
\pdfglyphtounicode{meizierasquare}{337E}
\pdfglyphtounicode{mekatakana}{30E1}
\pdfglyphtounicode{mekatakanahalfwidth}{FF92}
\pdfglyphtounicode{mem}{05DE}
\pdfglyphtounicode{memdagesh}{FB3E}
\pdfglyphtounicode{memdageshhebrew}{FB3E}
\pdfglyphtounicode{memhebrew}{05DE}
\pdfglyphtounicode{menarmenian}{0574}
\pdfglyphtounicode{merkhahebrew}{05A5}
\pdfglyphtounicode{merkhakefulahebrew}{05A6}
\pdfglyphtounicode{merkhakefulalefthebrew}{05A6}
\pdfglyphtounicode{merkhalefthebrew}{05A5}
\pdfglyphtounicode{mhook}{0271}
\pdfglyphtounicode{mhzsquare}{3392}
\pdfglyphtounicode{middledotkatakanahalfwidth}{FF65}
\pdfglyphtounicode{middot}{00B7}
\pdfglyphtounicode{mieumacirclekorean}{3272}
\pdfglyphtounicode{mieumaparenkorean}{3212}
\pdfglyphtounicode{mieumcirclekorean}{3264}
\pdfglyphtounicode{mieumkorean}{3141}
\pdfglyphtounicode{mieumpansioskorean}{3170}
\pdfglyphtounicode{mieumparenkorean}{3204}
\pdfglyphtounicode{mieumpieupkorean}{316E}
\pdfglyphtounicode{mieumsioskorean}{316F}
\pdfglyphtounicode{mihiragana}{307F}
\pdfglyphtounicode{mikatakana}{30DF}
\pdfglyphtounicode{mikatakanahalfwidth}{FF90}
\pdfglyphtounicode{minus}{2212}
\pdfglyphtounicode{minusbelowcmb}{0320}
\pdfglyphtounicode{minuscircle}{2296}
\pdfglyphtounicode{minusmod}{02D7}
\pdfglyphtounicode{minusplus}{2213}
\pdfglyphtounicode{minute}{2032}
\pdfglyphtounicode{miribaarusquare}{334A}
\pdfglyphtounicode{mirisquare}{3349}
\pdfglyphtounicode{mlonglegturned}{0270}
\pdfglyphtounicode{mlsquare}{3396}
\pdfglyphtounicode{mmcubedsquare}{33A3}
\pdfglyphtounicode{mmonospace}{FF4D}
\pdfglyphtounicode{mmsquaredsquare}{339F}
\pdfglyphtounicode{mohiragana}{3082}
\pdfglyphtounicode{mohmsquare}{33C1}
\pdfglyphtounicode{mokatakana}{30E2}
\pdfglyphtounicode{mokatakanahalfwidth}{FF93}
\pdfglyphtounicode{molsquare}{33D6}
\pdfglyphtounicode{momathai}{0E21}
\pdfglyphtounicode{moverssquare}{33A7}
\pdfglyphtounicode{moverssquaredsquare}{33A8}
\pdfglyphtounicode{mparen}{24A8}
\pdfglyphtounicode{mpasquare}{33AB}
\pdfglyphtounicode{mssquare}{33B3}
\pdfglyphtounicode{msuperior}{F6EF}
\pdfglyphtounicode{mturned}{026F}
\pdfglyphtounicode{mu}{00B5}
\pdfglyphtounicode{mu1}{00B5}
\pdfglyphtounicode{muasquare}{3382}
\pdfglyphtounicode{muchgreater}{226B}
\pdfglyphtounicode{muchless}{226A}
\pdfglyphtounicode{mufsquare}{338C}
\pdfglyphtounicode{mugreek}{03BC}
\pdfglyphtounicode{mugsquare}{338D}
\pdfglyphtounicode{muhiragana}{3080}
\pdfglyphtounicode{mukatakana}{30E0}
\pdfglyphtounicode{mukatakanahalfwidth}{FF91}
\pdfglyphtounicode{mulsquare}{3395}
\pdfglyphtounicode{multiply}{00D7}
\pdfglyphtounicode{mumsquare}{339B}
\pdfglyphtounicode{munahhebrew}{05A3}
\pdfglyphtounicode{munahlefthebrew}{05A3}
\pdfglyphtounicode{musicalnote}{266A}
\pdfglyphtounicode{musicalnotedbl}{266B}
\pdfglyphtounicode{musicflatsign}{266D}
\pdfglyphtounicode{musicsharpsign}{266F}
\pdfglyphtounicode{mussquare}{33B2}
\pdfglyphtounicode{muvsquare}{33B6}
\pdfglyphtounicode{muwsquare}{33BC}
\pdfglyphtounicode{mvmegasquare}{33B9}
\pdfglyphtounicode{mvsquare}{33B7}
\pdfglyphtounicode{mwmegasquare}{33BF}
\pdfglyphtounicode{mwsquare}{33BD}
\pdfglyphtounicode{n}{006E}
\pdfglyphtounicode{nabengali}{09A8}
\pdfglyphtounicode{nabla}{2207}
\pdfglyphtounicode{nacute}{0144}
\pdfglyphtounicode{nadeva}{0928}
\pdfglyphtounicode{nagujarati}{0AA8}
\pdfglyphtounicode{nagurmukhi}{0A28}
\pdfglyphtounicode{nahiragana}{306A}
\pdfglyphtounicode{nakatakana}{30CA}
\pdfglyphtounicode{nakatakanahalfwidth}{FF85}
\pdfglyphtounicode{napostrophe}{0149}
\pdfglyphtounicode{nasquare}{3381}
\pdfglyphtounicode{nbopomofo}{310B}
\pdfglyphtounicode{nbspace}{00A0}
\pdfglyphtounicode{ncaron}{0148}
\pdfglyphtounicode{ncedilla}{0146}
\pdfglyphtounicode{ncircle}{24DD}
\pdfglyphtounicode{ncircumflexbelow}{1E4B}
\pdfglyphtounicode{ncommaaccent}{0146}
\pdfglyphtounicode{ndotaccent}{1E45}
\pdfglyphtounicode{ndotbelow}{1E47}
\pdfglyphtounicode{nehiragana}{306D}
\pdfglyphtounicode{nekatakana}{30CD}
\pdfglyphtounicode{nekatakanahalfwidth}{FF88}
\pdfglyphtounicode{newsheqelsign}{20AA}
\pdfglyphtounicode{nfsquare}{338B}
\pdfglyphtounicode{ngabengali}{0999}
\pdfglyphtounicode{ngadeva}{0919}
\pdfglyphtounicode{ngagujarati}{0A99}
\pdfglyphtounicode{ngagurmukhi}{0A19}
\pdfglyphtounicode{ngonguthai}{0E07}
\pdfglyphtounicode{nhiragana}{3093}
\pdfglyphtounicode{nhookleft}{0272}
\pdfglyphtounicode{nhookretroflex}{0273}
\pdfglyphtounicode{nieunacirclekorean}{326F}
\pdfglyphtounicode{nieunaparenkorean}{320F}
\pdfglyphtounicode{nieuncieuckorean}{3135}
\pdfglyphtounicode{nieuncirclekorean}{3261}
\pdfglyphtounicode{nieunhieuhkorean}{3136}
\pdfglyphtounicode{nieunkorean}{3134}
\pdfglyphtounicode{nieunpansioskorean}{3168}
\pdfglyphtounicode{nieunparenkorean}{3201}
\pdfglyphtounicode{nieunsioskorean}{3167}
\pdfglyphtounicode{nieuntikeutkorean}{3166}
\pdfglyphtounicode{nihiragana}{306B}
\pdfglyphtounicode{nikatakana}{30CB}
\pdfglyphtounicode{nikatakanahalfwidth}{FF86}
\pdfglyphtounicode{nikhahitleftthai}{F899}
\pdfglyphtounicode{nikhahitthai}{0E4D}
\pdfglyphtounicode{nine}{0039}
\pdfglyphtounicode{ninearabic}{0669}
\pdfglyphtounicode{ninebengali}{09EF}
\pdfglyphtounicode{ninecircle}{2468}
\pdfglyphtounicode{ninecircleinversesansserif}{2792}
\pdfglyphtounicode{ninedeva}{096F}
\pdfglyphtounicode{ninegujarati}{0AEF}
\pdfglyphtounicode{ninegurmukhi}{0A6F}
\pdfglyphtounicode{ninehackarabic}{0669}
\pdfglyphtounicode{ninehangzhou}{3029}
\pdfglyphtounicode{nineideographicparen}{3228}
\pdfglyphtounicode{nineinferior}{2089}
\pdfglyphtounicode{ninemonospace}{FF19}
\pdfglyphtounicode{nineoldstyle}{F739}
\pdfglyphtounicode{nineparen}{247C}
\pdfglyphtounicode{nineperiod}{2490}
\pdfglyphtounicode{ninepersian}{06F9}
\pdfglyphtounicode{nineroman}{2178}
\pdfglyphtounicode{ninesuperior}{2079}
\pdfglyphtounicode{nineteencircle}{2472}
\pdfglyphtounicode{nineteenparen}{2486}
\pdfglyphtounicode{nineteenperiod}{249A}
\pdfglyphtounicode{ninethai}{0E59}
\pdfglyphtounicode{nj}{01CC}
\pdfglyphtounicode{njecyrillic}{045A}
\pdfglyphtounicode{nkatakana}{30F3}
\pdfglyphtounicode{nkatakanahalfwidth}{FF9D}
\pdfglyphtounicode{nlegrightlong}{019E}
\pdfglyphtounicode{nlinebelow}{1E49}
\pdfglyphtounicode{nmonospace}{FF4E}
\pdfglyphtounicode{nmsquare}{339A}
\pdfglyphtounicode{nnabengali}{09A3}
\pdfglyphtounicode{nnadeva}{0923}
\pdfglyphtounicode{nnagujarati}{0AA3}
\pdfglyphtounicode{nnagurmukhi}{0A23}
\pdfglyphtounicode{nnnadeva}{0929}
\pdfglyphtounicode{nohiragana}{306E}
\pdfglyphtounicode{nokatakana}{30CE}
\pdfglyphtounicode{nokatakanahalfwidth}{FF89}
\pdfglyphtounicode{nonbreakingspace}{00A0}
\pdfglyphtounicode{nonenthai}{0E13}
\pdfglyphtounicode{nonuthai}{0E19}
\pdfglyphtounicode{noonarabic}{0646}
\pdfglyphtounicode{noonfinalarabic}{FEE6}
\pdfglyphtounicode{noonghunnaarabic}{06BA}
\pdfglyphtounicode{noonghunnafinalarabic}{FB9F}
\pdfglyphtounicode{nooninitialarabic}{FEE7}
\pdfglyphtounicode{noonjeeminitialarabic}{FCD2}
\pdfglyphtounicode{noonjeemisolatedarabic}{FC4B}
\pdfglyphtounicode{noonmedialarabic}{FEE8}
\pdfglyphtounicode{noonmeeminitialarabic}{FCD5}
\pdfglyphtounicode{noonmeemisolatedarabic}{FC4E}
\pdfglyphtounicode{noonnoonfinalarabic}{FC8D}
\pdfglyphtounicode{notcontains}{220C}
\pdfglyphtounicode{notelement}{2209}
\pdfglyphtounicode{notelementof}{2209}
\pdfglyphtounicode{notequal}{2260}
\pdfglyphtounicode{notgreater}{226F}
\pdfglyphtounicode{notgreaternorequal}{2271}
\pdfglyphtounicode{notgreaternorless}{2279}
\pdfglyphtounicode{notidentical}{2262}
\pdfglyphtounicode{notless}{226E}
\pdfglyphtounicode{notlessnorequal}{2270}
\pdfglyphtounicode{notparallel}{2226}
\pdfglyphtounicode{notprecedes}{2280}
\pdfglyphtounicode{notsubset}{2284}
\pdfglyphtounicode{notsucceeds}{2281}
\pdfglyphtounicode{notsuperset}{2285}
\pdfglyphtounicode{nowarmenian}{0576}
\pdfglyphtounicode{nparen}{24A9}
\pdfglyphtounicode{nssquare}{33B1}
\pdfglyphtounicode{nsuperior}{207F}
\pdfglyphtounicode{ntilde}{00F1}
\pdfglyphtounicode{nu}{03BD}
\pdfglyphtounicode{nuhiragana}{306C}
\pdfglyphtounicode{nukatakana}{30CC}
\pdfglyphtounicode{nukatakanahalfwidth}{FF87}
\pdfglyphtounicode{nuktabengali}{09BC}
\pdfglyphtounicode{nuktadeva}{093C}
\pdfglyphtounicode{nuktagujarati}{0ABC}
\pdfglyphtounicode{nuktagurmukhi}{0A3C}
\pdfglyphtounicode{numbersign}{0023}
\pdfglyphtounicode{numbersignmonospace}{FF03}
\pdfglyphtounicode{numbersignsmall}{FE5F}
\pdfglyphtounicode{numeralsigngreek}{0374}
\pdfglyphtounicode{numeralsignlowergreek}{0375}
\pdfglyphtounicode{numero}{2116}
\pdfglyphtounicode{nun}{05E0}
\pdfglyphtounicode{nundagesh}{FB40}
\pdfglyphtounicode{nundageshhebrew}{FB40}
\pdfglyphtounicode{nunhebrew}{05E0}
\pdfglyphtounicode{nvsquare}{33B5}
\pdfglyphtounicode{nwsquare}{33BB}
\pdfglyphtounicode{nyabengali}{099E}
\pdfglyphtounicode{nyadeva}{091E}
\pdfglyphtounicode{nyagujarati}{0A9E}
\pdfglyphtounicode{nyagurmukhi}{0A1E}
\pdfglyphtounicode{o}{006F}
\pdfglyphtounicode{oacute}{00F3}
\pdfglyphtounicode{oangthai}{0E2D}
\pdfglyphtounicode{obarred}{0275}
\pdfglyphtounicode{obarredcyrillic}{04E9}
\pdfglyphtounicode{obarreddieresiscyrillic}{04EB}
\pdfglyphtounicode{obengali}{0993}
\pdfglyphtounicode{obopomofo}{311B}
\pdfglyphtounicode{obreve}{014F}
\pdfglyphtounicode{ocandradeva}{0911}
\pdfglyphtounicode{ocandragujarati}{0A91}
\pdfglyphtounicode{ocandravowelsigndeva}{0949}
\pdfglyphtounicode{ocandravowelsigngujarati}{0AC9}
\pdfglyphtounicode{ocaron}{01D2}
\pdfglyphtounicode{ocircle}{24DE}
\pdfglyphtounicode{ocircumflex}{00F4}
\pdfglyphtounicode{ocircumflexacute}{1ED1}
\pdfglyphtounicode{ocircumflexdotbelow}{1ED9}
\pdfglyphtounicode{ocircumflexgrave}{1ED3}
\pdfglyphtounicode{ocircumflexhookabove}{1ED5}
\pdfglyphtounicode{ocircumflextilde}{1ED7}
\pdfglyphtounicode{ocyrillic}{043E}
\pdfglyphtounicode{odblacute}{0151}
\pdfglyphtounicode{odblgrave}{020D}
\pdfglyphtounicode{odeva}{0913}
\pdfglyphtounicode{odieresis}{00F6}
\pdfglyphtounicode{odieresiscyrillic}{04E7}
\pdfglyphtounicode{odotbelow}{1ECD}
\pdfglyphtounicode{oe}{0153}
\pdfglyphtounicode{oekorean}{315A}
\pdfglyphtounicode{ogonek}{02DB}
\pdfglyphtounicode{ogonekcmb}{0328}
\pdfglyphtounicode{ograve}{00F2}
\pdfglyphtounicode{ogujarati}{0A93}
\pdfglyphtounicode{oharmenian}{0585}
\pdfglyphtounicode{ohiragana}{304A}
\pdfglyphtounicode{ohookabove}{1ECF}
\pdfglyphtounicode{ohorn}{01A1}
\pdfglyphtounicode{ohornacute}{1EDB}
\pdfglyphtounicode{ohorndotbelow}{1EE3}
\pdfglyphtounicode{ohorngrave}{1EDD}
\pdfglyphtounicode{ohornhookabove}{1EDF}
\pdfglyphtounicode{ohorntilde}{1EE1}
\pdfglyphtounicode{ohungarumlaut}{0151}
\pdfglyphtounicode{oi}{01A3}
\pdfglyphtounicode{oinvertedbreve}{020F}
\pdfglyphtounicode{okatakana}{30AA}
\pdfglyphtounicode{okatakanahalfwidth}{FF75}
\pdfglyphtounicode{okorean}{3157}
\pdfglyphtounicode{olehebrew}{05AB}
\pdfglyphtounicode{omacron}{014D}
\pdfglyphtounicode{omacronacute}{1E53}
\pdfglyphtounicode{omacrongrave}{1E51}
\pdfglyphtounicode{omdeva}{0950}
\pdfglyphtounicode{omega}{03C9}
\pdfglyphtounicode{omega1}{03D6}
\pdfglyphtounicode{omegacyrillic}{0461}
\pdfglyphtounicode{omegalatinclosed}{0277}
\pdfglyphtounicode{omegaroundcyrillic}{047B}
\pdfglyphtounicode{omegatitlocyrillic}{047D}
\pdfglyphtounicode{omegatonos}{03CE}
\pdfglyphtounicode{omgujarati}{0AD0}
\pdfglyphtounicode{omicron}{03BF}
\pdfglyphtounicode{omicrontonos}{03CC}
\pdfglyphtounicode{omonospace}{FF4F}
\pdfglyphtounicode{one}{0031}
\pdfglyphtounicode{onearabic}{0661}
\pdfglyphtounicode{onebengali}{09E7}
\pdfglyphtounicode{onecircle}{2460}
\pdfglyphtounicode{onecircleinversesansserif}{278A}
\pdfglyphtounicode{onedeva}{0967}
\pdfglyphtounicode{onedotenleader}{2024}
\pdfglyphtounicode{oneeighth}{215B}
\pdfglyphtounicode{onefitted}{F6DC}
\pdfglyphtounicode{onegujarati}{0AE7}
\pdfglyphtounicode{onegurmukhi}{0A67}
\pdfglyphtounicode{onehackarabic}{0661}
\pdfglyphtounicode{onehalf}{00BD}
\pdfglyphtounicode{onehangzhou}{3021}
\pdfglyphtounicode{oneideographicparen}{3220}
\pdfglyphtounicode{oneinferior}{2081}
\pdfglyphtounicode{onemonospace}{FF11}
\pdfglyphtounicode{onenumeratorbengali}{09F4}
\pdfglyphtounicode{oneoldstyle}{F731}
\pdfglyphtounicode{oneparen}{2474}
\pdfglyphtounicode{oneperiod}{2488}
\pdfglyphtounicode{onepersian}{06F1}
\pdfglyphtounicode{onequarter}{00BC}
\pdfglyphtounicode{oneroman}{2170}
\pdfglyphtounicode{onesuperior}{00B9}
\pdfglyphtounicode{onethai}{0E51}
\pdfglyphtounicode{onethird}{2153}
\pdfglyphtounicode{oogonek}{01EB}
\pdfglyphtounicode{oogonekmacron}{01ED}
\pdfglyphtounicode{oogurmukhi}{0A13}
\pdfglyphtounicode{oomatragurmukhi}{0A4B}
\pdfglyphtounicode{oopen}{0254}
\pdfglyphtounicode{oparen}{24AA}
\pdfglyphtounicode{openbullet}{25E6}
\pdfglyphtounicode{option}{2325}
\pdfglyphtounicode{ordfeminine}{00AA}
\pdfglyphtounicode{ordmasculine}{00BA}
\pdfglyphtounicode{orthogonal}{221F}
\pdfglyphtounicode{oshortdeva}{0912}
\pdfglyphtounicode{oshortvowelsigndeva}{094A}
\pdfglyphtounicode{oslash}{00F8}
\pdfglyphtounicode{oslashacute}{01FF}
\pdfglyphtounicode{osmallhiragana}{3049}
\pdfglyphtounicode{osmallkatakana}{30A9}
\pdfglyphtounicode{osmallkatakanahalfwidth}{FF6B}
\pdfglyphtounicode{ostrokeacute}{01FF}
\pdfglyphtounicode{osuperior}{F6F0}
\pdfglyphtounicode{otcyrillic}{047F}
\pdfglyphtounicode{otilde}{00F5}
\pdfglyphtounicode{otildeacute}{1E4D}
\pdfglyphtounicode{otildedieresis}{1E4F}
\pdfglyphtounicode{oubopomofo}{3121}
\pdfglyphtounicode{overline}{203E}
\pdfglyphtounicode{overlinecenterline}{FE4A}
\pdfglyphtounicode{overlinecmb}{0305}
\pdfglyphtounicode{overlinedashed}{FE49}
\pdfglyphtounicode{overlinedblwavy}{FE4C}
\pdfglyphtounicode{overlinewavy}{FE4B}
\pdfglyphtounicode{overscore}{00AF}
\pdfglyphtounicode{ovowelsignbengali}{09CB}
\pdfglyphtounicode{ovowelsigndeva}{094B}
\pdfglyphtounicode{ovowelsigngujarati}{0ACB}
\pdfglyphtounicode{p}{0070}
\pdfglyphtounicode{paampssquare}{3380}
\pdfglyphtounicode{paasentosquare}{332B}
\pdfglyphtounicode{pabengali}{09AA}
\pdfglyphtounicode{pacute}{1E55}
\pdfglyphtounicode{padeva}{092A}
\pdfglyphtounicode{pagedown}{21DF}
\pdfglyphtounicode{pageup}{21DE}
\pdfglyphtounicode{pagujarati}{0AAA}
\pdfglyphtounicode{pagurmukhi}{0A2A}
\pdfglyphtounicode{pahiragana}{3071}
\pdfglyphtounicode{paiyannoithai}{0E2F}
\pdfglyphtounicode{pakatakana}{30D1}
\pdfglyphtounicode{palatalizationcyrilliccmb}{0484}
\pdfglyphtounicode{palochkacyrillic}{04C0}
\pdfglyphtounicode{pansioskorean}{317F}
\pdfglyphtounicode{paragraph}{00B6}
\pdfglyphtounicode{parallel}{2225}
\pdfglyphtounicode{parenleft}{0028}
\pdfglyphtounicode{parenleftaltonearabic}{FD3E}
\pdfglyphtounicode{parenleftbt}{F8ED}
\pdfglyphtounicode{parenleftex}{F8EC}
\pdfglyphtounicode{parenleftinferior}{208D}
\pdfglyphtounicode{parenleftmonospace}{FF08}
\pdfglyphtounicode{parenleftsmall}{FE59}
\pdfglyphtounicode{parenleftsuperior}{207D}
\pdfglyphtounicode{parenlefttp}{F8EB}
\pdfglyphtounicode{parenleftvertical}{FE35}
\pdfglyphtounicode{parenright}{0029}
\pdfglyphtounicode{parenrightaltonearabic}{FD3F}
\pdfglyphtounicode{parenrightbt}{F8F8}
\pdfglyphtounicode{parenrightex}{F8F7}
\pdfglyphtounicode{parenrightinferior}{208E}
\pdfglyphtounicode{parenrightmonospace}{FF09}
\pdfglyphtounicode{parenrightsmall}{FE5A}
\pdfglyphtounicode{parenrightsuperior}{207E}
\pdfglyphtounicode{parenrighttp}{F8F6}
\pdfglyphtounicode{parenrightvertical}{FE36}
\pdfglyphtounicode{partialdiff}{2202}
\pdfglyphtounicode{paseqhebrew}{05C0}
\pdfglyphtounicode{pashtahebrew}{0599}
\pdfglyphtounicode{pasquare}{33A9}
\pdfglyphtounicode{patah}{05B7}
\pdfglyphtounicode{patah11}{05B7}
\pdfglyphtounicode{patah1d}{05B7}
\pdfglyphtounicode{patah2a}{05B7}
\pdfglyphtounicode{patahhebrew}{05B7}
\pdfglyphtounicode{patahnarrowhebrew}{05B7}
\pdfglyphtounicode{patahquarterhebrew}{05B7}
\pdfglyphtounicode{patahwidehebrew}{05B7}
\pdfglyphtounicode{pazerhebrew}{05A1}
\pdfglyphtounicode{pbopomofo}{3106}
\pdfglyphtounicode{pcircle}{24DF}
\pdfglyphtounicode{pdotaccent}{1E57}
\pdfglyphtounicode{pe}{05E4}
\pdfglyphtounicode{pecyrillic}{043F}
\pdfglyphtounicode{pedagesh}{FB44}
\pdfglyphtounicode{pedageshhebrew}{FB44}
\pdfglyphtounicode{peezisquare}{333B}
\pdfglyphtounicode{pefinaldageshhebrew}{FB43}
\pdfglyphtounicode{peharabic}{067E}
\pdfglyphtounicode{peharmenian}{057A}
\pdfglyphtounicode{pehebrew}{05E4}
\pdfglyphtounicode{pehfinalarabic}{FB57}
\pdfglyphtounicode{pehinitialarabic}{FB58}
\pdfglyphtounicode{pehiragana}{307A}
\pdfglyphtounicode{pehmedialarabic}{FB59}
\pdfglyphtounicode{pekatakana}{30DA}
\pdfglyphtounicode{pemiddlehookcyrillic}{04A7}
\pdfglyphtounicode{perafehebrew}{FB4E}
\pdfglyphtounicode{percent}{0025}
\pdfglyphtounicode{percentarabic}{066A}
\pdfglyphtounicode{percentmonospace}{FF05}
\pdfglyphtounicode{percentsmall}{FE6A}
\pdfglyphtounicode{period}{002E}
\pdfglyphtounicode{periodarmenian}{0589}
\pdfglyphtounicode{periodcentered}{00B7}
\pdfglyphtounicode{periodhalfwidth}{FF61}
\pdfglyphtounicode{periodinferior}{F6E7}
\pdfglyphtounicode{periodmonospace}{FF0E}
\pdfglyphtounicode{periodsmall}{FE52}
\pdfglyphtounicode{periodsuperior}{F6E8}
\pdfglyphtounicode{perispomenigreekcmb}{0342}
\pdfglyphtounicode{perpendicular}{22A5}
\pdfglyphtounicode{perthousand}{2030}
\pdfglyphtounicode{peseta}{20A7}
\pdfglyphtounicode{pfsquare}{338A}
\pdfglyphtounicode{phabengali}{09AB}
\pdfglyphtounicode{phadeva}{092B}
\pdfglyphtounicode{phagujarati}{0AAB}
\pdfglyphtounicode{phagurmukhi}{0A2B}
\pdfglyphtounicode{phi}{03C6}
\pdfglyphtounicode{phi1}{03D5}
\pdfglyphtounicode{phieuphacirclekorean}{327A}
\pdfglyphtounicode{phieuphaparenkorean}{321A}
\pdfglyphtounicode{phieuphcirclekorean}{326C}
\pdfglyphtounicode{phieuphkorean}{314D}
\pdfglyphtounicode{phieuphparenkorean}{320C}
\pdfglyphtounicode{philatin}{0278}
\pdfglyphtounicode{phinthuthai}{0E3A}
\pdfglyphtounicode{phisymbolgreek}{03D5}
\pdfglyphtounicode{phook}{01A5}
\pdfglyphtounicode{phophanthai}{0E1E}
\pdfglyphtounicode{phophungthai}{0E1C}
\pdfglyphtounicode{phosamphaothai}{0E20}
\pdfglyphtounicode{pi}{03C0}
\pdfglyphtounicode{pieupacirclekorean}{3273}
\pdfglyphtounicode{pieupaparenkorean}{3213}
\pdfglyphtounicode{pieupcieuckorean}{3176}
\pdfglyphtounicode{pieupcirclekorean}{3265}
\pdfglyphtounicode{pieupkiyeokkorean}{3172}
\pdfglyphtounicode{pieupkorean}{3142}
\pdfglyphtounicode{pieupparenkorean}{3205}
\pdfglyphtounicode{pieupsioskiyeokkorean}{3174}
\pdfglyphtounicode{pieupsioskorean}{3144}
\pdfglyphtounicode{pieupsiostikeutkorean}{3175}
\pdfglyphtounicode{pieupthieuthkorean}{3177}
\pdfglyphtounicode{pieuptikeutkorean}{3173}
\pdfglyphtounicode{pihiragana}{3074}
\pdfglyphtounicode{pikatakana}{30D4}
\pdfglyphtounicode{pisymbolgreek}{03D6}
\pdfglyphtounicode{piwrarmenian}{0583}
\pdfglyphtounicode{plus}{002B}
\pdfglyphtounicode{plusbelowcmb}{031F}
\pdfglyphtounicode{pluscircle}{2295}
\pdfglyphtounicode{plusminus}{00B1}
\pdfglyphtounicode{plusmod}{02D6}
\pdfglyphtounicode{plusmonospace}{FF0B}
\pdfglyphtounicode{plussmall}{FE62}
\pdfglyphtounicode{plussuperior}{207A}
\pdfglyphtounicode{pmonospace}{FF50}
\pdfglyphtounicode{pmsquare}{33D8}
\pdfglyphtounicode{pohiragana}{307D}
\pdfglyphtounicode{pointingindexdownwhite}{261F}
\pdfglyphtounicode{pointingindexleftwhite}{261C}
\pdfglyphtounicode{pointingindexrightwhite}{261E}
\pdfglyphtounicode{pointingindexupwhite}{261D}
\pdfglyphtounicode{pokatakana}{30DD}
\pdfglyphtounicode{poplathai}{0E1B}
\pdfglyphtounicode{postalmark}{3012}
\pdfglyphtounicode{postalmarkface}{3020}
\pdfglyphtounicode{pparen}{24AB}
\pdfglyphtounicode{precedes}{227A}
\pdfglyphtounicode{prescription}{211E}
\pdfglyphtounicode{primemod}{02B9}
\pdfglyphtounicode{primereversed}{2035}
\pdfglyphtounicode{product}{220F}
\pdfglyphtounicode{projective}{2305}
\pdfglyphtounicode{prolongedkana}{30FC}
\pdfglyphtounicode{propellor}{2318}
\pdfglyphtounicode{propersubset}{2282}
\pdfglyphtounicode{propersuperset}{2283}
\pdfglyphtounicode{proportion}{2237}
\pdfglyphtounicode{proportional}{221D}
\pdfglyphtounicode{psi}{03C8}
\pdfglyphtounicode{psicyrillic}{0471}
\pdfglyphtounicode{psilipneumatacyrilliccmb}{0486}
\pdfglyphtounicode{pssquare}{33B0}
\pdfglyphtounicode{puhiragana}{3077}
\pdfglyphtounicode{pukatakana}{30D7}
\pdfglyphtounicode{pvsquare}{33B4}
\pdfglyphtounicode{pwsquare}{33BA}
\pdfglyphtounicode{q}{0071}
\pdfglyphtounicode{qadeva}{0958}
\pdfglyphtounicode{qadmahebrew}{05A8}
\pdfglyphtounicode{qafarabic}{0642}
\pdfglyphtounicode{qaffinalarabic}{FED6}
\pdfglyphtounicode{qafinitialarabic}{FED7}
\pdfglyphtounicode{qafmedialarabic}{FED8}
\pdfglyphtounicode{qamats}{05B8}
\pdfglyphtounicode{qamats10}{05B8}
\pdfglyphtounicode{qamats1a}{05B8}
\pdfglyphtounicode{qamats1c}{05B8}
\pdfglyphtounicode{qamats27}{05B8}
\pdfglyphtounicode{qamats29}{05B8}
\pdfglyphtounicode{qamats33}{05B8}
\pdfglyphtounicode{qamatsde}{05B8}
\pdfglyphtounicode{qamatshebrew}{05B8}
\pdfglyphtounicode{qamatsnarrowhebrew}{05B8}
\pdfglyphtounicode{qamatsqatanhebrew}{05B8}
\pdfglyphtounicode{qamatsqatannarrowhebrew}{05B8}
\pdfglyphtounicode{qamatsqatanquarterhebrew}{05B8}
\pdfglyphtounicode{qamatsqatanwidehebrew}{05B8}
\pdfglyphtounicode{qamatsquarterhebrew}{05B8}
\pdfglyphtounicode{qamatswidehebrew}{05B8}
\pdfglyphtounicode{qarneyparahebrew}{059F}
\pdfglyphtounicode{qbopomofo}{3111}
\pdfglyphtounicode{qcircle}{24E0}
\pdfglyphtounicode{qhook}{02A0}
\pdfglyphtounicode{qmonospace}{FF51}
\pdfglyphtounicode{qof}{05E7}
\pdfglyphtounicode{qofdagesh}{FB47}
\pdfglyphtounicode{qofdageshhebrew}{FB47}
\pdfglyphtounicode{qofhebrew}{05E7}
\pdfglyphtounicode{qparen}{24AC}
\pdfglyphtounicode{quarternote}{2669}
\pdfglyphtounicode{qubuts}{05BB}
\pdfglyphtounicode{qubuts18}{05BB}
\pdfglyphtounicode{qubuts25}{05BB}
\pdfglyphtounicode{qubuts31}{05BB}
\pdfglyphtounicode{qubutshebrew}{05BB}
\pdfglyphtounicode{qubutsnarrowhebrew}{05BB}
\pdfglyphtounicode{qubutsquarterhebrew}{05BB}
\pdfglyphtounicode{qubutswidehebrew}{05BB}
\pdfglyphtounicode{question}{003F}
\pdfglyphtounicode{questionarabic}{061F}
\pdfglyphtounicode{questionarmenian}{055E}
\pdfglyphtounicode{questiondown}{00BF}
\pdfglyphtounicode{questiondownsmall}{F7BF}
\pdfglyphtounicode{questiongreek}{037E}
\pdfglyphtounicode{questionmonospace}{FF1F}
\pdfglyphtounicode{questionsmall}{F73F}
\pdfglyphtounicode{quotedbl}{0022}
\pdfglyphtounicode{quotedblbase}{201E}
\pdfglyphtounicode{quotedblleft}{201C}
\pdfglyphtounicode{quotedblmonospace}{FF02}
\pdfglyphtounicode{quotedblprime}{301E}
\pdfglyphtounicode{quotedblprimereversed}{301D}
\pdfglyphtounicode{quotedblright}{201D}
\pdfglyphtounicode{quoteleft}{2018}
\pdfglyphtounicode{quoteleftreversed}{201B}
\pdfglyphtounicode{quotereversed}{201B}
\pdfglyphtounicode{quoteright}{2019}
\pdfglyphtounicode{quoterightn}{0149}
\pdfglyphtounicode{quotesinglbase}{201A}
\pdfglyphtounicode{quotesingle}{0027}
\pdfglyphtounicode{quotesinglemonospace}{FF07}
\pdfglyphtounicode{r}{0072}
\pdfglyphtounicode{raarmenian}{057C}
\pdfglyphtounicode{rabengali}{09B0}
\pdfglyphtounicode{racute}{0155}
\pdfglyphtounicode{radeva}{0930}
\pdfglyphtounicode{radical}{221A}
\pdfglyphtounicode{radicalex}{F8E5}
\pdfglyphtounicode{radoverssquare}{33AE}
\pdfglyphtounicode{radoverssquaredsquare}{33AF}
\pdfglyphtounicode{radsquare}{33AD}
\pdfglyphtounicode{rafe}{05BF}
\pdfglyphtounicode{rafehebrew}{05BF}
\pdfglyphtounicode{ragujarati}{0AB0}
\pdfglyphtounicode{ragurmukhi}{0A30}
\pdfglyphtounicode{rahiragana}{3089}
\pdfglyphtounicode{rakatakana}{30E9}
\pdfglyphtounicode{rakatakanahalfwidth}{FF97}
\pdfglyphtounicode{ralowerdiagonalbengali}{09F1}
\pdfglyphtounicode{ramiddlediagonalbengali}{09F0}
\pdfglyphtounicode{ramshorn}{0264}
\pdfglyphtounicode{ratio}{2236}
\pdfglyphtounicode{rbopomofo}{3116}
\pdfglyphtounicode{rcaron}{0159}
\pdfglyphtounicode{rcedilla}{0157}
\pdfglyphtounicode{rcircle}{24E1}
\pdfglyphtounicode{rcommaaccent}{0157}
\pdfglyphtounicode{rdblgrave}{0211}
\pdfglyphtounicode{rdotaccent}{1E59}
\pdfglyphtounicode{rdotbelow}{1E5B}
\pdfglyphtounicode{rdotbelowmacron}{1E5D}
\pdfglyphtounicode{referencemark}{203B}
\pdfglyphtounicode{reflexsubset}{2286}
\pdfglyphtounicode{reflexsuperset}{2287}
\pdfglyphtounicode{registered}{00AE}
\pdfglyphtounicode{registersans}{F8E8}
\pdfglyphtounicode{registerserif}{F6DA}
\pdfglyphtounicode{reharabic}{0631}
\pdfglyphtounicode{reharmenian}{0580}
\pdfglyphtounicode{rehfinalarabic}{FEAE}
\pdfglyphtounicode{rehiragana}{308C}
\pdfglyphtounicode{rekatakana}{30EC}
\pdfglyphtounicode{rekatakanahalfwidth}{FF9A}
\pdfglyphtounicode{resh}{05E8}
\pdfglyphtounicode{reshdageshhebrew}{FB48}
\pdfglyphtounicode{reshhebrew}{05E8}
\pdfglyphtounicode{reversedtilde}{223D}
\pdfglyphtounicode{reviahebrew}{0597}
\pdfglyphtounicode{reviamugrashhebrew}{0597}
\pdfglyphtounicode{revlogicalnot}{2310}
\pdfglyphtounicode{rfishhook}{027E}
\pdfglyphtounicode{rfishhookreversed}{027F}
\pdfglyphtounicode{rhabengali}{09DD}
\pdfglyphtounicode{rhadeva}{095D}
\pdfglyphtounicode{rho}{03C1}
\pdfglyphtounicode{rhook}{027D}
\pdfglyphtounicode{rhookturned}{027B}
\pdfglyphtounicode{rhookturnedsuperior}{02B5}
\pdfglyphtounicode{rhosymbolgreek}{03F1}
\pdfglyphtounicode{rhotichookmod}{02DE}
\pdfglyphtounicode{rieulacirclekorean}{3271}
\pdfglyphtounicode{rieulaparenkorean}{3211}
\pdfglyphtounicode{rieulcirclekorean}{3263}
\pdfglyphtounicode{rieulhieuhkorean}{3140}
\pdfglyphtounicode{rieulkiyeokkorean}{313A}
\pdfglyphtounicode{rieulkiyeoksioskorean}{3169}
\pdfglyphtounicode{rieulkorean}{3139}
\pdfglyphtounicode{rieulmieumkorean}{313B}
\pdfglyphtounicode{rieulpansioskorean}{316C}
\pdfglyphtounicode{rieulparenkorean}{3203}
\pdfglyphtounicode{rieulphieuphkorean}{313F}
\pdfglyphtounicode{rieulpieupkorean}{313C}
\pdfglyphtounicode{rieulpieupsioskorean}{316B}
\pdfglyphtounicode{rieulsioskorean}{313D}
\pdfglyphtounicode{rieulthieuthkorean}{313E}
\pdfglyphtounicode{rieultikeutkorean}{316A}
\pdfglyphtounicode{rieulyeorinhieuhkorean}{316D}
\pdfglyphtounicode{rightangle}{221F}
\pdfglyphtounicode{righttackbelowcmb}{0319}
\pdfglyphtounicode{righttriangle}{22BF}
\pdfglyphtounicode{rihiragana}{308A}
\pdfglyphtounicode{rikatakana}{30EA}
\pdfglyphtounicode{rikatakanahalfwidth}{FF98}
\pdfglyphtounicode{ring}{02DA}
\pdfglyphtounicode{ringbelowcmb}{0325}
\pdfglyphtounicode{ringcmb}{030A}
\pdfglyphtounicode{ringhalfleft}{02BF}
\pdfglyphtounicode{ringhalfleftarmenian}{0559}
\pdfglyphtounicode{ringhalfleftbelowcmb}{031C}
\pdfglyphtounicode{ringhalfleftcentered}{02D3}
\pdfglyphtounicode{ringhalfright}{02BE}
\pdfglyphtounicode{ringhalfrightbelowcmb}{0339}
\pdfglyphtounicode{ringhalfrightcentered}{02D2}
\pdfglyphtounicode{rinvertedbreve}{0213}
\pdfglyphtounicode{rittorusquare}{3351}
\pdfglyphtounicode{rlinebelow}{1E5F}
\pdfglyphtounicode{rlongleg}{027C}
\pdfglyphtounicode{rlonglegturned}{027A}
\pdfglyphtounicode{rmonospace}{FF52}
\pdfglyphtounicode{rohiragana}{308D}
\pdfglyphtounicode{rokatakana}{30ED}
\pdfglyphtounicode{rokatakanahalfwidth}{FF9B}
\pdfglyphtounicode{roruathai}{0E23}
\pdfglyphtounicode{rparen}{24AD}
\pdfglyphtounicode{rrabengali}{09DC}
\pdfglyphtounicode{rradeva}{0931}
\pdfglyphtounicode{rragurmukhi}{0A5C}
\pdfglyphtounicode{rreharabic}{0691}
\pdfglyphtounicode{rrehfinalarabic}{FB8D}
\pdfglyphtounicode{rrvocalicbengali}{09E0}
\pdfglyphtounicode{rrvocalicdeva}{0960}
\pdfglyphtounicode{rrvocalicgujarati}{0AE0}
\pdfglyphtounicode{rrvocalicvowelsignbengali}{09C4}
\pdfglyphtounicode{rrvocalicvowelsigndeva}{0944}
\pdfglyphtounicode{rrvocalicvowelsigngujarati}{0AC4}
\pdfglyphtounicode{rsuperior}{F6F1}
\pdfglyphtounicode{rtblock}{2590}
\pdfglyphtounicode{rturned}{0279}
\pdfglyphtounicode{rturnedsuperior}{02B4}
\pdfglyphtounicode{ruhiragana}{308B}
\pdfglyphtounicode{rukatakana}{30EB}
\pdfglyphtounicode{rukatakanahalfwidth}{FF99}
\pdfglyphtounicode{rupeemarkbengali}{09F2}
\pdfglyphtounicode{rupeesignbengali}{09F3}
\pdfglyphtounicode{rupiah}{F6DD}
\pdfglyphtounicode{ruthai}{0E24}
\pdfglyphtounicode{rvocalicbengali}{098B}
\pdfglyphtounicode{rvocalicdeva}{090B}
\pdfglyphtounicode{rvocalicgujarati}{0A8B}
\pdfglyphtounicode{rvocalicvowelsignbengali}{09C3}
\pdfglyphtounicode{rvocalicvowelsigndeva}{0943}
\pdfglyphtounicode{rvocalicvowelsigngujarati}{0AC3}
\pdfglyphtounicode{s}{0073}
\pdfglyphtounicode{sabengali}{09B8}
\pdfglyphtounicode{sacute}{015B}
\pdfglyphtounicode{sacutedotaccent}{1E65}
\pdfglyphtounicode{sadarabic}{0635}
\pdfglyphtounicode{sadeva}{0938}
\pdfglyphtounicode{sadfinalarabic}{FEBA}
\pdfglyphtounicode{sadinitialarabic}{FEBB}
\pdfglyphtounicode{sadmedialarabic}{FEBC}
\pdfglyphtounicode{sagujarati}{0AB8}
\pdfglyphtounicode{sagurmukhi}{0A38}
\pdfglyphtounicode{sahiragana}{3055}
\pdfglyphtounicode{sakatakana}{30B5}
\pdfglyphtounicode{sakatakanahalfwidth}{FF7B}
\pdfglyphtounicode{sallallahoualayhewasallamarabic}{FDFA}
\pdfglyphtounicode{samekh}{05E1}
\pdfglyphtounicode{samekhdagesh}{FB41}
\pdfglyphtounicode{samekhdageshhebrew}{FB41}
\pdfglyphtounicode{samekhhebrew}{05E1}
\pdfglyphtounicode{saraaathai}{0E32}
\pdfglyphtounicode{saraaethai}{0E41}
\pdfglyphtounicode{saraaimaimalaithai}{0E44}
\pdfglyphtounicode{saraaimaimuanthai}{0E43}
\pdfglyphtounicode{saraamthai}{0E33}
\pdfglyphtounicode{saraathai}{0E30}
\pdfglyphtounicode{saraethai}{0E40}
\pdfglyphtounicode{saraiileftthai}{F886}
\pdfglyphtounicode{saraiithai}{0E35}
\pdfglyphtounicode{saraileftthai}{F885}
\pdfglyphtounicode{saraithai}{0E34}
\pdfglyphtounicode{saraothai}{0E42}
\pdfglyphtounicode{saraueeleftthai}{F888}
\pdfglyphtounicode{saraueethai}{0E37}
\pdfglyphtounicode{saraueleftthai}{F887}
\pdfglyphtounicode{sarauethai}{0E36}
\pdfglyphtounicode{sarauthai}{0E38}
\pdfglyphtounicode{sarauuthai}{0E39}
\pdfglyphtounicode{sbopomofo}{3119}
\pdfglyphtounicode{scaron}{0161}
\pdfglyphtounicode{scarondotaccent}{1E67}
\pdfglyphtounicode{scedilla}{015F}
\pdfglyphtounicode{schwa}{0259}
\pdfglyphtounicode{schwacyrillic}{04D9}
\pdfglyphtounicode{schwadieresiscyrillic}{04DB}
\pdfglyphtounicode{schwahook}{025A}
\pdfglyphtounicode{scircle}{24E2}
\pdfglyphtounicode{scircumflex}{015D}
\pdfglyphtounicode{scommaaccent}{0219}
\pdfglyphtounicode{sdotaccent}{1E61}
\pdfglyphtounicode{sdotbelow}{1E63}
\pdfglyphtounicode{sdotbelowdotaccent}{1E69}
\pdfglyphtounicode{seagullbelowcmb}{033C}
\pdfglyphtounicode{second}{2033}
\pdfglyphtounicode{secondtonechinese}{02CA}
\pdfglyphtounicode{section}{00A7}
\pdfglyphtounicode{seenarabic}{0633}
\pdfglyphtounicode{seenfinalarabic}{FEB2}
\pdfglyphtounicode{seeninitialarabic}{FEB3}
\pdfglyphtounicode{seenmedialarabic}{FEB4}
\pdfglyphtounicode{segol}{05B6}
\pdfglyphtounicode{segol13}{05B6}
\pdfglyphtounicode{segol1f}{05B6}
\pdfglyphtounicode{segol2c}{05B6}
\pdfglyphtounicode{segolhebrew}{05B6}
\pdfglyphtounicode{segolnarrowhebrew}{05B6}
\pdfglyphtounicode{segolquarterhebrew}{05B6}
\pdfglyphtounicode{segoltahebrew}{0592}
\pdfglyphtounicode{segolwidehebrew}{05B6}
\pdfglyphtounicode{seharmenian}{057D}
\pdfglyphtounicode{sehiragana}{305B}
\pdfglyphtounicode{sekatakana}{30BB}
\pdfglyphtounicode{sekatakanahalfwidth}{FF7E}
\pdfglyphtounicode{semicolon}{003B}
\pdfglyphtounicode{semicolonarabic}{061B}
\pdfglyphtounicode{semicolonmonospace}{FF1B}
\pdfglyphtounicode{semicolonsmall}{FE54}
\pdfglyphtounicode{semivoicedmarkkana}{309C}
\pdfglyphtounicode{semivoicedmarkkanahalfwidth}{FF9F}
\pdfglyphtounicode{sentisquare}{3322}
\pdfglyphtounicode{sentosquare}{3323}
\pdfglyphtounicode{seven}{0037}
\pdfglyphtounicode{sevenarabic}{0667}
\pdfglyphtounicode{sevenbengali}{09ED}
\pdfglyphtounicode{sevencircle}{2466}
\pdfglyphtounicode{sevencircleinversesansserif}{2790}
\pdfglyphtounicode{sevendeva}{096D}
\pdfglyphtounicode{seveneighths}{215E}
\pdfglyphtounicode{sevengujarati}{0AED}
\pdfglyphtounicode{sevengurmukhi}{0A6D}
\pdfglyphtounicode{sevenhackarabic}{0667}
\pdfglyphtounicode{sevenhangzhou}{3027}
\pdfglyphtounicode{sevenideographicparen}{3226}
\pdfglyphtounicode{seveninferior}{2087}
\pdfglyphtounicode{sevenmonospace}{FF17}
\pdfglyphtounicode{sevenoldstyle}{F737}
\pdfglyphtounicode{sevenparen}{247A}
\pdfglyphtounicode{sevenperiod}{248E}
\pdfglyphtounicode{sevenpersian}{06F7}
\pdfglyphtounicode{sevenroman}{2176}
\pdfglyphtounicode{sevensuperior}{2077}
\pdfglyphtounicode{seventeencircle}{2470}
\pdfglyphtounicode{seventeenparen}{2484}
\pdfglyphtounicode{seventeenperiod}{2498}
\pdfglyphtounicode{seventhai}{0E57}
\pdfglyphtounicode{sfthyphen}{00AD}
\pdfglyphtounicode{shaarmenian}{0577}
\pdfglyphtounicode{shabengali}{09B6}
\pdfglyphtounicode{shacyrillic}{0448}
\pdfglyphtounicode{shaddaarabic}{0651}
\pdfglyphtounicode{shaddadammaarabic}{FC61}
\pdfglyphtounicode{shaddadammatanarabic}{FC5E}
\pdfglyphtounicode{shaddafathaarabic}{FC60}
\pdfglyphtounicode{shaddakasraarabic}{FC62}
\pdfglyphtounicode{shaddakasratanarabic}{FC5F}
\pdfglyphtounicode{shade}{2592}
\pdfglyphtounicode{shadedark}{2593}
\pdfglyphtounicode{shadelight}{2591}
\pdfglyphtounicode{shademedium}{2592}
\pdfglyphtounicode{shadeva}{0936}
\pdfglyphtounicode{shagujarati}{0AB6}
\pdfglyphtounicode{shagurmukhi}{0A36}
\pdfglyphtounicode{shalshelethebrew}{0593}
\pdfglyphtounicode{shbopomofo}{3115}
\pdfglyphtounicode{shchacyrillic}{0449}
\pdfglyphtounicode{sheenarabic}{0634}
\pdfglyphtounicode{sheenfinalarabic}{FEB6}
\pdfglyphtounicode{sheeninitialarabic}{FEB7}
\pdfglyphtounicode{sheenmedialarabic}{FEB8}
\pdfglyphtounicode{sheicoptic}{03E3}
\pdfglyphtounicode{sheqel}{20AA}
\pdfglyphtounicode{sheqelhebrew}{20AA}
\pdfglyphtounicode{sheva}{05B0}
\pdfglyphtounicode{sheva115}{05B0}
\pdfglyphtounicode{sheva15}{05B0}
\pdfglyphtounicode{sheva22}{05B0}
\pdfglyphtounicode{sheva2e}{05B0}
\pdfglyphtounicode{shevahebrew}{05B0}
\pdfglyphtounicode{shevanarrowhebrew}{05B0}
\pdfglyphtounicode{shevaquarterhebrew}{05B0}
\pdfglyphtounicode{shevawidehebrew}{05B0}
\pdfglyphtounicode{shhacyrillic}{04BB}
\pdfglyphtounicode{shimacoptic}{03ED}
\pdfglyphtounicode{shin}{05E9}
\pdfglyphtounicode{shindagesh}{FB49}
\pdfglyphtounicode{shindageshhebrew}{FB49}
\pdfglyphtounicode{shindageshshindot}{FB2C}
\pdfglyphtounicode{shindageshshindothebrew}{FB2C}
\pdfglyphtounicode{shindageshsindot}{FB2D}
\pdfglyphtounicode{shindageshsindothebrew}{FB2D}
\pdfglyphtounicode{shindothebrew}{05C1}
\pdfglyphtounicode{shinhebrew}{05E9}
\pdfglyphtounicode{shinshindot}{FB2A}
\pdfglyphtounicode{shinshindothebrew}{FB2A}
\pdfglyphtounicode{shinsindot}{FB2B}
\pdfglyphtounicode{shinsindothebrew}{FB2B}
\pdfglyphtounicode{shook}{0282}
\pdfglyphtounicode{sigma}{03C3}
\pdfglyphtounicode{sigma1}{03C2}
\pdfglyphtounicode{sigmafinal}{03C2}
\pdfglyphtounicode{sigmalunatesymbolgreek}{03F2}
\pdfglyphtounicode{sihiragana}{3057}
\pdfglyphtounicode{sikatakana}{30B7}
\pdfglyphtounicode{sikatakanahalfwidth}{FF7C}
\pdfglyphtounicode{siluqhebrew}{05BD}
\pdfglyphtounicode{siluqlefthebrew}{05BD}
\pdfglyphtounicode{similar}{223C}
\pdfglyphtounicode{sindothebrew}{05C2}
\pdfglyphtounicode{siosacirclekorean}{3274}
\pdfglyphtounicode{siosaparenkorean}{3214}
\pdfglyphtounicode{sioscieuckorean}{317E}
\pdfglyphtounicode{sioscirclekorean}{3266}
\pdfglyphtounicode{sioskiyeokkorean}{317A}
\pdfglyphtounicode{sioskorean}{3145}
\pdfglyphtounicode{siosnieunkorean}{317B}
\pdfglyphtounicode{siosparenkorean}{3206}
\pdfglyphtounicode{siospieupkorean}{317D}
\pdfglyphtounicode{siostikeutkorean}{317C}
\pdfglyphtounicode{six}{0036}
\pdfglyphtounicode{sixarabic}{0666}
\pdfglyphtounicode{sixbengali}{09EC}
\pdfglyphtounicode{sixcircle}{2465}
\pdfglyphtounicode{sixcircleinversesansserif}{278F}
\pdfglyphtounicode{sixdeva}{096C}
\pdfglyphtounicode{sixgujarati}{0AEC}
\pdfglyphtounicode{sixgurmukhi}{0A6C}
\pdfglyphtounicode{sixhackarabic}{0666}
\pdfglyphtounicode{sixhangzhou}{3026}
\pdfglyphtounicode{sixideographicparen}{3225}
\pdfglyphtounicode{sixinferior}{2086}
\pdfglyphtounicode{sixmonospace}{FF16}
\pdfglyphtounicode{sixoldstyle}{F736}
\pdfglyphtounicode{sixparen}{2479}
\pdfglyphtounicode{sixperiod}{248D}
\pdfglyphtounicode{sixpersian}{06F6}
\pdfglyphtounicode{sixroman}{2175}
\pdfglyphtounicode{sixsuperior}{2076}
\pdfglyphtounicode{sixteencircle}{246F}
\pdfglyphtounicode{sixteencurrencydenominatorbengali}{09F9}
\pdfglyphtounicode{sixteenparen}{2483}
\pdfglyphtounicode{sixteenperiod}{2497}
\pdfglyphtounicode{sixthai}{0E56}
\pdfglyphtounicode{slash}{002F}
\pdfglyphtounicode{slashmonospace}{FF0F}
\pdfglyphtounicode{slong}{017F}
\pdfglyphtounicode{slongdotaccent}{1E9B}
\pdfglyphtounicode{smileface}{263A}
\pdfglyphtounicode{smonospace}{FF53}
\pdfglyphtounicode{sofpasuqhebrew}{05C3}
\pdfglyphtounicode{softhyphen}{00AD}
\pdfglyphtounicode{softsigncyrillic}{044C}
\pdfglyphtounicode{sohiragana}{305D}
\pdfglyphtounicode{sokatakana}{30BD}
\pdfglyphtounicode{sokatakanahalfwidth}{FF7F}
\pdfglyphtounicode{soliduslongoverlaycmb}{0338}
\pdfglyphtounicode{solidusshortoverlaycmb}{0337}
\pdfglyphtounicode{sorusithai}{0E29}
\pdfglyphtounicode{sosalathai}{0E28}
\pdfglyphtounicode{sosothai}{0E0B}
\pdfglyphtounicode{sosuathai}{0E2A}
\pdfglyphtounicode{space}{0020}
\pdfglyphtounicode{spacehackarabic}{0020}
\pdfglyphtounicode{spade}{2660}
\pdfglyphtounicode{spadesuitblack}{2660}
\pdfglyphtounicode{spadesuitwhite}{2664}
\pdfglyphtounicode{sparen}{24AE}
\pdfglyphtounicode{squarebelowcmb}{033B}
\pdfglyphtounicode{squarecc}{33C4}
\pdfglyphtounicode{squarecm}{339D}
\pdfglyphtounicode{squarediagonalcrosshatchfill}{25A9}
\pdfglyphtounicode{squarehorizontalfill}{25A4}
\pdfglyphtounicode{squarekg}{338F}
\pdfglyphtounicode{squarekm}{339E}
\pdfglyphtounicode{squarekmcapital}{33CE}
\pdfglyphtounicode{squareln}{33D1}
\pdfglyphtounicode{squarelog}{33D2}
\pdfglyphtounicode{squaremg}{338E}
\pdfglyphtounicode{squaremil}{33D5}
\pdfglyphtounicode{squaremm}{339C}
\pdfglyphtounicode{squaremsquared}{33A1}
\pdfglyphtounicode{squareorthogonalcrosshatchfill}{25A6}
\pdfglyphtounicode{squareupperlefttolowerrightfill}{25A7}
\pdfglyphtounicode{squareupperrighttolowerleftfill}{25A8}
\pdfglyphtounicode{squareverticalfill}{25A5}
\pdfglyphtounicode{squarewhitewithsmallblack}{25A3}
\pdfglyphtounicode{srsquare}{33DB}
\pdfglyphtounicode{ssabengali}{09B7}
\pdfglyphtounicode{ssadeva}{0937}
\pdfglyphtounicode{ssagujarati}{0AB7}
\pdfglyphtounicode{ssangcieuckorean}{3149}
\pdfglyphtounicode{ssanghieuhkorean}{3185}
\pdfglyphtounicode{ssangieungkorean}{3180}
\pdfglyphtounicode{ssangkiyeokkorean}{3132}
\pdfglyphtounicode{ssangnieunkorean}{3165}
\pdfglyphtounicode{ssangpieupkorean}{3143}
\pdfglyphtounicode{ssangsioskorean}{3146}
\pdfglyphtounicode{ssangtikeutkorean}{3138}
\pdfglyphtounicode{ssuperior}{F6F2}
\pdfglyphtounicode{sterling}{00A3}
\pdfglyphtounicode{sterlingmonospace}{FFE1}
\pdfglyphtounicode{strokelongoverlaycmb}{0336}
\pdfglyphtounicode{strokeshortoverlaycmb}{0335}
\pdfglyphtounicode{subset}{2282}
\pdfglyphtounicode{subsetnotequal}{228A}
\pdfglyphtounicode{subsetorequal}{2286}
\pdfglyphtounicode{succeeds}{227B}
\pdfglyphtounicode{suchthat}{220B}
\pdfglyphtounicode{suhiragana}{3059}
\pdfglyphtounicode{sukatakana}{30B9}
\pdfglyphtounicode{sukatakanahalfwidth}{FF7D}
\pdfglyphtounicode{sukunarabic}{0652}
\pdfglyphtounicode{summation}{2211}
\pdfglyphtounicode{sun}{263C}
\pdfglyphtounicode{superset}{2283}
\pdfglyphtounicode{supersetnotequal}{228B}
\pdfglyphtounicode{supersetorequal}{2287}
\pdfglyphtounicode{svsquare}{33DC}
\pdfglyphtounicode{syouwaerasquare}{337C}
\pdfglyphtounicode{t}{0074}
\pdfglyphtounicode{tabengali}{09A4}
\pdfglyphtounicode{tackdown}{22A4}
\pdfglyphtounicode{tackleft}{22A3}
\pdfglyphtounicode{tadeva}{0924}
\pdfglyphtounicode{tagujarati}{0AA4}
\pdfglyphtounicode{tagurmukhi}{0A24}
\pdfglyphtounicode{taharabic}{0637}
\pdfglyphtounicode{tahfinalarabic}{FEC2}
\pdfglyphtounicode{tahinitialarabic}{FEC3}
\pdfglyphtounicode{tahiragana}{305F}
\pdfglyphtounicode{tahmedialarabic}{FEC4}
\pdfglyphtounicode{taisyouerasquare}{337D}
\pdfglyphtounicode{takatakana}{30BF}
\pdfglyphtounicode{takatakanahalfwidth}{FF80}
\pdfglyphtounicode{tatweelarabic}{0640}
\pdfglyphtounicode{tau}{03C4}
\pdfglyphtounicode{tav}{05EA}
\pdfglyphtounicode{tavdages}{FB4A}
\pdfglyphtounicode{tavdagesh}{FB4A}
\pdfglyphtounicode{tavdageshhebrew}{FB4A}
\pdfglyphtounicode{tavhebrew}{05EA}
\pdfglyphtounicode{tbar}{0167}
\pdfglyphtounicode{tbopomofo}{310A}
\pdfglyphtounicode{tcaron}{0165}
\pdfglyphtounicode{tccurl}{02A8}
\pdfglyphtounicode{tcedilla}{0163}
\pdfglyphtounicode{tcheharabic}{0686}
\pdfglyphtounicode{tchehfinalarabic}{FB7B}
\pdfglyphtounicode{tchehinitialarabic}{FB7C}
\pdfglyphtounicode{tchehmedialarabic}{FB7D}
\pdfglyphtounicode{tcircle}{24E3}
\pdfglyphtounicode{tcircumflexbelow}{1E71}
\pdfglyphtounicode{tcommaaccent}{0163}
\pdfglyphtounicode{tdieresis}{1E97}
\pdfglyphtounicode{tdotaccent}{1E6B}
\pdfglyphtounicode{tdotbelow}{1E6D}
\pdfglyphtounicode{tecyrillic}{0442}
\pdfglyphtounicode{tedescendercyrillic}{04AD}
\pdfglyphtounicode{teharabic}{062A}
\pdfglyphtounicode{tehfinalarabic}{FE96}
\pdfglyphtounicode{tehhahinitialarabic}{FCA2}
\pdfglyphtounicode{tehhahisolatedarabic}{FC0C}
\pdfglyphtounicode{tehinitialarabic}{FE97}
\pdfglyphtounicode{tehiragana}{3066}
\pdfglyphtounicode{tehjeeminitialarabic}{FCA1}
\pdfglyphtounicode{tehjeemisolatedarabic}{FC0B}
\pdfglyphtounicode{tehmarbutaarabic}{0629}
\pdfglyphtounicode{tehmarbutafinalarabic}{FE94}
\pdfglyphtounicode{tehmedialarabic}{FE98}
\pdfglyphtounicode{tehmeeminitialarabic}{FCA4}
\pdfglyphtounicode{tehmeemisolatedarabic}{FC0E}
\pdfglyphtounicode{tehnoonfinalarabic}{FC73}
\pdfglyphtounicode{tekatakana}{30C6}
\pdfglyphtounicode{tekatakanahalfwidth}{FF83}
\pdfglyphtounicode{telephone}{2121}
\pdfglyphtounicode{telephoneblack}{260E}
\pdfglyphtounicode{telishagedolahebrew}{05A0}
\pdfglyphtounicode{telishaqetanahebrew}{05A9}
\pdfglyphtounicode{tencircle}{2469}
\pdfglyphtounicode{tenideographicparen}{3229}
\pdfglyphtounicode{tenparen}{247D}
\pdfglyphtounicode{tenperiod}{2491}
\pdfglyphtounicode{tenroman}{2179}
\pdfglyphtounicode{tesh}{02A7}
\pdfglyphtounicode{tet}{05D8}
\pdfglyphtounicode{tetdagesh}{FB38}
\pdfglyphtounicode{tetdageshhebrew}{FB38}
\pdfglyphtounicode{tethebrew}{05D8}
\pdfglyphtounicode{tetsecyrillic}{04B5}
\pdfglyphtounicode{tevirhebrew}{059B}
\pdfglyphtounicode{tevirlefthebrew}{059B}
\pdfglyphtounicode{thabengali}{09A5}
\pdfglyphtounicode{thadeva}{0925}
\pdfglyphtounicode{thagujarati}{0AA5}
\pdfglyphtounicode{thagurmukhi}{0A25}
\pdfglyphtounicode{thalarabic}{0630}
\pdfglyphtounicode{thalfinalarabic}{FEAC}
\pdfglyphtounicode{thanthakhatlowleftthai}{F898}
\pdfglyphtounicode{thanthakhatlowrightthai}{F897}
\pdfglyphtounicode{thanthakhatthai}{0E4C}
\pdfglyphtounicode{thanthakhatupperleftthai}{F896}
\pdfglyphtounicode{theharabic}{062B}
\pdfglyphtounicode{thehfinalarabic}{FE9A}
\pdfglyphtounicode{thehinitialarabic}{FE9B}
\pdfglyphtounicode{thehmedialarabic}{FE9C}
\pdfglyphtounicode{thereexists}{2203}
\pdfglyphtounicode{therefore}{2234}
\pdfglyphtounicode{theta}{03B8}
\pdfglyphtounicode{theta1}{03D1}
\pdfglyphtounicode{thetasymbolgreek}{03D1}
\pdfglyphtounicode{thieuthacirclekorean}{3279}
\pdfglyphtounicode{thieuthaparenkorean}{3219}
\pdfglyphtounicode{thieuthcirclekorean}{326B}
\pdfglyphtounicode{thieuthkorean}{314C}
\pdfglyphtounicode{thieuthparenkorean}{320B}
\pdfglyphtounicode{thirteencircle}{246C}
\pdfglyphtounicode{thirteenparen}{2480}
\pdfglyphtounicode{thirteenperiod}{2494}
\pdfglyphtounicode{thonangmonthothai}{0E11}
\pdfglyphtounicode{thook}{01AD}
\pdfglyphtounicode{thophuthaothai}{0E12}
\pdfglyphtounicode{thorn}{00FE}
\pdfglyphtounicode{thothahanthai}{0E17}
\pdfglyphtounicode{thothanthai}{0E10}
\pdfglyphtounicode{thothongthai}{0E18}
\pdfglyphtounicode{thothungthai}{0E16}
\pdfglyphtounicode{thousandcyrillic}{0482}
\pdfglyphtounicode{thousandsseparatorarabic}{066C}
\pdfglyphtounicode{thousandsseparatorpersian}{066C}
\pdfglyphtounicode{three}{0033}
\pdfglyphtounicode{threearabic}{0663}
\pdfglyphtounicode{threebengali}{09E9}
\pdfglyphtounicode{threecircle}{2462}
\pdfglyphtounicode{threecircleinversesansserif}{278C}
\pdfglyphtounicode{threedeva}{0969}
\pdfglyphtounicode{threeeighths}{215C}
\pdfglyphtounicode{threegujarati}{0AE9}
\pdfglyphtounicode{threegurmukhi}{0A69}
\pdfglyphtounicode{threehackarabic}{0663}
\pdfglyphtounicode{threehangzhou}{3023}
\pdfglyphtounicode{threeideographicparen}{3222}
\pdfglyphtounicode{threeinferior}{2083}
\pdfglyphtounicode{threemonospace}{FF13}
\pdfglyphtounicode{threenumeratorbengali}{09F6}
\pdfglyphtounicode{threeoldstyle}{F733}
\pdfglyphtounicode{threeparen}{2476}
\pdfglyphtounicode{threeperiod}{248A}
\pdfglyphtounicode{threepersian}{06F3}
\pdfglyphtounicode{threequarters}{00BE}
\pdfglyphtounicode{threequartersemdash}{F6DE}
\pdfglyphtounicode{threeroman}{2172}
\pdfglyphtounicode{threesuperior}{00B3}
\pdfglyphtounicode{threethai}{0E53}
\pdfglyphtounicode{thzsquare}{3394}
\pdfglyphtounicode{tihiragana}{3061}
\pdfglyphtounicode{tikatakana}{30C1}
\pdfglyphtounicode{tikatakanahalfwidth}{FF81}
\pdfglyphtounicode{tikeutacirclekorean}{3270}
\pdfglyphtounicode{tikeutaparenkorean}{3210}
\pdfglyphtounicode{tikeutcirclekorean}{3262}
\pdfglyphtounicode{tikeutkorean}{3137}
\pdfglyphtounicode{tikeutparenkorean}{3202}
\pdfglyphtounicode{tilde}{02DC}
\pdfglyphtounicode{tildebelowcmb}{0330}
\pdfglyphtounicode{tildecmb}{0303}
\pdfglyphtounicode{tildecomb}{0303}
\pdfglyphtounicode{tildedoublecmb}{0360}
\pdfglyphtounicode{tildeoperator}{223C}
\pdfglyphtounicode{tildeoverlaycmb}{0334}
\pdfglyphtounicode{tildeverticalcmb}{033E}
\pdfglyphtounicode{timescircle}{2297}
\pdfglyphtounicode{tipehahebrew}{0596}
\pdfglyphtounicode{tipehalefthebrew}{0596}
\pdfglyphtounicode{tippigurmukhi}{0A70}
\pdfglyphtounicode{titlocyrilliccmb}{0483}
\pdfglyphtounicode{tiwnarmenian}{057F}
\pdfglyphtounicode{tlinebelow}{1E6F}
\pdfglyphtounicode{tmonospace}{FF54}
\pdfglyphtounicode{toarmenian}{0569}
\pdfglyphtounicode{tohiragana}{3068}
\pdfglyphtounicode{tokatakana}{30C8}
\pdfglyphtounicode{tokatakanahalfwidth}{FF84}
\pdfglyphtounicode{tonebarextrahighmod}{02E5}
\pdfglyphtounicode{tonebarextralowmod}{02E9}
\pdfglyphtounicode{tonebarhighmod}{02E6}
\pdfglyphtounicode{tonebarlowmod}{02E8}
\pdfglyphtounicode{tonebarmidmod}{02E7}
\pdfglyphtounicode{tonefive}{01BD}
\pdfglyphtounicode{tonesix}{0185}
\pdfglyphtounicode{tonetwo}{01A8}
\pdfglyphtounicode{tonos}{0384}
\pdfglyphtounicode{tonsquare}{3327}
\pdfglyphtounicode{topatakthai}{0E0F}
\pdfglyphtounicode{tortoiseshellbracketleft}{3014}
\pdfglyphtounicode{tortoiseshellbracketleftsmall}{FE5D}
\pdfglyphtounicode{tortoiseshellbracketleftvertical}{FE39}
\pdfglyphtounicode{tortoiseshellbracketright}{3015}
\pdfglyphtounicode{tortoiseshellbracketrightsmall}{FE5E}
\pdfglyphtounicode{tortoiseshellbracketrightvertical}{FE3A}
\pdfglyphtounicode{totaothai}{0E15}
\pdfglyphtounicode{tpalatalhook}{01AB}
\pdfglyphtounicode{tparen}{24AF}
\pdfglyphtounicode{trademark}{2122}
\pdfglyphtounicode{trademarksans}{F8EA}
\pdfglyphtounicode{trademarkserif}{F6DB}
\pdfglyphtounicode{tretroflexhook}{0288}
\pdfglyphtounicode{triagdn}{25BC}
\pdfglyphtounicode{triaglf}{25C4}
\pdfglyphtounicode{triagrt}{25BA}
\pdfglyphtounicode{triagup}{25B2}
\pdfglyphtounicode{ts}{02A6}
\pdfglyphtounicode{tsadi}{05E6}
\pdfglyphtounicode{tsadidagesh}{FB46}
\pdfglyphtounicode{tsadidageshhebrew}{FB46}
\pdfglyphtounicode{tsadihebrew}{05E6}
\pdfglyphtounicode{tsecyrillic}{0446}
\pdfglyphtounicode{tsere}{05B5}
\pdfglyphtounicode{tsere12}{05B5}
\pdfglyphtounicode{tsere1e}{05B5}
\pdfglyphtounicode{tsere2b}{05B5}
\pdfglyphtounicode{tserehebrew}{05B5}
\pdfglyphtounicode{tserenarrowhebrew}{05B5}
\pdfglyphtounicode{tserequarterhebrew}{05B5}
\pdfglyphtounicode{tserewidehebrew}{05B5}
\pdfglyphtounicode{tshecyrillic}{045B}
\pdfglyphtounicode{tsuperior}{F6F3}
\pdfglyphtounicode{ttabengali}{099F}
\pdfglyphtounicode{ttadeva}{091F}
\pdfglyphtounicode{ttagujarati}{0A9F}
\pdfglyphtounicode{ttagurmukhi}{0A1F}
\pdfglyphtounicode{tteharabic}{0679}
\pdfglyphtounicode{ttehfinalarabic}{FB67}
\pdfglyphtounicode{ttehinitialarabic}{FB68}
\pdfglyphtounicode{ttehmedialarabic}{FB69}
\pdfglyphtounicode{tthabengali}{09A0}
\pdfglyphtounicode{tthadeva}{0920}
\pdfglyphtounicode{tthagujarati}{0AA0}
\pdfglyphtounicode{tthagurmukhi}{0A20}
\pdfglyphtounicode{tturned}{0287}
\pdfglyphtounicode{tuhiragana}{3064}
\pdfglyphtounicode{tukatakana}{30C4}
\pdfglyphtounicode{tukatakanahalfwidth}{FF82}
\pdfglyphtounicode{tusmallhiragana}{3063}
\pdfglyphtounicode{tusmallkatakana}{30C3}
\pdfglyphtounicode{tusmallkatakanahalfwidth}{FF6F}
\pdfglyphtounicode{twelvecircle}{246B}
\pdfglyphtounicode{twelveparen}{247F}
\pdfglyphtounicode{twelveperiod}{2493}
\pdfglyphtounicode{twelveroman}{217B}
\pdfglyphtounicode{twentycircle}{2473}
\pdfglyphtounicode{twentyhangzhou}{5344}
\pdfglyphtounicode{twentyparen}{2487}
\pdfglyphtounicode{twentyperiod}{249B}
\pdfglyphtounicode{two}{0032}
\pdfglyphtounicode{twoarabic}{0662}
\pdfglyphtounicode{twobengali}{09E8}
\pdfglyphtounicode{twocircle}{2461}
\pdfglyphtounicode{twocircleinversesansserif}{278B}
\pdfglyphtounicode{twodeva}{0968}
\pdfglyphtounicode{twodotenleader}{2025}
\pdfglyphtounicode{twodotleader}{2025}
\pdfglyphtounicode{twodotleadervertical}{FE30}
\pdfglyphtounicode{twogujarati}{0AE8}
\pdfglyphtounicode{twogurmukhi}{0A68}
\pdfglyphtounicode{twohackarabic}{0662}
\pdfglyphtounicode{twohangzhou}{3022}
\pdfglyphtounicode{twoideographicparen}{3221}
\pdfglyphtounicode{twoinferior}{2082}
\pdfglyphtounicode{twomonospace}{FF12}
\pdfglyphtounicode{twonumeratorbengali}{09F5}
\pdfglyphtounicode{twooldstyle}{F732}
\pdfglyphtounicode{twoparen}{2475}
\pdfglyphtounicode{twoperiod}{2489}
\pdfglyphtounicode{twopersian}{06F2}
\pdfglyphtounicode{tworoman}{2171}
\pdfglyphtounicode{twostroke}{01BB}
\pdfglyphtounicode{twosuperior}{00B2}
\pdfglyphtounicode{twothai}{0E52}
\pdfglyphtounicode{twothirds}{2154}
\pdfglyphtounicode{u}{0075}
\pdfglyphtounicode{uacute}{00FA}
\pdfglyphtounicode{ubar}{0289}
\pdfglyphtounicode{ubengali}{0989}
\pdfglyphtounicode{ubopomofo}{3128}
\pdfglyphtounicode{ubreve}{016D}
\pdfglyphtounicode{ucaron}{01D4}
\pdfglyphtounicode{ucircle}{24E4}
\pdfglyphtounicode{ucircumflex}{00FB}
\pdfglyphtounicode{ucircumflexbelow}{1E77}
\pdfglyphtounicode{ucyrillic}{0443}
\pdfglyphtounicode{udattadeva}{0951}
\pdfglyphtounicode{udblacute}{0171}
\pdfglyphtounicode{udblgrave}{0215}
\pdfglyphtounicode{udeva}{0909}
\pdfglyphtounicode{udieresis}{00FC}
\pdfglyphtounicode{udieresisacute}{01D8}
\pdfglyphtounicode{udieresisbelow}{1E73}
\pdfglyphtounicode{udieresiscaron}{01DA}
\pdfglyphtounicode{udieresiscyrillic}{04F1}
\pdfglyphtounicode{udieresisgrave}{01DC}
\pdfglyphtounicode{udieresismacron}{01D6}
\pdfglyphtounicode{udotbelow}{1EE5}
\pdfglyphtounicode{ugrave}{00F9}
\pdfglyphtounicode{ugujarati}{0A89}
\pdfglyphtounicode{ugurmukhi}{0A09}
\pdfglyphtounicode{uhiragana}{3046}
\pdfglyphtounicode{uhookabove}{1EE7}
\pdfglyphtounicode{uhorn}{01B0}
\pdfglyphtounicode{uhornacute}{1EE9}
\pdfglyphtounicode{uhorndotbelow}{1EF1}
\pdfglyphtounicode{uhorngrave}{1EEB}
\pdfglyphtounicode{uhornhookabove}{1EED}
\pdfglyphtounicode{uhorntilde}{1EEF}
\pdfglyphtounicode{uhungarumlaut}{0171}
\pdfglyphtounicode{uhungarumlautcyrillic}{04F3}
\pdfglyphtounicode{uinvertedbreve}{0217}
\pdfglyphtounicode{ukatakana}{30A6}
\pdfglyphtounicode{ukatakanahalfwidth}{FF73}
\pdfglyphtounicode{ukcyrillic}{0479}
\pdfglyphtounicode{ukorean}{315C}
\pdfglyphtounicode{umacron}{016B}
\pdfglyphtounicode{umacroncyrillic}{04EF}
\pdfglyphtounicode{umacrondieresis}{1E7B}
\pdfglyphtounicode{umatragurmukhi}{0A41}
\pdfglyphtounicode{umonospace}{FF55}
\pdfglyphtounicode{underscore}{005F}
\pdfglyphtounicode{underscoredbl}{2017}
\pdfglyphtounicode{underscoremonospace}{FF3F}
\pdfglyphtounicode{underscorevertical}{FE33}
\pdfglyphtounicode{underscorewavy}{FE4F}
\pdfglyphtounicode{union}{222A}
\pdfglyphtounicode{universal}{2200}
\pdfglyphtounicode{uogonek}{0173}
\pdfglyphtounicode{uparen}{24B0}
\pdfglyphtounicode{upblock}{2580}
\pdfglyphtounicode{upperdothebrew}{05C4}
\pdfglyphtounicode{upsilon}{03C5}
\pdfglyphtounicode{upsilondieresis}{03CB}
\pdfglyphtounicode{upsilondieresistonos}{03B0}
\pdfglyphtounicode{upsilonlatin}{028A}
\pdfglyphtounicode{upsilontonos}{03CD}
\pdfglyphtounicode{uptackbelowcmb}{031D}
\pdfglyphtounicode{uptackmod}{02D4}
\pdfglyphtounicode{uragurmukhi}{0A73}
\pdfglyphtounicode{uring}{016F}
\pdfglyphtounicode{ushortcyrillic}{045E}
\pdfglyphtounicode{usmallhiragana}{3045}
\pdfglyphtounicode{usmallkatakana}{30A5}
\pdfglyphtounicode{usmallkatakanahalfwidth}{FF69}
\pdfglyphtounicode{ustraightcyrillic}{04AF}
\pdfglyphtounicode{ustraightstrokecyrillic}{04B1}
\pdfglyphtounicode{utilde}{0169}
\pdfglyphtounicode{utildeacute}{1E79}
\pdfglyphtounicode{utildebelow}{1E75}
\pdfglyphtounicode{uubengali}{098A}
\pdfglyphtounicode{uudeva}{090A}
\pdfglyphtounicode{uugujarati}{0A8A}
\pdfglyphtounicode{uugurmukhi}{0A0A}
\pdfglyphtounicode{uumatragurmukhi}{0A42}
\pdfglyphtounicode{uuvowelsignbengali}{09C2}
\pdfglyphtounicode{uuvowelsigndeva}{0942}
\pdfglyphtounicode{uuvowelsigngujarati}{0AC2}
\pdfglyphtounicode{uvowelsignbengali}{09C1}
\pdfglyphtounicode{uvowelsigndeva}{0941}
\pdfglyphtounicode{uvowelsigngujarati}{0AC1}
\pdfglyphtounicode{v}{0076}
\pdfglyphtounicode{vadeva}{0935}
\pdfglyphtounicode{vagujarati}{0AB5}
\pdfglyphtounicode{vagurmukhi}{0A35}
\pdfglyphtounicode{vakatakana}{30F7}
\pdfglyphtounicode{vav}{05D5}
\pdfglyphtounicode{vavdagesh}{FB35}
\pdfglyphtounicode{vavdagesh65}{FB35}
\pdfglyphtounicode{vavdageshhebrew}{FB35}
\pdfglyphtounicode{vavhebrew}{05D5}
\pdfglyphtounicode{vavholam}{FB4B}
\pdfglyphtounicode{vavholamhebrew}{FB4B}
\pdfglyphtounicode{vavvavhebrew}{05F0}
\pdfglyphtounicode{vavyodhebrew}{05F1}
\pdfglyphtounicode{vcircle}{24E5}
\pdfglyphtounicode{vdotbelow}{1E7F}
\pdfglyphtounicode{vecyrillic}{0432}
\pdfglyphtounicode{veharabic}{06A4}
\pdfglyphtounicode{vehfinalarabic}{FB6B}
\pdfglyphtounicode{vehinitialarabic}{FB6C}
\pdfglyphtounicode{vehmedialarabic}{FB6D}
\pdfglyphtounicode{vekatakana}{30F9}
\pdfglyphtounicode{venus}{2640}
\pdfglyphtounicode{verticalbar}{007C}
\pdfglyphtounicode{verticallineabovecmb}{030D}
\pdfglyphtounicode{verticallinebelowcmb}{0329}
\pdfglyphtounicode{verticallinelowmod}{02CC}
\pdfglyphtounicode{verticallinemod}{02C8}
\pdfglyphtounicode{vewarmenian}{057E}
\pdfglyphtounicode{vhook}{028B}
\pdfglyphtounicode{vikatakana}{30F8}
\pdfglyphtounicode{viramabengali}{09CD}
\pdfglyphtounicode{viramadeva}{094D}
\pdfglyphtounicode{viramagujarati}{0ACD}
\pdfglyphtounicode{visargabengali}{0983}
\pdfglyphtounicode{visargadeva}{0903}
\pdfglyphtounicode{visargagujarati}{0A83}
\pdfglyphtounicode{vmonospace}{FF56}
\pdfglyphtounicode{voarmenian}{0578}
\pdfglyphtounicode{voicediterationhiragana}{309E}
\pdfglyphtounicode{voicediterationkatakana}{30FE}
\pdfglyphtounicode{voicedmarkkana}{309B}
\pdfglyphtounicode{voicedmarkkanahalfwidth}{FF9E}
\pdfglyphtounicode{vokatakana}{30FA}
\pdfglyphtounicode{vparen}{24B1}
\pdfglyphtounicode{vtilde}{1E7D}
\pdfglyphtounicode{vturned}{028C}
\pdfglyphtounicode{vuhiragana}{3094}
\pdfglyphtounicode{vukatakana}{30F4}
\pdfglyphtounicode{w}{0077}
\pdfglyphtounicode{wacute}{1E83}
\pdfglyphtounicode{waekorean}{3159}
\pdfglyphtounicode{wahiragana}{308F}
\pdfglyphtounicode{wakatakana}{30EF}
\pdfglyphtounicode{wakatakanahalfwidth}{FF9C}
\pdfglyphtounicode{wakorean}{3158}
\pdfglyphtounicode{wasmallhiragana}{308E}
\pdfglyphtounicode{wasmallkatakana}{30EE}
\pdfglyphtounicode{wattosquare}{3357}
\pdfglyphtounicode{wavedash}{301C}
\pdfglyphtounicode{wavyunderscorevertical}{FE34}
\pdfglyphtounicode{wawarabic}{0648}
\pdfglyphtounicode{wawfinalarabic}{FEEE}
\pdfglyphtounicode{wawhamzaabovearabic}{0624}
\pdfglyphtounicode{wawhamzaabovefinalarabic}{FE86}
\pdfglyphtounicode{wbsquare}{33DD}
\pdfglyphtounicode{wcircle}{24E6}
\pdfglyphtounicode{wcircumflex}{0175}
\pdfglyphtounicode{wdieresis}{1E85}
\pdfglyphtounicode{wdotaccent}{1E87}
\pdfglyphtounicode{wdotbelow}{1E89}
\pdfglyphtounicode{wehiragana}{3091}
\pdfglyphtounicode{weierstrass}{2118}
\pdfglyphtounicode{wekatakana}{30F1}
\pdfglyphtounicode{wekorean}{315E}
\pdfglyphtounicode{weokorean}{315D}
\pdfglyphtounicode{wgrave}{1E81}
\pdfglyphtounicode{whitebullet}{25E6}
\pdfglyphtounicode{whitecircle}{25CB}
\pdfglyphtounicode{whitecircleinverse}{25D9}
\pdfglyphtounicode{whitecornerbracketleft}{300E}
\pdfglyphtounicode{whitecornerbracketleftvertical}{FE43}
\pdfglyphtounicode{whitecornerbracketright}{300F}
\pdfglyphtounicode{whitecornerbracketrightvertical}{FE44}
\pdfglyphtounicode{whitediamond}{25C7}
\pdfglyphtounicode{whitediamondcontainingblacksmalldiamond}{25C8}
\pdfglyphtounicode{whitedownpointingsmalltriangle}{25BF}
\pdfglyphtounicode{whitedownpointingtriangle}{25BD}
\pdfglyphtounicode{whiteleftpointingsmalltriangle}{25C3}
\pdfglyphtounicode{whiteleftpointingtriangle}{25C1}
\pdfglyphtounicode{whitelenticularbracketleft}{3016}
\pdfglyphtounicode{whitelenticularbracketright}{3017}
\pdfglyphtounicode{whiterightpointingsmalltriangle}{25B9}
\pdfglyphtounicode{whiterightpointingtriangle}{25B7}
\pdfglyphtounicode{whitesmallsquare}{25AB}
\pdfglyphtounicode{whitesmilingface}{263A}
\pdfglyphtounicode{whitesquare}{25A1}
\pdfglyphtounicode{whitestar}{2606}
\pdfglyphtounicode{whitetelephone}{260F}
\pdfglyphtounicode{whitetortoiseshellbracketleft}{3018}
\pdfglyphtounicode{whitetortoiseshellbracketright}{3019}
\pdfglyphtounicode{whiteuppointingsmalltriangle}{25B5}
\pdfglyphtounicode{whiteuppointingtriangle}{25B3}
\pdfglyphtounicode{wihiragana}{3090}
\pdfglyphtounicode{wikatakana}{30F0}
\pdfglyphtounicode{wikorean}{315F}
\pdfglyphtounicode{wmonospace}{FF57}
\pdfglyphtounicode{wohiragana}{3092}
\pdfglyphtounicode{wokatakana}{30F2}
\pdfglyphtounicode{wokatakanahalfwidth}{FF66}
\pdfglyphtounicode{won}{20A9}
\pdfglyphtounicode{wonmonospace}{FFE6}
\pdfglyphtounicode{wowaenthai}{0E27}
\pdfglyphtounicode{wparen}{24B2}
\pdfglyphtounicode{wring}{1E98}
\pdfglyphtounicode{wsuperior}{02B7}
\pdfglyphtounicode{wturned}{028D}
\pdfglyphtounicode{wynn}{01BF}
\pdfglyphtounicode{x}{0078}
\pdfglyphtounicode{xabovecmb}{033D}
\pdfglyphtounicode{xbopomofo}{3112}
\pdfglyphtounicode{xcircle}{24E7}
\pdfglyphtounicode{xdieresis}{1E8D}
\pdfglyphtounicode{xdotaccent}{1E8B}
\pdfglyphtounicode{xeharmenian}{056D}
\pdfglyphtounicode{xi}{03BE}
\pdfglyphtounicode{xmonospace}{FF58}
\pdfglyphtounicode{xparen}{24B3}
\pdfglyphtounicode{xsuperior}{02E3}
\pdfglyphtounicode{y}{0079}
\pdfglyphtounicode{yaadosquare}{334E}
\pdfglyphtounicode{yabengali}{09AF}
\pdfglyphtounicode{yacute}{00FD}
\pdfglyphtounicode{yadeva}{092F}
\pdfglyphtounicode{yaekorean}{3152}
\pdfglyphtounicode{yagujarati}{0AAF}
\pdfglyphtounicode{yagurmukhi}{0A2F}
\pdfglyphtounicode{yahiragana}{3084}
\pdfglyphtounicode{yakatakana}{30E4}
\pdfglyphtounicode{yakatakanahalfwidth}{FF94}
\pdfglyphtounicode{yakorean}{3151}
\pdfglyphtounicode{yamakkanthai}{0E4E}
\pdfglyphtounicode{yasmallhiragana}{3083}
\pdfglyphtounicode{yasmallkatakana}{30E3}
\pdfglyphtounicode{yasmallkatakanahalfwidth}{FF6C}
\pdfglyphtounicode{yatcyrillic}{0463}
\pdfglyphtounicode{ycircle}{24E8}
\pdfglyphtounicode{ycircumflex}{0177}
\pdfglyphtounicode{ydieresis}{00FF}
\pdfglyphtounicode{ydotaccent}{1E8F}
\pdfglyphtounicode{ydotbelow}{1EF5}
\pdfglyphtounicode{yeharabic}{064A}
\pdfglyphtounicode{yehbarreearabic}{06D2}
\pdfglyphtounicode{yehbarreefinalarabic}{FBAF}
\pdfglyphtounicode{yehfinalarabic}{FEF2}
\pdfglyphtounicode{yehhamzaabovearabic}{0626}
\pdfglyphtounicode{yehhamzaabovefinalarabic}{FE8A}
\pdfglyphtounicode{yehhamzaaboveinitialarabic}{FE8B}
\pdfglyphtounicode{yehhamzaabovemedialarabic}{FE8C}
\pdfglyphtounicode{yehinitialarabic}{FEF3}
\pdfglyphtounicode{yehmedialarabic}{FEF4}
\pdfglyphtounicode{yehmeeminitialarabic}{FCDD}
\pdfglyphtounicode{yehmeemisolatedarabic}{FC58}
\pdfglyphtounicode{yehnoonfinalarabic}{FC94}
\pdfglyphtounicode{yehthreedotsbelowarabic}{06D1}
\pdfglyphtounicode{yekorean}{3156}
\pdfglyphtounicode{yen}{00A5}
\pdfglyphtounicode{yenmonospace}{FFE5}
\pdfglyphtounicode{yeokorean}{3155}
\pdfglyphtounicode{yeorinhieuhkorean}{3186}
\pdfglyphtounicode{yerahbenyomohebrew}{05AA}
\pdfglyphtounicode{yerahbenyomolefthebrew}{05AA}
\pdfglyphtounicode{yericyrillic}{044B}
\pdfglyphtounicode{yerudieresiscyrillic}{04F9}
\pdfglyphtounicode{yesieungkorean}{3181}
\pdfglyphtounicode{yesieungpansioskorean}{3183}
\pdfglyphtounicode{yesieungsioskorean}{3182}
\pdfglyphtounicode{yetivhebrew}{059A}
\pdfglyphtounicode{ygrave}{1EF3}
\pdfglyphtounicode{yhook}{01B4}
\pdfglyphtounicode{yhookabove}{1EF7}
\pdfglyphtounicode{yiarmenian}{0575}
\pdfglyphtounicode{yicyrillic}{0457}
\pdfglyphtounicode{yikorean}{3162}
\pdfglyphtounicode{yinyang}{262F}
\pdfglyphtounicode{yiwnarmenian}{0582}
\pdfglyphtounicode{ymonospace}{FF59}
\pdfglyphtounicode{yod}{05D9}
\pdfglyphtounicode{yoddagesh}{FB39}
\pdfglyphtounicode{yoddageshhebrew}{FB39}
\pdfglyphtounicode{yodhebrew}{05D9}
\pdfglyphtounicode{yodyodhebrew}{05F2}
\pdfglyphtounicode{yodyodpatahhebrew}{FB1F}
\pdfglyphtounicode{yohiragana}{3088}
\pdfglyphtounicode{yoikorean}{3189}
\pdfglyphtounicode{yokatakana}{30E8}
\pdfglyphtounicode{yokatakanahalfwidth}{FF96}
\pdfglyphtounicode{yokorean}{315B}
\pdfglyphtounicode{yosmallhiragana}{3087}
\pdfglyphtounicode{yosmallkatakana}{30E7}
\pdfglyphtounicode{yosmallkatakanahalfwidth}{FF6E}
\pdfglyphtounicode{yotgreek}{03F3}
\pdfglyphtounicode{yoyaekorean}{3188}
\pdfglyphtounicode{yoyakorean}{3187}
\pdfglyphtounicode{yoyakthai}{0E22}
\pdfglyphtounicode{yoyingthai}{0E0D}
\pdfglyphtounicode{yparen}{24B4}
\pdfglyphtounicode{ypogegrammeni}{037A}
\pdfglyphtounicode{ypogegrammenigreekcmb}{0345}
\pdfglyphtounicode{yr}{01A6}
\pdfglyphtounicode{yring}{1E99}
\pdfglyphtounicode{ysuperior}{02B8}
\pdfglyphtounicode{ytilde}{1EF9}
\pdfglyphtounicode{yturned}{028E}
\pdfglyphtounicode{yuhiragana}{3086}
\pdfglyphtounicode{yuikorean}{318C}
\pdfglyphtounicode{yukatakana}{30E6}
\pdfglyphtounicode{yukatakanahalfwidth}{FF95}
\pdfglyphtounicode{yukorean}{3160}
\pdfglyphtounicode{yusbigcyrillic}{046B}
\pdfglyphtounicode{yusbigiotifiedcyrillic}{046D}
\pdfglyphtounicode{yuslittlecyrillic}{0467}
\pdfglyphtounicode{yuslittleiotifiedcyrillic}{0469}
\pdfglyphtounicode{yusmallhiragana}{3085}
\pdfglyphtounicode{yusmallkatakana}{30E5}
\pdfglyphtounicode{yusmallkatakanahalfwidth}{FF6D}
\pdfglyphtounicode{yuyekorean}{318B}
\pdfglyphtounicode{yuyeokorean}{318A}
\pdfglyphtounicode{yyabengali}{09DF}
\pdfglyphtounicode{yyadeva}{095F}
\pdfglyphtounicode{z}{007A}
\pdfglyphtounicode{zaarmenian}{0566}
\pdfglyphtounicode{zacute}{017A}
\pdfglyphtounicode{zadeva}{095B}
\pdfglyphtounicode{zagurmukhi}{0A5B}
\pdfglyphtounicode{zaharabic}{0638}
\pdfglyphtounicode{zahfinalarabic}{FEC6}
\pdfglyphtounicode{zahinitialarabic}{FEC7}
\pdfglyphtounicode{zahiragana}{3056}
\pdfglyphtounicode{zahmedialarabic}{FEC8}
\pdfglyphtounicode{zainarabic}{0632}
\pdfglyphtounicode{zainfinalarabic}{FEB0}
\pdfglyphtounicode{zakatakana}{30B6}
\pdfglyphtounicode{zaqefgadolhebrew}{0595}
\pdfglyphtounicode{zaqefqatanhebrew}{0594}
\pdfglyphtounicode{zarqahebrew}{0598}
\pdfglyphtounicode{zayin}{05D6}
\pdfglyphtounicode{zayindagesh}{FB36}
\pdfglyphtounicode{zayindageshhebrew}{FB36}
\pdfglyphtounicode{zayinhebrew}{05D6}
\pdfglyphtounicode{zbopomofo}{3117}
\pdfglyphtounicode{zcaron}{017E}
\pdfglyphtounicode{zcircle}{24E9}
\pdfglyphtounicode{zcircumflex}{1E91}
\pdfglyphtounicode{zcurl}{0291}
\pdfglyphtounicode{zdot}{017C}
\pdfglyphtounicode{zdotaccent}{017C}
\pdfglyphtounicode{zdotbelow}{1E93}
\pdfglyphtounicode{zecyrillic}{0437}
\pdfglyphtounicode{zedescendercyrillic}{0499}
\pdfglyphtounicode{zedieresiscyrillic}{04DF}
\pdfglyphtounicode{zehiragana}{305C}
\pdfglyphtounicode{zekatakana}{30BC}
\pdfglyphtounicode{zero}{0030}
\pdfglyphtounicode{zeroarabic}{0660}
\pdfglyphtounicode{zerobengali}{09E6}
\pdfglyphtounicode{zerodeva}{0966}
\pdfglyphtounicode{zerogujarati}{0AE6}
\pdfglyphtounicode{zerogurmukhi}{0A66}
\pdfglyphtounicode{zerohackarabic}{0660}
\pdfglyphtounicode{zeroinferior}{2080}
\pdfglyphtounicode{zeromonospace}{FF10}
\pdfglyphtounicode{zerooldstyle}{F730}
\pdfglyphtounicode{zeropersian}{06F0}
\pdfglyphtounicode{zerosuperior}{2070}
\pdfglyphtounicode{zerothai}{0E50}
\pdfglyphtounicode{zerowidthjoiner}{FEFF}
\pdfglyphtounicode{zerowidthnonjoiner}{200C}
\pdfglyphtounicode{zerowidthspace}{200B}
\pdfglyphtounicode{zeta}{03B6}
\pdfglyphtounicode{zhbopomofo}{3113}
\pdfglyphtounicode{zhearmenian}{056A}
\pdfglyphtounicode{zhebrevecyrillic}{04C2}
\pdfglyphtounicode{zhecyrillic}{0436}
\pdfglyphtounicode{zhedescendercyrillic}{0497}
\pdfglyphtounicode{zhedieresiscyrillic}{04DD}
\pdfglyphtounicode{zihiragana}{3058}
\pdfglyphtounicode{zikatakana}{30B8}
\pdfglyphtounicode{zinorhebrew}{05AE}
\pdfglyphtounicode{zlinebelow}{1E95}
\pdfglyphtounicode{zmonospace}{FF5A}
\pdfglyphtounicode{zohiragana}{305E}
\pdfglyphtounicode{zokatakana}{30BE}
\pdfglyphtounicode{zparen}{24B5}
\pdfglyphtounicode{zretroflexhook}{0290}
\pdfglyphtounicode{zstroke}{01B6}
\pdfglyphtounicode{zuhiragana}{305A}
\pdfglyphtounicode{zukatakana}{30BA}

\pdfglyphtounicode{a100}{275E}
\pdfglyphtounicode{a101}{2761}
\pdfglyphtounicode{a102}{2762}
\pdfglyphtounicode{a103}{2763}
\pdfglyphtounicode{a104}{2764}
\pdfglyphtounicode{a105}{2710}
\pdfglyphtounicode{a106}{2765}
\pdfglyphtounicode{a107}{2766}
\pdfglyphtounicode{a108}{2767}
\pdfglyphtounicode{a109}{2660}
\pdfglyphtounicode{a10}{2721}
\pdfglyphtounicode{a110}{2665}
\pdfglyphtounicode{a111}{2666}
\pdfglyphtounicode{a112}{2663}
\pdfglyphtounicode{a117}{2709}
\pdfglyphtounicode{a118}{2708}
\pdfglyphtounicode{a119}{2707}
\pdfglyphtounicode{a11}{261B}
\pdfglyphtounicode{a120}{2460}
\pdfglyphtounicode{a121}{2461}
\pdfglyphtounicode{a122}{2462}
\pdfglyphtounicode{a123}{2463}
\pdfglyphtounicode{a124}{2464}
\pdfglyphtounicode{a125}{2465}
\pdfglyphtounicode{a126}{2466}
\pdfglyphtounicode{a127}{2467}
\pdfglyphtounicode{a128}{2468}
\pdfglyphtounicode{a129}{2469}
\pdfglyphtounicode{a12}{261E}
\pdfglyphtounicode{a130}{2776}
\pdfglyphtounicode{a131}{2777}
\pdfglyphtounicode{a132}{2778}
\pdfglyphtounicode{a133}{2779}
\pdfglyphtounicode{a134}{277A}
\pdfglyphtounicode{a135}{277B}
\pdfglyphtounicode{a136}{277C}
\pdfglyphtounicode{a137}{277D}
\pdfglyphtounicode{a138}{277E}
\pdfglyphtounicode{a139}{277F}
\pdfglyphtounicode{a13}{270C}
\pdfglyphtounicode{a140}{2780}
\pdfglyphtounicode{a141}{2781}
\pdfglyphtounicode{a142}{2782}
\pdfglyphtounicode{a143}{2783}
\pdfglyphtounicode{a144}{2784}
\pdfglyphtounicode{a145}{2785}
\pdfglyphtounicode{a146}{2786}
\pdfglyphtounicode{a147}{2787}
\pdfglyphtounicode{a148}{2788}
\pdfglyphtounicode{a149}{2789}
\pdfglyphtounicode{a14}{270D}
\pdfglyphtounicode{a150}{278A}
\pdfglyphtounicode{a151}{278B}
\pdfglyphtounicode{a152}{278C}
\pdfglyphtounicode{a153}{278D}
\pdfglyphtounicode{a154}{278E}
\pdfglyphtounicode{a155}{278F}
\pdfglyphtounicode{a156}{2790}
\pdfglyphtounicode{a157}{2791}
\pdfglyphtounicode{a158}{2792}
\pdfglyphtounicode{a159}{2793}
\pdfglyphtounicode{a15}{270E}
\pdfglyphtounicode{a160}{2794}
\pdfglyphtounicode{a161}{2192}
\pdfglyphtounicode{a162}{27A3}
\pdfglyphtounicode{a163}{2194}
\pdfglyphtounicode{a164}{2195}
\pdfglyphtounicode{a165}{2799}
\pdfglyphtounicode{a166}{279B}
\pdfglyphtounicode{a167}{279C}
\pdfglyphtounicode{a168}{279D}
\pdfglyphtounicode{a169}{279E}
\pdfglyphtounicode{a16}{270F}
\pdfglyphtounicode{a170}{279F}
\pdfglyphtounicode{a171}{27A0}
\pdfglyphtounicode{a172}{27A1}
\pdfglyphtounicode{a173}{27A2}
\pdfglyphtounicode{a174}{27A4}
\pdfglyphtounicode{a175}{27A5}
\pdfglyphtounicode{a176}{27A6}
\pdfglyphtounicode{a177}{27A7}
\pdfglyphtounicode{a178}{27A8}
\pdfglyphtounicode{a179}{27A9}
\pdfglyphtounicode{a17}{2711}
\pdfglyphtounicode{a180}{27AB}
\pdfglyphtounicode{a181}{27AD}
\pdfglyphtounicode{a182}{27AF}
\pdfglyphtounicode{a183}{27B2}
\pdfglyphtounicode{a184}{27B3}
\pdfglyphtounicode{a185}{27B5}
\pdfglyphtounicode{a186}{27B8}
\pdfglyphtounicode{a187}{27BA}
\pdfglyphtounicode{a188}{27BB}
\pdfglyphtounicode{a189}{27BC}
\pdfglyphtounicode{a18}{2712}
\pdfglyphtounicode{a190}{27BD}
\pdfglyphtounicode{a191}{27BE}
\pdfglyphtounicode{a192}{279A}
\pdfglyphtounicode{a193}{27AA}
\pdfglyphtounicode{a194}{27B6}
\pdfglyphtounicode{a195}{27B9}
\pdfglyphtounicode{a196}{2798}
\pdfglyphtounicode{a197}{27B4}
\pdfglyphtounicode{a198}{27B7}
\pdfglyphtounicode{a199}{27AC}
\pdfglyphtounicode{a19}{2713}
\pdfglyphtounicode{a1}{2701}
\pdfglyphtounicode{a200}{27AE}
\pdfglyphtounicode{a201}{27B1}
\pdfglyphtounicode{a202}{2703}
\pdfglyphtounicode{a203}{2750}
\pdfglyphtounicode{a204}{2752}
\pdfglyphtounicode{a205}{276E}
\pdfglyphtounicode{a206}{2770}
\pdfglyphtounicode{a20}{2714}
\pdfglyphtounicode{a21}{2715}
\pdfglyphtounicode{a22}{2716}
\pdfglyphtounicode{a23}{2717}
\pdfglyphtounicode{a24}{2718}
\pdfglyphtounicode{a25}{2719}
\pdfglyphtounicode{a26}{271A}
\pdfglyphtounicode{a27}{271B}
\pdfglyphtounicode{a28}{271C}
\pdfglyphtounicode{a29}{2722}
\pdfglyphtounicode{a2}{2702}
\pdfglyphtounicode{a30}{2723}
\pdfglyphtounicode{a31}{2724}
\pdfglyphtounicode{a32}{2725}
\pdfglyphtounicode{a33}{2726}
\pdfglyphtounicode{a34}{2727}
\pdfglyphtounicode{a35}{2605}
\pdfglyphtounicode{a36}{2729}
\pdfglyphtounicode{a37}{272A}
\pdfglyphtounicode{a38}{272B}
\pdfglyphtounicode{a39}{272C}
\pdfglyphtounicode{a3}{2704}
\pdfglyphtounicode{a40}{272D}
\pdfglyphtounicode{a41}{272E}
\pdfglyphtounicode{a42}{272F}
\pdfglyphtounicode{a43}{2730}
\pdfglyphtounicode{a44}{2731}
\pdfglyphtounicode{a45}{2732}
\pdfglyphtounicode{a46}{2733}
\pdfglyphtounicode{a47}{2734}
\pdfglyphtounicode{a48}{2735}
\pdfglyphtounicode{a49}{2736}
\pdfglyphtounicode{a4}{260E}
\pdfglyphtounicode{a50}{2737}
\pdfglyphtounicode{a51}{2738}
\pdfglyphtounicode{a52}{2739}
\pdfglyphtounicode{a53}{273A}
\pdfglyphtounicode{a54}{273B}
\pdfglyphtounicode{a55}{273C}
\pdfglyphtounicode{a56}{273D}
\pdfglyphtounicode{a57}{273E}
\pdfglyphtounicode{a58}{273F}
\pdfglyphtounicode{a59}{2740}
\pdfglyphtounicode{a5}{2706}
\pdfglyphtounicode{a60}{2741}
\pdfglyphtounicode{a61}{2742}
\pdfglyphtounicode{a62}{2743}
\pdfglyphtounicode{a63}{2744}
\pdfglyphtounicode{a64}{2745}
\pdfglyphtounicode{a65}{2746}
\pdfglyphtounicode{a66}{2747}
\pdfglyphtounicode{a67}{2748}
\pdfglyphtounicode{a68}{2749}
\pdfglyphtounicode{a69}{274A}
\pdfglyphtounicode{a6}{271D}
\pdfglyphtounicode{a70}{274B}
\pdfglyphtounicode{a71}{25CF}
\pdfglyphtounicode{a72}{274D}
\pdfglyphtounicode{a73}{25A0}
\pdfglyphtounicode{a74}{274F}
\pdfglyphtounicode{a75}{2751}
\pdfglyphtounicode{a76}{25B2}
\pdfglyphtounicode{a77}{25BC}
\pdfglyphtounicode{a78}{25C6}
\pdfglyphtounicode{a79}{2756}
\pdfglyphtounicode{a7}{271E}
\pdfglyphtounicode{a81}{25D7}
\pdfglyphtounicode{a82}{2758}
\pdfglyphtounicode{a83}{2759}
\pdfglyphtounicode{a84}{275A}
\pdfglyphtounicode{a85}{276F}
\pdfglyphtounicode{a86}{2771}
\pdfglyphtounicode{a87}{2772}
\pdfglyphtounicode{a88}{2773}
\pdfglyphtounicode{a89}{2768}
\pdfglyphtounicode{a8}{271F}
\pdfglyphtounicode{a90}{2769}
\pdfglyphtounicode{a91}{276C}
\pdfglyphtounicode{a92}{276D}
\pdfglyphtounicode{a93}{276A}
\pdfglyphtounicode{a94}{276B}
\pdfglyphtounicode{a95}{2774}
\pdfglyphtounicode{a96}{2775}
\pdfglyphtounicode{a97}{275B}
\pdfglyphtounicode{a98}{275C}
\pdfglyphtounicode{a99}{275D}
\pdfglyphtounicode{a9}{2720}

\pdfglyphtounicode{Ifractur}{2111}
\pdfglyphtounicode{FFsmall}{D804}
\pdfglyphtounicode{FFIsmall}{D807}
\pdfglyphtounicode{FFLsmall}{D808}
\pdfglyphtounicode{FIsmall}{D805}
\pdfglyphtounicode{FLsmall}{D806}
\pdfglyphtounicode{Germandbls}{D800}
\pdfglyphtounicode{Germandblssmall}{D803}
\pdfglyphtounicode{Ng}{014A}
\pdfglyphtounicode{Rfractur}{211C}
\pdfglyphtounicode{SS}{D800}
\pdfglyphtounicode{SSsmall}{D803}
\pdfglyphtounicode{altselector}{D802}
\pdfglyphtounicode{angbracketleft}{27E8}
\pdfglyphtounicode{angbracketright}{27E9}
\pdfglyphtounicode{arrowbothv}{2195}
\pdfglyphtounicode{arrowdblbothv}{21D5}
\pdfglyphtounicode{arrowleftbothalf}{21BD}
\pdfglyphtounicode{arrowlefttophalf}{21BC}
\pdfglyphtounicode{arrownortheast}{2197}
\pdfglyphtounicode{arrownorthwest}{2196}
\pdfglyphtounicode{arrowrightbothalf}{21C1}
\pdfglyphtounicode{arrowrighttophalf}{21C0}
\pdfglyphtounicode{arrowsoutheast}{2198}
\pdfglyphtounicode{arrowsouthwest}{2199}
\pdfglyphtounicode{ascendercompwordmark}{D80A}
\pdfglyphtounicode{asteriskcentered}{2217}
\pdfglyphtounicode{bardbl}{2225}
\pdfglyphtounicode{capitalcompwordmark}{D809}
\pdfglyphtounicode{ceilingleft}{2308}
\pdfglyphtounicode{ceilingright}{2309}
\pdfglyphtounicode{circlecopyrt}{20DD}
\pdfglyphtounicode{circledivide}{2298}
\pdfglyphtounicode{circledot}{2299}
\pdfglyphtounicode{circleminus}{2296}
\pdfglyphtounicode{coproduct}{2A3F}
\pdfglyphtounicode{cwm}{200C}
\pdfglyphtounicode{dblbracketleft}{27E6}
\pdfglyphtounicode{dblbracketright}{27E7}
\pdfglyphtounicode{emptyslot}{D801}
\pdfglyphtounicode{epsilon1}{03F5}
\pdfglyphtounicode{equivasymptotic}{224D}
\pdfglyphtounicode{flat}{266D}
\pdfglyphtounicode{floorleft}{230A}
\pdfglyphtounicode{floorright}{230B}
\pdfglyphtounicode{follows}{227B}
\pdfglyphtounicode{followsequal}{227D}
\pdfglyphtounicode{greatermuch}{226B}
\pdfglyphtounicode{interrobang}{203D}
\pdfglyphtounicode{interrobangdown}{D80B}
\pdfglyphtounicode{intersectionsq}{2293}
\pdfglyphtounicode{latticetop}{22A4}
\pdfglyphtounicode{lessmuch}{226A}
\pdfglyphtounicode{lscript}{2113}
\pdfglyphtounicode{natural}{266E}
\pdfglyphtounicode{negationslash}{0338}
\pdfglyphtounicode{ng}{014B}
\pdfglyphtounicode{owner}{220B}
\pdfglyphtounicode{pertenthousand}{2031}
\pdfglyphtounicode{pi1}{03D6}
\pdfglyphtounicode{precedesequal}{227C}
\pdfglyphtounicode{prime}{2032}
\pdfglyphtounicode{rho1}{03F1}
\pdfglyphtounicode{ringfitted}{D80D}
\pdfglyphtounicode{sharp}{266F}
\pdfglyphtounicode{similarequal}{2243}
\pdfglyphtounicode{slurabove}{2322}
\pdfglyphtounicode{slurbelow}{2323}
\pdfglyphtounicode{star}{22C6}
\pdfglyphtounicode{subsetsqequal}{2291}
\pdfglyphtounicode{supersetsqequal}{2292}
\pdfglyphtounicode{triangle}{25B3}
\pdfglyphtounicode{triangleinv}{25BD}
\pdfglyphtounicode{triangleleft}{25B9}
\pdfglyphtounicode{triangleright}{25C3}
\pdfglyphtounicode{turnstileleft}{22A2}
\pdfglyphtounicode{turnstileright}{22A3}
\pdfglyphtounicode{twelveudash}{D80C}
\pdfglyphtounicode{unionmulti}{228E}
\pdfglyphtounicode{unionsq}{2294}
\pdfglyphtounicode{vector}{20D7}
\pdfglyphtounicode{visualspace}{2423}
\pdfglyphtounicode{wreathproduct}{2240}
\pdfglyphtounicode{Dbar}{0110}
\pdfglyphtounicode{compwordmark}{200C}
\pdfglyphtounicode{dbar}{0111}
\pdfglyphtounicode{rangedash}{2013}
\pdfglyphtounicode{hyphenchar}{002D}
\pdfglyphtounicode{punctdash}{2014}
\pdfglyphtounicode{visiblespace}{2423}

\pdfglyphtounicode{ff}{0066 0066}
\pdfglyphtounicode{fi}{0066 0069}
\pdfglyphtounicode{fl}{0066 006C}
\pdfglyphtounicode{ffi}{0066 0066 0069}
\pdfglyphtounicode{ffl}{0066 0066 006C}
\pdfglyphtounicode{IJ}{0049 004A}
\pdfglyphtounicode{ij}{0069 006A}
\pdfglyphtounicode{longs}{0073}

\documentclass{article}
\usepackage[utf8]{inputenc}
\usepackage{longtable}
\usepackage{paralist}
\usepackage[UKenglish,english]{babel}
\usepackage{amsmath}
\usepackage[amsmath,thmmarks,hyperref]{ntheorem}
\usepackage{amssymb}
\usepackage{graphicx}
\usepackage{hyperref}
\newcommand{\proof}{{\bf Proof:  }}

\newcommand{\remark}{{\bf Remark:   }}

\newcommand{\examples}{{\bf Examples:  }}
\newcommand{\dimv}{\underline{\dim}}

\newcommand{\hb}{\newline\hspace*{\fill}$\Box$}

\newtheorem{theorem}{Theorem}[section]
\newtheorem{lemma}[theorem]{Lemma}
\newtheorem{definition}[theorem]{Definition}
\newtheorem{proposition}[theorem]{Proposition}
\newtheorem{corollary}[theorem]{Corollary}

\theoremheaderfont{\bfseries}\theorembodyfont{\normalfont\upshape}\newtheorem{examp}[theorem]{Example}\theoremsymbol{\ensuremath{\Box}}\theoremstyle{nonumberplain}\renewtheorem{proof}{Proof}

\DeclareMathOperator{\Natural}{\mathbf{N}}

\DeclareMathOperator{\Complex}{\mathbf{C}}

\DeclareMathOperator{\Hilb}{H}
\renewcommand{\succ}{\mathrm{succ}}
\newcommand{\ohne}{\setminus}

\DeclareMathOperator{\coloneq}{:=}

\begin{document}
\author{{Johannes Engel and Markus Reineke}\\ Fachbereich C - Mathematik\\ Bergische Universit\"at Wuppertal\\ Gaussstr. 20\\ D - 42097 Wuppertal\\ {engel@math.uni-wuppertal.de, reineke@math.uni-wuppertal.de}}
\title{Smooth models of quiver moduli}
\date{}

\maketitle

\begin{abstract}
For any moduli space of stable representations of quivers, certain smooth varieties, compactifying projective space fibrations over the moduli space, are constructed. The boundary of this compactification is analyzed. Explicit formulas for the Betti numbers of the smooth models are derived. In the case of moduli of simple representations, explicit cell decompositions of the smooth models are constructed.
\end{abstract}

\parindent0pt

\section{Introduction}\label{introduction}

Moduli spaces constructed using Geometric Invariant Theory \cite{Mum}, parametrizing appropriate stable objects in an abelian category up to isomorphism, are usually not projective, making the determination of their global topological and geometric invariants a difficult problem. In contrast, the standard compactification, parametrizing equivalence classes of semistable objects, is usually highly singular. Posing a moduli problem resulting in smooth projective moduli spaces is therefore a rather subtle, and in many cases unsolved, problem. The standard approach is to parametrize objects in the given category together with some additional structure, but to choose such structure in an appropriate way is in no way canonical, and depends, if possible at all, on a deep understanding of the particular moduli problem.\\[1ex]
Moduli spaces of representations of quivers \cite{King} form a particularly interesting testing ground for techniques of moduli theory: they are easily defined and parametrize basic linear algebra type objects, they behave quite analogously to moduli of vector bundles on curves or surfaces in many respects, and of course they are interesting in themselves, since they play a key role in approaching wild classification problems in representation theory \cite{RU}.\\[1ex]
The above mentioned general dichotomy between smooth non-projective and singular projective moduli also applies to quiver moduli as soon as the basic discrete invariant of a quiver representation, its dimension vector, is not coprime.\\[2ex]
In this paper, we formulate and study a moduli problem closely related to the original problem of parametrizing stable representations of quivers up to isomorphism. Namely, we consider representations, together with an additional "framing datum", consisting of maps from given vector spaces to the quiver representation. The subtle point in the construction is to choose the correct notion of stability. Once this is found, the moduli spaces of stable such pairs, which are called smooth models here, parametrize semistable quiver representations, together with a map "avoiding the subspaces contradicting stability" (see Proposition \ref{cos} and Lemma \ref{cosplus}).\\[1ex]
The construction is inspired both by the idea of framing quiver data, prominent in H.~Nakajima's work on quiver varieties (see e.g. \cite{Nak}), and by the construction and study of Brauer-Severi type varieties in \cite{BSV} and \cite{VdB}. A similar construction in the context of moduli of vector bundles on curves is provided by the moduli of stable pairs \cite{Th}. Two special cases of the present construction were already studied by the second named author: the case of quivers without oriented cycles, with respect to trivial stability, in \cite{RF}, and the case of the multiple-loop quiver in \cite{RNCH}.\\[2ex]
The smooth models of quiver moduli resulting from our construction, if non-empty, are irreducible smooth varieties of known dimension, projective over the corresponding moduli space of semisimple representations (thus projective in case of quivers without oriented cycles), see Proposition \ref{basicgeomprop}. They admit a natural projective morphism to the moduli space of polystable representations, an analogue of the Hilbert-Chow morphism from Hilbert schemes to symmetric products. No general simple criterion for non-emptyness of the smooth models is known, except for a recursive one; but for "large enough" framing data, the smooth models are always non-empty (Lemma \ref{nonempty}). A notable exception is the case of trivial stability, for which an efficient criterion is given in Theorem \ref{criterionhilb}.\\[1ex] 
The fibres of the analogue of the Hilbert-Chow morphism can be explicitely described as nilpotent parts of smooth models for other quiver data, using a generalization of well-known Luna slice techniques (Theorem \ref{theoremfibres}). In particular, the fibres over the stable locus are always isomorphic to projective spaces of appropriate dimension. Furthermore, the analogue of the Hilbert-Chow morphism is a fibration (that is, locally trivial in the \'etale topology) over each Luna stratum of the moduli space of polystable representations. This allows us to define a stratification of the smooth models, whose generic stratum is given by a projective space fibration over the moduli of stable representations, and whose boundary is decomposed into fibrations with "known" fibres over the other Luna strata (Corollary \ref{strat}).\\[1ex]
Using Harder-Narasimhan techniques and a resolution of the occuring recursions, the Betti numbers in singular cohomology of quiver moduli in the coprime case were computed in \cite{RHNS}. Based on these formulas, two different formulas for the Betti numbers of smooth models are given. One formula (Theorem \ref{summationformula}) is given by an explicit summation, a variant of the main formula from \cite{RHNS}. The other formula (Theorem \ref{genfunbetti}) computes the Poincar\'e polynomials of all smooth models at the same time, by expressing their generating function as a quotient of two generating functions involving explicit rational functions already used in \cite{RHNS}. This result is established using Hall algebra techniques (Lemma \ref{keyidentity}), reminiscent of similar techniques for numbers of rational points of moduli of stable representations \cite{RC} and of quiver Grassmannians \cite{CR}.\\[1ex]
In the special case of trivial stability, the smooth models parametrize arbitrary representations together with a presentation as a factor of a projective representation, and we call them Hilbert schemes of path algebras of quivers. As mentioned above, we are able to give an explicit criterion for non-emptyness in this case. Furthermore, we derive from the general cohomology formulas a positive combinatorial formula for the Betti numbers, that is, a formula given by a weighted counting of certain combinatorial objects, namely, a restricted class of multipartitions.\\[1ex]
We give a conceptual explanation for this formula by constructing an explicit cell decomposition (Theorem \ref{Th:ZU}), generalizing a construction in \cite{RNCH}. The cells in this decomposition are naturally parametrized by certain types of forests (more precisely, certain subquivers of covering quivers), yielding a combinatorial formula for the Betti numbers in terms of such objects. The final result, Theorem \ref{Th:MPbij}, therefore establishes a direct combinatorial relation between the multipartitions and the forests relevant for our combinatorial formulas.\\[2ex]
The paper is organized as follows: in section \ref{recoll}, we first recall basic definitions and facts on quivers, their representations and on quiver moduli (subsection \ref{21}). We also recall some results from \cite{RHNS} which will be used in the following (subsection \ref{23}), thereby generalizing them to quivers possessing oriented cycles using a general purity result (subsection \ref{22}).\\
Section \ref{smoothmodels} is devoted to the construction of the smooth models by a framing process. The objects parametrized by the smooth models are described (Proposition \ref{basicgeomprop}), and basic geometric properties are discussed. The section ends with an illustration of the general construction by several examples, which will be studied in more detail subsequently.\\
In section \ref{fibres}, we first recall the Luna stratification of quiver moduli. This is used to define the stratification of the smooth models, whose geometric properties rely on the analysis of the fibres of the analogue of the Hilbert-Chow morphism in Theorem \ref{theoremfibres}.\\
After recalling some Hall algebra techniques, section \ref{cohomology} derives two formulas for Betti numbers of smooth models mentioned above by direct computations relying on the formulas of subsection \ref{23}, and illustrates the use of the formulas in two examples.\\
Section \ref{hilbertpath} applies the techniques of the previous sections to the case of Hilbert schemes of path algebras, the methods being of a more combinatorial flavour.\\
The cell decompositions of Hilbert schemes of path algebras are constructed explicitely in section \ref{sec:NCHilbert}. Section \ref{sec:MP} compares the relevant combinatorial concepts and gives a typical example.\\[2ex]
{\bf Acknowledgments:} The first author would like to thank the Graduiertenkolleg "Darstellungstheorie und ihre Anwendungen in Mathematik und Physik" at the Bergische Universität Wuppertal for the generous support.
The second author would like to thank R. Bocklandt, W. Crawley-Boevey, A. King, L. Le Bruyn, O. Schiffmann and G. Van de Weyer for helpful comments and interesting discussions on the material developed in this paper.

\section{Recollections on quiver moduli}\label{recoll}

\subsection{Definition of quiver moduli}\label{21}

Let $Q$ be a finite quiver with set of vertices $I$ and set of arrows $Q_1$, where an arrow $\alpha\in Q_1$ starting in $i\in I$ and ending in $j\in I$ will always be denoted by $\alpha:i\rightarrow j$. Denote by ${\bf Z}I$ the free abelian group generated by $I$, whose elements will be written as $d=\sum_{i\in I}d_ii$. On ${\bf Z}I$, we have the Euler form of $Q$ defined by $$\langle d,e\rangle=\sum_{i\in I}d_ie_i-\sum_{\alpha:i\rightarrow j}d_ie_j.$$
The subsemigroup ${\bf N}I$ of ${\bf Z}I$ will be viewed as the set of dimension vectors of representations of $Q$.\\[1ex]
Given a dimension vector $d\in{\bf N}I$, we fix complex vector spaces $M_i$ of dimension $d_i$ for any $i\in I$ and consider the affine space
$$R_d(Q):=\bigoplus_{\alpha:i\rightarrow j}{\rm Hom}(M_i,M_j),$$
on which the reductive algebraic group
$$G_d:=\prod_{i\in I}{\rm GL}(M_i)$$
acts via
$$(g_i)_i\cdot (M_\alpha)_\alpha:=(g_jM_\alpha g_i^{-1})_{\alpha:i\rightarrow j},$$
so that the orbits $\mathcal{O}_M$ correspond naturally to the isomorphism classes $[M]$ of representations of dimension vector $d$.\\[1ex]
Let $\Theta\in({\bf Q}I)^*$ be a linear form, called a stability condition, which will be fixed throughout, and let $\dim d\in{\bf N}$ be defined by $$\dim d:=\sum_{i\in I}d_i.$$
This allows to define the slope of a non-zero dimension vector by
$$\mu(d):=\frac{\Theta(d)}{\dim d}\in{\bf Q}.$$
For $s\in{\bf Q}$, we define $$({\bf N}I)_s:=\{0\not=d\in{\bf N}I\,:\, \mu(d)=s\}\cup\{0\},$$
a subsemigroup of ${\bf N}I$.\\[2ex]
We denote by ${\rm Rep}_{\bf C}(Q)$ the category of finite-dimensional complex representations of $Q$. The dimension vector of such a representation $$X=((X_i)_{i\in I},(X_\alpha:X_i\rightarrow X_j)_{\alpha:i\rightarrow j})$$ will be denoted by $\dimv X\in{\bf N}I$. Its slope $\mu(X)\in{\bf Q}$ is then defined as the slope $\mu(\dimv X)$ of its dimension vector. The representation $X$ is called {($\mu$-)stable} (resp.~($\mu$-)semistable) if $$\mu(U)<\mu(X)\mbox{ (resp. }\mu(U)\leq \mu(X))$$ for all non-zero proper subrepresentations $U$ of $X$. Moreover, the representation $X$ is called ($\mu$-)polystable if it is isomorphic to a direct sum of stable representations, all of which have the same slope.\\[1ex]
In the case $\Theta=0$, any representation is semistable, and the notions of stability, resp.~polystability, reduce to the notions of simplicity, resp.~semisimplicity. We define ${\rm Rep}_{\bf C}^s(Q)$ as the full subcategory of ${\rm Rep}_{\bf C}(Q)$ consisting of semistable representations of slope $s$. By simple properties of semistability (see \cite{RHNS}), this is an abelian subcategory, that is, it is closed under extensions, kernels and cokernels. Its simple objects are the stable representations of slope $s$.\\[1ex]
We denote by $R_d^{\rm sst}(Q)=R_d^{\rm \Theta-sst}(Q)$, resp.~$R_d^{\rm st}(Q)$, the open subset of the variety $R_d(Q)$ corresponding to semistable, resp.~stable, representations.\\[1ex]
By \cite{King}, there exists a smooth complex algebraic variety $M_d^{\rm st}(Q)$ parametrizing isomorphism classes of $\mu$-stable representations of $Q$ of dimension vector $d$, and a complex algebraic variety $M_d^{\rm sst}(Q)$ parametrizing isomorphism classes of $\mu$-polystable representations of $Q$ of dimension vector $d$ (or, equivalently, so-called $S$-equivalence classes of $\mu$-semistable representations of dimension vector $d$). The former is given as the geometric quotient of $R_d^{\rm st}(Q)$ by the action of $G_d$, whereas the latter is given as the algebraic quotient of $R_s^{\rm sst}(Q)$ by the action of $G_d$. The moduli space $M_d^{\rm sst}(Q)$ is therefore defined as the ${\bf Proj}$ of the ring of semi-invariants of $G_d$ in the coordinate ring of $R_d$ with respect to a character of $G_d$ defined via $\Theta$. The moduli space $M_d^{\rm st}(Q)$ is an open (but possibly empty) subset of $M_d^{\rm sst}(Q)$.\\[1ex]
In the special case $\Theta=0$, these moduli will be denoted by $M_d^{\rm simp}(Q)$ and $M_d^{\rm ssimp}(Q)$, thus parametrizing isomorphism classes of simple and semisimple representations, respectively. The variety $M_d^{\rm ssimp}(Q)$ is affine, since it is defined as the ${\bf Spec}$ of the ring of $G_d$-invariants in the coordinate ring of $R_d$. Again $M_d^{\rm simp}(Q)$ is an open, but possibly empty, subset. We denote by $$0:=\bigoplus_{i\in I}S_i^{d_i}\in M_d^{\rm ssimp}(Q)$$ the point corresponding to a canonical semisimple representation of dimension vector $d$, where $S_i$ denotes the one-dimensional representation associated to the vertex $i\in I$. The fibre $q^{-1}(0)$ of the quotient map $$q:R_d(Q)\rightarrow M_d^{\rm ssimp}(Q),$$
attaching to a representation the isomorphism class of its semisimplification, is denoted by $N_d(Q)$ and is called the nullcone in $R_d(Q)$. The points of the nullcone correspond to representations admitting a filtration with subquotients isomorphic to the $S_i$, or equivalently, representations such that the trace along all non-trivial oriented cycles vanishes. We call such representations nilpotent.\\[1ex]
By the above, the variety $M_d^{\rm sst}(Q)$ admits a projective morphism $$p:M_d^{\rm sst}(Q)\rightarrow M_d^{\rm ssimp}(Q),$$
again associating to a representation its semisimplification.\\[1ex]
The situation is summarized in the following diagram:
\begin{center}
\includegraphics{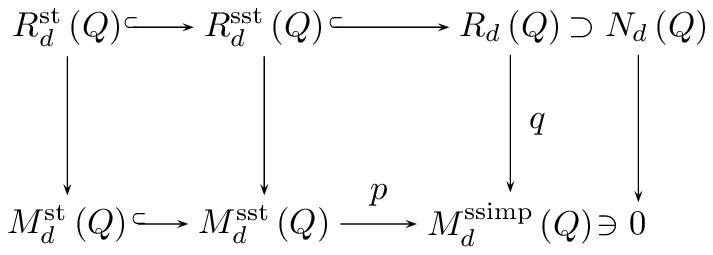}
\end{center}
%$$\begin{array}{ccccl}
%R_d^{\rm st}(Q)&\subset&R_d^{\rm sst}(Q)&\subset&R_d(Q)\supset N_d(Q)\\
%\downarrow&&\downarrow&&\downarrow q\\
%M_d^{\rm st}(Q)&\subset&M_d^{\rm sst}(Q)&\stackrel{p}{\rightarrow}&M_d^{\rm ssimp}(Q)\end{array}$$

In case the quiver $Q$ has no oriented cycles, the moduli space $M_d^{\rm ssimp}(Q)$ reduces to a single point, since all simple representations are one-dimensional, and correspond to the vertices of $Q$ up to isomorphism. Thus, $M_d^{\rm sst}(Q)$ is a projective variety in this case.\\[1ex]
Following \cite{RHNS}, we call $d$ coprime for $\Theta$ if $$\mu(e)\not=\mu(d)\mbox{ for all }0<e<d.$$
In particular, we then have ${\rm gcd}(d_i)_{i\in I}=1$, and this condition is sufficient for coprimality of $d$ at least for generic $\Theta$ by \cite{King}. For $d$ coprime for $\Theta$, we have $M_d^{\rm st}(Q)=M_d^{\rm sst}(Q)$ by definition, thus $M_d^{\rm st}(Q)$ is smooth and projective over $M_d^{\rm ssimp}(Q)$ in this case.\\[1ex]
There are two natural operations on stability conditions which are easily seen to respect the class of stable (resp.~semistable) representations: $\Theta$ can be multiplied by a positive rational number, and $\Theta$ can be translated by a multiple of the functional $\dim$. Using these operations, one can always assume without loss of generality that $$\Theta\in({\bf Z}I)^*\mbox{ and }\Theta(d)=0$$ 
for some given $d\in{\bf N}I$.\\[2ex]
We state the following criterion (see \cite[Corollary 3.5]{RHNS}) for non-emptyness of $M_d^{\rm sst}(Q)$ for future reference.

\begin{theorem}\label{criterion} The moduli space $M_d^{\rm sst}(Q)$ is non-empty if and only if there exists no no-trival decomposition $d=d^1+\ldots+d^s$ such that the following conditions hold:
\begin{description}
\item[a)] $M_{d^k}^{\rm sst}(Q)\not=\emptyset$ for all $k$,
\item[b)] $\mu(d^{1})>\ldots>\mu(d^k)$,
\item[c)] $\langle d^k,d^l\rangle=0$ for all $k<l$.
\end{description}
\end{theorem}

\subsection{Purity}\label{22}

Although our primary interest is in quiver moduli over the complex numbers, all the varieties considered so far can be defined over arbitrary fields: by \cite[section 6]{RC}, there exist schemes over ${\bf Z}$ whose base extensions to ${\bf C}$ are isomorphic to $M_d^{\rm sst}(Q)$ and $M_d^{\rm st}(Q)$, respectively. We can thus extend scalars to any field (in particular, to finite fields or their algebraic closures), which, by abuse of notation, will also be denoted by $M_d^{\rm sst}(Q)$ and $M_d^{\rm st}(Q)$, respectively. The same applies to the varieties of representations $R_d(Q)$, $R_d^{\rm sst}(Q)$, etc..~In particular, we can consider the number of rational points over finite fields ${\bf F}_q$ of these varieties, which will be denoted (again by abuse of notation) by $|M_d^{\rm st}(Q)({\bf F}_q)|$,         
etc.\\[2ex]
We need a general remark on purity of quiver moduli. Assume that a datum $(Q,d,\Theta)$ as before is given, and assume that $d$ is coprime for $\Theta$. We will prove that $M_d^{\rm st}(Q)$, viewed over an algebraic closure of a finite field, is not only smooth, but also cohomologically pure, although it is not projective for general $Q$. We adopt the technique of \cite[2.4.]{CBVdB}:

\begin{proposition}\cite[Proposition A.2]{CBVdB} Assume that Z is a smooth and quasi-projective variety over an algebraic closure $k$ of a finite field ${\bf F}_p$, such that $Z$ is defined over a finite extension field ${\bf F}_q$, and that there is
an action $\lambda:{\bf G}_m\times Z\rightarrow Z$ such that, for every $x\in Z$, the limit $\lim_{t\rightarrow 0}\lambda(t, x)$
exists. Assume in addition that $Z^{{\bf G}_m}$ is projective. Then $Z$ is cohomologically pure, that is, the eigenvalues of Frobenius acting on the $i$-th $\ell$-adic cohomology of $Z$ have absolute values $q^{i/2}$ for all $i$.
\end{proposition}

Application to the quiver setup yields the following:

\begin{proposition}\label{purity} If $d$ is coprime for $\Theta$, then $M_d^{\rm st}(Q)$ is pure.
\end{proposition}

\proof The multiplicative group ${\bf G}_m$ acts on $R_d(Q)$ by scaling the linear maps representing the arrows. This restricts to an action on $R_d^{\rm sst}(Q)$, since scaling a representation does not change its subrepresentations. Consequently, ${\bf G}_m$ acts on the moduli spaces $M_d^{\rm (s)st}(Q)$ and $M_d^{\rm (s)simp}(Q)$, and the projective morphism $p:M_d^{\rm sst}(Q)\rightarrow M_d^{\rm ssimp}(Q)$ is ${\bf G}_m$-equivariant. The invariant ring $k[R_d]^{G_d}$ is generated by traces along oriented cycles by \cite{LBP}, thus ${\bf G}_m$ is non-negatively graded by the weight spaces of the ${\bf G}_m$-action. Consequently, the point $0\in M_d^{\rm ssimp}(Q)={\rm Spec}(k[R_d]^{G_d})$ is the unique ${\bf G}_m$-fixed point, to which all points limit. Therefore, the ${\bf G}_m$-fixed points in the smooth quasi-projective variety $M_d^{\rm st}(Q)$ form a closed subvariety of the projective variety $p^{-1}(0)$, to which all of $M_d^{\rm st}(Q)$ limits. Application of the proposition above yields the result.\hb

\subsection{Cohomology}\label{23}

Next, we consider the Betti numbers of quiver moduli. For a complex variety $X$, we denote by
$$P_X(q):=\sum_{i\in{\bf Z}}\dim_{\bf Q}H^i(X,{\bf Q})q^{i/2}\in{\bf Z}[q^{1/2}]$$
its Poincar\'e polynomial in singular cohomology with rational coefficients (the half-powers of $q$ are reasonable since in the present situation, only varieties with vanishing odd cohomology will be proved to appear).

\begin{definition}\label{admissible} Given $(Q,d,\Theta)$ as before, define a rational function
$$P_d(q)=\sum_{d^*}(-1)^{s-1}q^{-\sum_{k<l}\langle d^l,d^k\rangle}\prod_{k=1}^s\frac{|R_{d^k}(Q)({\bf F}_q)|}{|G_{d^k}({\bf F}_q)|}\in{\bf Q}(q),$$
where the sum runs over all decompositions $d=d^1+\ldots+d^s$ into non-zero dimension vectors such that $$\mu(d^1+\ldots+d^k)>\mu(d)\mbox{ for all }k<s.$$
Such decompositions of $d$ will be called $\mu$-admissible.
\end{definition}

\begin{theorem}\label{Betti} Given $(Q,d,\Theta)$ such that $d$ is coprime for $\Theta$, we have
$$P_{M_d^{\rm st}(Q)}(q)=(q-1)\cdot P_d(q).$$
\end{theorem}

\proof In the case that $Q$ has no oriented cycles, this is proved in \cite[Theorem 6.7]{RHNS}, where it is derived via purity from the formula $$|M_d^{\rm st}(Q)({\bf F}_q)|=(q-1)\cdot P_d(Q)$$ for the number of rational points of $M_d^{\rm st}(Q)$ over ${\bf F}_q$. Since purity holds for general $Q$ by Proposition \ref{purity}, this proof generalizes.\hb

\remark If $d$ is not coprime for $\Theta$, the rational function $(q-1)\cdot P_d(q)$ is not a polynomial, and neither moduli $M_d^{\rm st}(Q)$ or $M_d^{\rm sst}(Q)$ is pure. Therefore, there is no obvious relation between $(q-1)\cdot P_d(Q)$ and the cohomology of $M_d^{\rm (s)st}(Q)$. Nevertheless, the functions $P_d(Q)$ will enter (see Theorem \ref{genfunbetti}) in the description of the cohomology of smooth models.

\section{Definition of smooth models}\label{smoothmodels}

\begin{definition}\label{defsm} Given a datum $(Q,d,\Theta)$ as before and another dimension vector $0\not=n\in{\bf N}I$, we associate to it a new datum $( \hat{Q}, \hat{d}, \hat{\Theta})$ as follows: 
\begin{itemize}
\item the vertices of $ \hat{Q}$ are those of $Q$, together with one additional vertex $\infty$,
\item the arrows of $ \hat{Q}$ are those of $Q$, together with $n_i$ arrows from $\infty$ to $i$, for every vertex $i$ of $Q$,
\item we define $ \hat{d}_i=d_i$ for all $i\in I$ and $ \hat{d}_\infty=1$,
\item we define $ \hat{\Theta}_i=\Theta_i$ for $i\in I$ and $ \hat{\Theta}_\infty=\mu(d)+\epsilon$ for some sufficiently small positive $\epsilon\in{\bf Q}$.
\end{itemize}
\end{definition}

The slope function associated to $ \hat{\Theta}$ is denoted by $ \hat{\mu}$, and the Euler form with respect to $ \hat{Q}$ is denoted by $$\langle\_,\_\rangle_{ \hat{Q}}:{\bf Z} \hat Q\times{\bf Z} \hat{Q}\rightarrow {\bf Z}.$$
Any dimension vector $e\in{\bf N}I$ can be viewed as a dimension vector for $ \hat{Q}$ via the natural embedding $I\subset \hat{I}$. Furthermore, for any dimension vector $e\in{\bf N}I$, we define a dimension vector $ \hat{e}$ for $ \hat{Q}$ as in the definition. The product $n\cdot e$ is defined by $\sum_{i\in I}n_ie_i$.

\begin{lemma}\label{bphs} The following properties hold for the new datum $( \hat{Q}, \hat{d}, \hat{\Theta})$:
\begin{enumerate}

\item We have $\langle  \hat{e},f\rangle_{ \hat{Q}}=\langle e,f\rangle-n\cdot f$ and $\langle e, \hat{f}\rangle_{ \hat{Q}}=\langle e,f\rangle$ for all $e,f\in{\bf N}I$.
\item For all $0\not=e\leq d$, we have $ \hat{\mu}(e)< \hat{\mu}( \hat{d})$ if and only if $ \hat{\mu}(e)\leq \hat{\mu}( \hat{d})$ if and only if $\mu(e)\leq\mu(d)$.
\item For all $e<d$, we have $ \hat{\mu}( \hat{e})< \hat{\mu}( \hat{d})$ if and only if $ \hat{\mu}( \hat{e})\leq \hat{\mu}( \hat{d})$ if and only if $\mu(e)<\mu(d)$.
\item The dimension vector $ \hat{d}$ is coprime for $ \hat{\Theta}$.
\end{enumerate}
\end{lemma}

\proof The first part follows from a direct computation using the definition of $ \hat{Q}$. By the operations on stabilities mentioned above, we can assume without loss of generality that $\Theta(d)=0$ and $\Theta\in{\bf Z}I^*$. We assume $\varepsilon$ to be sufficiently small in the above definition, thus we can assume $\varepsilon\leq 1$. For the second part, assume $0\not=e\leq d$. Then
$$ \hat{\mu}(e)< \hat{\mu}( \hat{d}) \iff\frac{\Theta(e)}{\dim e}<\frac{\varepsilon}{\dim d+1} \iff \Theta(e)<\varepsilon\cdot\frac{\dim e}{\dim d+1}\in]0,1[$$
(and similarly for $\leq$ instead of $<$), thus the statement follows since $\Theta(e)\leq 0$ if and only if $\mu(e)\leq\mu(d)$. The third part follows analogously: assume $e<d$. Then
$$ \hat{\mu}( \hat{e})< \hat{\mu}( \hat{d})\iff \frac{\Theta(e)+\varepsilon}{\dim e+1}<\frac{\varepsilon}{\dim d+1}\iff \Theta(e)<\epsilon\cdot\frac{\dim e-\dim d}{\dim d+1}\in]-1,0[$$
(and similarly for $\leq$ instead of $<$), thus the statement follows as before. Now the final statement follows.\hb

As in \cite{RF}, we can view representations of $ \hat{Q}$ of dimension vector $ \hat{d}$ as representations of $Q$, together with a framing datum, as follows:\\[1ex]
We fix vector spaces $V_i$ of dimension $n_i$ for all $i\in I$, and consider $V=\oplus_{i\in I}V_i$. It is easy to see that the representations of $ \hat{Q}$ of dimension vector $ \hat{d}$ can be identified with pairs $(M,f)$ consisting of a representation $M$ of $Q$ of dimension vector $d$ and a tuple $f=(f_i:V_i\rightarrow M_i)_{i\in I}$ of linear maps.

\begin{proposition}\label{cos} For a representation $(M,f)$ of $ \hat{Q}$ of dimension vector $ \hat{d}$, the following are equivalent:
\begin{description}
\item[a)] $(M,f)$ is $ \hat{\mu}$-semistable
\item[b)] $(M,f)$ is $ \hat{\mu}$-stable
\item[c)] $M$ is a $\mu$-semistable representation of $Q$, and $\mu(U)<\mu(M)$ for all proper subrepresentations $U$ of $M$ containing the image of $f$.
\end{description}
\end{proposition}

\proof  We consider the dimension vectors of non-trivial subrepresentations of the representation $(M,f)$.\\[1ex]
The subrepresentations of $(M,f)$ of dimension vector $e$ clearly correspond to subrepresentations of $M$ of dimension vector $e$, whereas the subrepresentations of $(M,f)$ of dimension vector $ \hat{e}$ correspond to subrepresentations $U$ of $M$ such that ${\rm Im}(f)\subset U$; namely, they are of the form $(U,f)$.\\[1ex]
Now assume that $(M,f)$ is $ \hat{\mu}$-semistable (resp. $ \hat{\mu}$-stable), and let $U$ be a non-trivial subrepresentation of $M$. 
Viewing $U$ as a subrepresentation of $(M,f)$, Lemma \ref{bphs} yields $\mu(U)\leq\mu(M)$. Thus, $M$ is $\mu$-semistable. Now assume that ${\rm Im}(f)\subset U$. Then Lemma \ref{bphs} yields $\mu(U)<\mu(M)$, as claimed.\\[1ex]
Conversely, assume that $M$ fulfills the claimed conditions. The above analysis of the possible subrepresentations of $(M,f)$, together with Lemma \ref{bphs}, immediately yields $ \hat{\mu}$-(semi-)stability of $(M,f)$.\hb

By Lemma \ref{bphs}, the dimension vector $ \hat{d}$ is coprime for $ \hat{\Theta}$. We thus have an equality of moduli $M_{ \hat{d}}^{\rm st}( \hat{Q})=M_{ \hat{d}}^{\rm sst}( \hat{Q})$. We also denote $$R_{d,n}^\Theta(Q)=R_{ \hat{d}}^{\rm  \hat{\Theta}-st}( \hat{Q}).$$ 

\begin{definition} We denote $M_{ \hat{d}}^{\rm st}( \hat{Q})$ by $M_{d,n}^{\Theta}(Q)$ and call this variety a smooth model for $M_d^{\rm sst}(Q)$.
\end{definition}

From the definition of $ \hat{Q}$ and the general properties of quiver moduli, we now immediately get:

\begin{lemma}\label{cosplus} The smooth model $M_{d,n}^\Theta(Q)$ parametrizes equivalence classes of pairs $(M,f)$ as in Proposition \ref{cos}, under the equivalence relation identifying $(M,f)$ and $(M',f')$ if and only if there exists an isomorphism $\varphi:M\rightarrow M'$ such that $f'=\varphi f$.
\end{lemma}

\remark For each vertex $i\in I$, we consider the projective representation $P_i$ of $Q$, which has the paths from $i$ to $j$ as a basis of the space $(P_i)_j$. We have $${\rm Hom}_Q(P_i,M)\simeq M_i$$ for all representations $M$. Therefore, defining $$P^{(n)}=\bigoplus_{i\in I}P_i^{n_i},$$
maps of $Q$-representations from $P^{(n)}$ to $M$ can be naturally identified with maps $f:V\rightarrow M$ as before. Thus, we can also interpret the points of $M_{d,n}^\Theta(Q)$ as equivalence classes of morphisms $f:P^{(n)}\rightarrow M$ such that $M$ is $\mu$-semistable and $\mu(U)<\mu(M)$ for any proper subrepresentations $U$ of $M$ containing the image of $f$, under the equivalence relation identifying $(f:P^{(n)}\rightarrow M)$ and $(f':P^{(n)}\rightarrow M')$ if and only if there exists an isomophism $\varphi:M\rightarrow M'$ such that $f'=\varphi f$.

\begin{proposition}\label{basicgeomprop} If $M_{d,n}^\Theta(Q)$ is non-empty, it is a smooth and pure variety of dimension $n\cdot d-\langle d,d\rangle$, admitting a projective morphism (the Hilbert-Chow morphism) $\pi:M_{d,n}^\Theta(Q)\rightarrow M_d^{\rm sst}(Q)$.
\end{proposition}

\proof Smoothness and purity follow from Proposition \ref{purity}, using the fact that $ \hat{d}$ is coprime for $ \hat{\Theta}$ by Lemma \ref{bphs}. The map $$R_{d,n}^\Theta(Q)\rightarrow R_d(Q),\;\;\;(M,f)\mapsto M,$$
thus forgetting the framing datum $f:V\rightarrow M$, has image contained in $R_d^{\rm sst}(Q)$ by Proposition \ref{cos}. Thus, it descends to a morphism $$\pi:M_{d,n}^\Theta(Q)\rightarrow M_d^{\rm sst}(Q).$$
On the other hand, we have projective morphisms $$p:M_d^{\rm sst}(Q)\rightarrow M_d^{\rm ssimp}(Q)\mbox{ and } \hat{p}:M_{d,n}^\Theta(Q)\rightarrow M_{ \hat{d}}^{\rm ssimp}( \hat{Q}).$$
The moduli spaces of semisimple representations are spectra of invariant rings, which are generated by traces along oriented cycles in the quivers. But the oriented cycles of $Q$ and $ \hat{Q}$ coincide by definition of $ \hat{Q}$, and thus $$M_{ \hat{d}}^{\rm ssimp}( \hat{Q})\simeq M_d^{\rm ssimp}(Q).$$
Consequently, the composite $ \hat{p}$ of the morphisms
$$M_{d,n}^\Theta(Q)\stackrel{\pi}{\rightarrow} M_d^{\rm sst}(Q)\stackrel{p}{\rightarrow} M_d^{\rm ssimp}(Q)\simeq M_{ \hat{d}}^{\rm ssimp}( \hat{Q})$$
being projective, the morphism $\pi:M_{d,n}^\Theta(Q)\rightarrow M_d^{\rm sst}(Q)$ is projective, too. This proof is summarized by the diagram
\begin{center}
\includegraphics{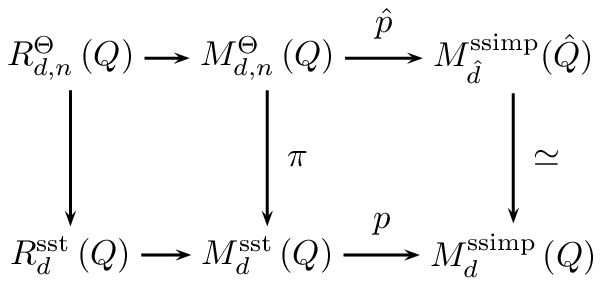}
\end{center}
%$$\begin{array}{cccccc}
%R_{d,n}^\Theta(Q)&\rightarrow&M_{d,n}^\Theta(Q)&\stackrel{ \hat{p}}{\rightarrow}&M_{ \hat{d}}^{\rm ssimp}( \hat{Q})\\
%\downarrow&&\downarrow\pi&&\downarrow\simeq\\
%R_d^{\rm sst}(Q)&\rightarrow&M_d^{\rm sst}(Q)&\stackrel{p}{\rightarrow}&M_d^{\rm ssimp}(Q)
%\end{array}$$\hb

\remark No general effective exact criterion for non-emptyness of $M_{d,n}(Q)$ is available at the moment, except for the general recursive criterion Theorem \ref{criterion}. We will derive such a criterion later in case $\Theta=0$ (see Theorem \ref{criterionhilb}).

\begin{lemma}\label{nonempty} For large enough $n$ (in fact, for $n\geq d$), the smooth model $M_{d,n}^\Theta(Q)$ is non-empty if and only if $M_d^{\rm sst}(Q)$ is.
\end{lemma}

\proof By Proposition \ref{cos}, non-emptyness of $M_d^{\rm sst}(Q)$ is clearly necessary for non-emptyness of $M_{d,n}^\Theta(Q)$. Sufficiency follows again from Proposition \ref{cos}, since in case $n_i\geq d_i$, the maps $f_i:V_i\rightarrow M_i$ can be chosen to be surjective.\hb

In case $d$ is coprime for $\Theta$, the moduli space $M_d^{\rm st}(Q)$ carries tautological bundles $\mathcal{M}_i$ for $i\in I$, providing a universal representation $\mathcal{M}$ of $Q$ in the category of vector bundles on $M_d^{\rm st}(Q)$, in the sense that the representation $${\cal M}_M=((\mathcal{M}_i)_M)_{i\in I}$$ induced on the fibres of the various bundles ${\cal M}_i$ over the point $M$ is isomorphic to $M$, for all stable representations $M\in M_d^{\rm st}(Q)$.

\begin{proposition} If $d$ is coprime for $\Theta$, the smooth model $M_{d,n}^\Theta(Q)$ is isomorphic to the projective bundle ${\bf P}(\oplus_{i\in I}\mathcal{M}_i^{n_i})$ over $M_d^{\rm st}(Q)$.
\end{proposition}

\proof We recall the construction of the universal bundles $\mathcal{M}_i$ from \cite{King}. Since $d$ is coprime for $\Theta$, it is in particular indivisible, so we can choose integers $a_i$ for $i\in I$ such that $$\sum_{i\in I}a_id_i=1.$$
Consider the trivial vector bundle $$q:R_d^{\rm st}(Q)\times M_i\rightarrow R_d^{\rm st}(Q),$$
with action of $G_d$ given by
$$(g_j)_j\cdot((M_\alpha)_\alpha,m_i):=( (g_kM_\alpha g_j^{-1})_{\alpha:j\rightarrow k},\prod_{j\in I}\det(g_j)^{a_j}\cdot g_im_i).$$
The stabilizer of any point in $R_d^{\rm st}(Q)$ reduces to the scalars, which act trivially on the fibres of $q$ by definition of the action. Therefore, this trivial vector bundle descends to a bundle $\mathcal{M}_i$ on the geometric quotient $M_d^{\rm st}(Q)$. Consequently, the projective bundle ${\bf P}(\mathcal{M}_i)$ can be realized as the quotient of $R_d\times (M_i\setminus 0)$ by the action of $G_d$ via $$(g_i)_i\cdot((M_\alpha)_\alpha,m_i):=((g_jM_\alpha g_i^{-1})_{\alpha:i\rightarrow j},g_im_i).$$
More generally, the projective bundle ${\bf P}(\bigoplus_{i\in I}\mathcal{M}_i^{a_i})$ can be realized as the quotient
$${\bf P}(\bigoplus_{i\in I}\mathcal{M}_i^{a_i})=(R_d\times (\bigoplus_{i\in I}M_i^{a_i}\setminus 0))/G_d$$ by the induced action of $G_d$.\\[1ex]
On the other hand, we consider the variety $R_{d,n}^\Theta(Q)$. By the characterization of Proposition \ref{cos}, it consists of pairs $(M,f)$ of a stable representation $M$ and a non-zero map $f:V\rightarrow M$ as above, and $M_{d,n}^\Theta(Q)$ is the quotient of $G_d$ by this action. Consequently, we see that the two varieties in question are isomorphic.\hb

\remark It is likely that no universal bundle on $M_d^{\rm st}(Q)$ (or $M_d^{\rm sst}(Q)$) exists in case $d$ is not coprime. By the above proposition, we can therefore view the smooth models $M_{d,i}^\Theta(Q)$ as optimal approximations to the missing universal bundles.\\[2ex]
\examples We work out some examples of smooth models:

\begin{description}
\item[Example A] Let $Q$ be the quiver consisting of a single vertex and $m$ loops, so that the stability $\Theta$ is trivial. The smooth model $M_{d,1}^0(Q)$ coincides with the non-commutative Hilbert scheme $H_{d}^{(m)}$ considered in \cite{RNCH}, parametrizing finite codimensional left ideals in free algebras.

\item[Example B] Let $Q$ be the $r$-subspace quiver, given by vertices $I=\{0,1,\ldots,r\}$, and one arrow from $i$ to $0$, for $i=1,\ldots,r$. Consider dimension vectors $d$ such that $d_i\leq d_0$ for all $i=1,\ldots,r$. Define the stability $\Theta$ by $$\Theta_i=0\mbox{ for }i=1,\ldots,r\mbox{, and }\Theta_0=-1.$$
Then the moduli space $M_d^{\rm st}(Q)$ parametrizes stable ordered tuples of $d_i$-dimensional subspaces $U_i\subset V$ of a $d$-dimensional vector space, up to the action of ${\rm GL}(V)$. A tuple is
semistable if for all tuples of subspaces $U_i'\subset U_i$, we have
$$\frac{\dim\sum_iU_i'}{\dim V}\geq\frac{\sum_i\dim U_i'}{\sum_i\dim U_i},$$
and it is stable if the above inequality is strict.
This is one of the classical examples of Geometric Invariant Theory \cite{Mum}. As soon as $\sum_i \dim U_i$ and $\dim V$ are not coprime, there might exist properly semistable points.
Define the dimension vector $n$ by $$n_0=1\mbox{ and }n_i=0\mbox{ for }i=1,\ldots,r.$$
The smooth model $M_{d,n}^\Theta(Q)$ then parametrizes
ordered tuples $(U_i\subset V)_i$, together with a vector $v\in V$, such that the following condition holds: 
for all tuples of subspaces $U_i'\subset U_i$, the previous inequality holds, and it is strict if $v\in\sum_iU_i'$.
Such data are considered up to the natural action of ${\rm GL}(V)$.

\item[Example C] With obvious modifications, the previous example can be extended to moduli spaces parametrizing tuples of flags in vectorspaces, as in \cite[Example B]{RHNS}.

\item[Example D] Let $Q$ be the $m$-Kronecker quiver, with set of vertices $I=\{i,j\}$, and $m$ arrows from $i$ to $j$. Define the stability $\Theta$ by $$\Theta_i=1\mbox{ and }\Theta_j=0,$$
and let $d$ and $n$ be arbitrary dimension vectors. A representation of $Q$ of dimension vector $d$ corresponds to an $m$-tuple of linear maps $$(f_k:V\rightarrow W)_{k=1,\ldots,m},$$
considered up to simultaneous base change in $V$ and $W$. A tuple is semistable if for all proper subspaces $U\subset V$, we have
$$\sum_kf_k(U)\geq\frac{\dim W}{\dim V}\cdot\dim U,$$
and it is stable if the inequality is strict.
Fix vector spaces $V'$ and $W'$ of dimension $n_i$ and $n_j$, respectively. The smooth model $M_{d,n}^\Theta(Q)$ parametrizes stable tuples $(f_k)$, together with linear maps $$\varphi:V'\rightarrow V\mbox{ and }\psi:W'\rightarrow W,$$
up to base change in $V$ and $W$. Stability is defined by the above inequality, which has to be strict whenever $${\rm Im}(\varphi)\subset U\mbox{ and }{\rm Im }(\varphi)\subset\sum_kf_k(U).$$ 

\item[Example E] As a particular case of Example B, we consider the case $d_i=1$ for all $i=1,\ldots,r$. In this case, the moduli space $M_d^{\rm st}(Q)$ parametrizes stable ordered $r$-tuples of points $$(v_1,\ldots,v_r)\in({\bf P}^{d-1})^r$$ in projective space of dimension $d-1$ up to the action of ${\rm PGL}_r$, and the semistability condition reads
$$\frac{\dim\langle v_j\,:\, j\in J\rangle}{d}\geq\frac{|J|}{r}\mbox{ for all subsets }J\subset\{1,\ldots,r\}.$$
For the choice of $n$ as above, the smooth model $M_{d,n}^\Theta(Q)$ parametrizes stable $r+1$-tuples $(v_0,v_1,\ldots.v_r)$ in projective space of dimension $r-1$ up to the action of ${\rm PGL}_r$, where stability is defined by the above inequality, together
with the condition that $v_0\not\in\langle v_j\, :\, j\in J\rangle $ if equality holds.

\item[Example F] Even more particularly, consider the special case $d=2$ in the previous example. Then a tuple of points $(v_1,\ldots,v_r)$ in the projective line is stable if no more than $r/2$ of the points coincide. If $r$ is odd, all semistable tuples are already stable. If $r$ is even, the $\frac{1}{2}{r\choose r/2}$ isomorphism classes of nonstable polystable representations correspond to tuples consisting of two different points, each occuring with multiplicity $r/2$. The smooth model parametrizes tuples $(v_0,\ldots,v_r)$, such that $v_0$ is different from a point occuring with multiplicity $r/2$.

\end{description}

\section{The fibres of the Hilbert-Chow morphism and stratifications of smooth models}\label{fibres}

First we recall the Luna type stratification (see \cite{ALB, LBP}) of $M_d^{\rm sst}(Q)$. As mentioned in section \ref{recoll}, this moduli space parametrizes isomorphism classes $[M]$ of $\mu$-polystable representations of dimension vector $d$. Therefore, we can decompose such a representation $M$ into a direct sum
$$M=U_1^{z_1}\oplus\ldots\oplus U_s^{z_s}$$
of pairwise non-isomorphic $\mu$-stable representations $U_k$ of dimension vectors $\dimv U_k=d^k$ such that $\mu(d^k)=\mu(d)$ for all $k=1\ldots s$. We call the pair of tuples
$$\xi=(d^*=(d^1,\ldots,d^s),\;\;\; z_*=(z_1,\ldots,z_s))$$
the polystable type of $M$. Conversely, given such a pair of tuples $\xi=(d^*,z_*)$, we associate to it the Luna stratum $S_\xi$, the subset consisting of representations $M$ admitting a decomposition as above, for pairwise non-isomorphic stable representations $U_k$ of dimension vector $d^k$. It is clear from the above that this gives a decomposition of $M_d^{\rm sst}(Q)$ into finitely many disjoint locally closed subsets. The generic Luna stratum $S_{((d),(1))}$ obviously coincides with $M_d^{\rm st}(Q)\subset M_d^{\rm sst}(Q)$.\\[1ex]
For a given polystable type $\xi$, we define a new quiver datum as follows:\\[1ex]
Denote by $Q_\xi$ the quiver with set of vertices $\{1,\ldots,s\}$ and $\delta_{k,l}-\langle d^k,d^l\rangle$ arrows from $k$ to $l$, for each pair $1\leq k,l\leq s$. We define a dimension vector $d_\xi$ by $(d_\xi)_k:=z_k$ for all $k=1\ldots s$, and we define $(n_\xi)_k=n\cdot d^k$ for all $k$. We denote by
$$\pi_\xi:M_{d,n}^0(Q_\xi)\rightarrow M_{d_\xi}^{\rm ssimp}(Q_\xi)$$
the natural morphism. The fibre $\pi_\xi^{-1}(0)$ parametrizes (equivalence classes of) pairs $(Z,h)$ consisting of a nilpotent representation $Z$ of $Q_\xi$ of dimension vector $d_\xi$, together with a map $h$ whose image generates the representation $Z$. We therefore call $$\pi_\xi^{-1}(0)=:M_{d_\xi,n_\xi}^{\rm 0,nilp}(Q_\xi)$$ the nilpotent part of the smooth model $M_{d_\xi,n_\xi}^0(Q_\xi)$.\\[1ex]
Next we recall a strong form of the Luna slice theorem (\cite{Luna}, see \cite{Dr} for an introduction), which will be used in the proof of the following theorem. Let a reductive algebraic group $G$ act on an affine variety $Y$ with quotient $\pi_Y:Y\rightarrow Y//G$, and let $y\in Y$ be a point whose orbit $Gy$ is closed in $Y$, and thus has reductive stabilizer $G_y$. Luna's slice theorem asserts the existence of a $G_y$-invariant locally closed affine subvariety $S$ of $Y$ containing $y$ (the \'etale slice) such that the canonical map $\psi:G\times^{G_y}S\rightarrow Y$ is strongly \'etale in the following sense: the map $\psi$ is \'etale, its image is an open $G$-stable subset of $Y$, the induced map on quotients $\psi_{/G}:(G\times^{G_y}S)//G\simeq S//G_y\rightarrow U//G$ is \'etale, and the map $\psi$, together with the quotient map $G\times^{G_y}S\rightarrow S//G_y$, induces an isomorphism $G\times^{G_y}S\simeq S//G_y\times_{U//G}U$. These three properties can be summarized into a Cartesian square with vertical quotient morphisms and \'etale horizontal morphisms
$$\begin{array}{ccc}G\times^{G_y}S&\stackrel{\psi}{\rightarrow}&U\\ \downarrow&&\downarrow\\ S//G_y&\stackrel{\psi_{/G}}{\rightarrow}&U//G.\end{array}$$
As a consequence, for a point $s\in S//G_y$ with image $u=\psi_{/G}(s)$, we have an isomorphism $G\times^{G_y}\pi_S^{-1}(s)\simeq\pi_Y^{-1}(u)$, where $\pi_S:S\rightarrow S//G_y$ denotes the quotient map. Furthermore, if $Y$ is smooth at $y$, the slice $S$ can be chosen to be smooth, such that the inclusion $S\subset Y$ induces an isomorphism $T_yS\simeq N_y$, where $N_y$ denotes the normal space to the orbit $Gy$ in $Y$. In this case, there exists a strongly \'etale $G_y$-invariant morphism $\phi:S\rightarrow N_y$, resulting in a second Cartesian square as above.

\begin{theorem}\label{theoremfibres} The fibre $\pi^{-1}(M)$ of the morphism $\pi:M_{d,n}^\Theta(Q)\rightarrow M_d^{\rm sst}(Q)$ over a polystable representation $M$ of type $\xi$ is isomorphic to the nilpotent part of a smooth model:
$$\pi^{-1}(M)\simeq M_{d_\xi,n_\xi}^{\rm 0,nilp}(Q_\xi).$$
Moreover, the restriction of $\pi$ to the inverse image of the Luna stratum $S_\xi$ is a fibration, that is, it is locally trivial in the \'etale topology.
\end{theorem}

\proof We consider the datum $( \hat{Q}, \hat{d})$ as in section \ref{smoothmodels}, but define a degenerate stability $ \hat{\Theta}_0$ by $$( \hat{\Theta}_0)_i=\Theta_i\mbox{ for }i\in I\mbox{ and }( \hat{\Theta}_0)_\infty=\mu(d).$$
This corresponds to the case $\varepsilon=0$ in Definition \ref{defsm}. We can easily analyze stability of representations $(M,f)$ of $ \hat{Q}$ of dimension vector $ \hat{d}$ with respect to the slope function $ \hat{\mu}_0$ defined via $ \hat{\Theta}_0$, using obvious variants of Lemma \ref{bphs} and Proposition \ref{cos}, resulting in the following:
\begin{itemize}
\item $ \hat{\mu}_0(e)\leq \hat{\mu}_0( \hat{d})$ if and only if $\mu(e)\leq\mu(d)$,
\item $ \hat{\mu}_0( \hat{e})\leq \hat{\mu}_0( \hat{d})$ if and only if $\mu(e)\leq\mu(d)$,
\item a representation $(M,f)$ is $ \hat{\mu}_0$-semistable if and only if $M$ is a semistable representation of $Q$,
\item $(M,f)$ is $ \hat{\mu}_0$-stable if and only if either $M$ is stable and $f=0$, or $M=0$,
\item $(M,f)$ is $ \hat{\mu}_0$-polystable if and only if $M$ is polystable and $f=0$.
\end{itemize}
Since the points of the moduli space $M_{ \hat{d}}^{\rm  \hat{\Theta}_0-sst}( \hat{Q})$ parametrize isomorphism classes of polystables, we conclude that the natural map $R_{ \hat{d}}( \hat{Q})\rightarrow R_d(Q)$ induces a map $$R_{ \hat{d}}^{\rm  \hat{\Theta}_0-sst}( \hat{Q})\rightarrow R_d^{\rm sst}(Q),$$
which descends to an isomorphism of moduli $$M_{ \hat{d}}^{\rm  \hat{\Theta}_0-sst}( \hat{Q})\stackrel{\sim}{\rightarrow}M_d^{\rm sst}(Q).$$
We want to apply the Luna slice theorem to the fibres of the resulting morphism $$ \hat{q}:R_{ \hat{d}}^{\rm  \hat{\Theta}_0-sst}( \hat{Q})\rightarrow M_d^{\rm sst}(Q),$$
using techniques of \cite{ALB}. Given a point $M\in M_d^{\rm sst}(Q)$ of polystable type $\xi$ as above, we construct a smooth affine open $G_{ \hat{d}}$-stable subvariety $X$ such that
$$ \hat{q}^{-1}(M)\subset X\subset R_{ \hat{d}}^{\rm  \hat{\Theta}_0-sst}( \hat{Q})$$ as in \cite[section 3]{ALB}. We associate extended local quiver data $ \hat{Q}_\xi$ and $ \hat{d}_\xi$ to $(Q_\xi, d_\xi, n_\xi)$ as in section \ref{smoothmodels} and consider the quotient morphism $$q_\xi:R_{ \hat{d}_\xi}( \hat{Q}_\xi)\rightarrow M_{d_\xi}^{\rm ssimp}(Q_\xi).$$
According to \cite{ALB}, the affine space $R_{ \hat{d}_\xi}( \hat{Q}_\xi)$ can be identified with the normal space at the point $$(M,0)\in R_{ \hat{d}}^{\rm  \hat{\Theta}_0-sst}( \hat{Q})$$ to the $G_{ \hat{d}}$-orbit of $(M,0)$. As in \cite[section 4]{ALB}, the Luna slice theorem in the form recalled above can be applied to the restriction of $ \hat{q}$ to $X$. It yields a locally closed affine subvariety $S$ of $X$ containing the point $(M,0)$, which is stable under the stabilizer $$(G_{ \hat{d}})_{(M,0)}\simeq G_{d_\xi}.$$
Moreover, there exists a diagram of Cartesian squares with \'etale horizontal maps
\begin{center}
\includegraphics{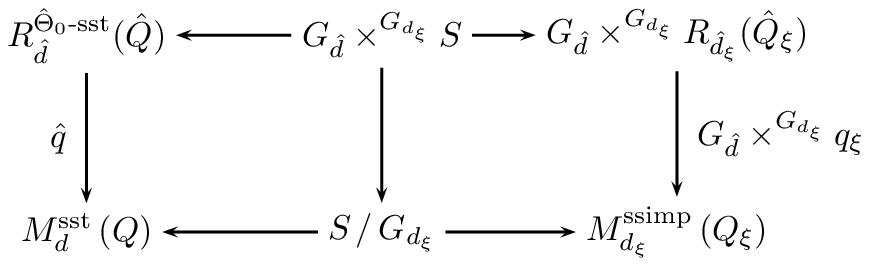}
\end{center}
%$$\begin{array}{rcccl}
%R_{ \hat{d}}^{\rm  \hat{\Theta}_0-sst}( \hat{Q})&\leftarrow&G_{ \hat{d}}\times^{G_{d_\xi}}S&\rightarrow&G_{ \hat{d}}\times^{G_{d_\xi}}R_{ \hat{d}_\xi}( \hat{Q}_\xi)\\  \hat{q}\downarrow&&\downarrow&&\downarrow G_{ \hat{d}}\times^{G_{d_\xi}}q_\xi\\
%M_d^{\rm sst}(Q)&\leftarrow&S/G_{d_\xi}&\rightarrow&M_{d_\xi}^{\rm ssimp}(Q_\xi)\end{array}$$
%with all horizontal maps being \'etale, and all vertical maps being $G_{ \hat{d}}$-quotients. Additionaly, the restriction of the morphism $ \hat{q}$ to the inverse image of the stratum $S_\xi$ is a $G_d$-fibration, thus $G_d$-equivariantly locally trivial in the \'etale topology by \cite[Corollaire 5]{Luna}.  Thus, we have an isomorphism
%$$ \hat{q}^{-1}(M)\stackrel{\sim}{\leftarrow}G_{ \hat{d}}\times^{G_{d_\xi}}q_\xi^{-1}(0).$$
We now want to compare the intersections
$$ \hat{q}^{-1}(M)\cap R_{d,n}^\Theta(Q)\mbox{ and }q_\xi^{-1}(0)\cap R_{d_\xi,n_\xi}^0(Q_\xi).$$
Suppose we are given a point $(Z,h)$ in the fibre $q_\xi^{-1}(0)$, which corresponds to a point $(N,f)$ in the fibre $ \hat{q}^{-1}(M)$ under the above isomorphism. The representation $N$ thus admits a filtration $F^*$ with subquotients isomorphic to the stable direct summands $U_k$ of $M$. Assume we are given a subrepresentation $V\subset N$ containing the image of $f$ such that $\mu(V)=\mu(N)$. Since semistable representations of a fixed slope form an abelian subcategory, the intersection of $V$ with the filtration $F^*$ induces a filtration of $V$ with subquotients isomorphic to some of the $U_k$. Thus, there exists a subrepresentation $W\subset Z$ containing the image of $h$. Conversely, a subrepresentation $W$ of $Z$ containing the image of $h$ induces a subrepresentation $V$ of $N$ containing the image of $f$, such that $\mu(V)=\mu(N)$.\\[1ex]
Since the open subsets $R_{d,n}^\Theta(Q)$ and $R_{d_\xi,n_\xi}^0(Q_\xi)$, respectively, are defined by the triviality of such subrepresentations, we conclude that the above isomorphism of fibres restricts to an isomorphism
$$ \hat{q}^{-1}(M)\cap R_{d,n}^\Theta(Q)\stackrel{\sim}{\leftarrow}G_{ \hat{d}}\times^{G_{d_\xi}}(q_\xi^{-1}(0)\cap R_{d_\xi,n_\xi}^0(Q_\xi)).$$
Passage to $G_{ \hat{d}}$-quotients yields the desired isomorphism and the \'etale local triviality.\hb

We define a stratification of $M_{d,n}^\Theta(Q)$ as follows: for any polystable type $\xi$, define
$$M_{d,n}^\Theta(Q)_\xi=\pi^{-1}(S_\xi).$$
The methods developed above now immediately give the following.

\begin{corollary}\label{strat} The smooth model $M_{d,n}^\Theta(Q)$ is the disjoint union of the finitely many locally closed strata $M_{d,n}^\Theta(Q)$, for various polystable types $\xi$. Each stratum $M_{d,n}^\Theta(Q)$ is an \'etale locally trivial fibration over the corresponding Luna stratum $S_\xi$, with fibre isomorphic to $M_{d_\xi,n_\xi}^{\rm 0,nilp}(Q_\xi)$. In particular, the generic stratum $M_{d,n}^\Theta(Q)_{((d),(1))}$ is a ${\bf P}^{n\cdot d-1}$-fibration over $M_d^{\rm st}(Q)$.
\end{corollary}

\remark The smooth models $M_{d,n}^\Theta(Q)$ can thus be viewed as a "compactification" (being projective over the affine moduli $M_d^{\rm ssimp}(Q)$) of a ${\bf P}^{n\cdot d-1}$-fibration over $M_d^{\rm st}(Q)$.\\[2ex]
\examples
\begin{enumerate}
\item The above result was obtained in the case of Example A in \cite{LBS}.
\item Continuing Example F above, we describe the fibre of $\pi$ over the polystable nonstable points in the case $r$ is even: the quiver $Q_\xi$ has two vertices $1,2$, and the number of arrows from $i$ to $j$ equals $r/2-1+\delta_{i,j}$. The dimension vector is given by $$(d_\xi)_1=1=(d_\xi)_2.$$
Thus, the loops at the vertices can be ignored, since the nilpoteny condition forces them to be represented by zero. The dimension vector $n_\xi$ is given by $$(n_\xi)_1=1=(n_\xi)_2.$$
The resulting nilpotent part of the smooth models is isomorphic to two copies of a line bundle over projective space of dimension $r-2$, glued at a single point.
\end{enumerate}

\section{Cohomology}\label{cohomology}

As in \cite{CR,RC}, we consider a completed version of the Hall algebra of the quiver $Q$ and perform computations with certain generating functions in it. By applying an evaluation map, this yields an identity involving the number of rational points of $M_{d,n}^\Theta(Q)$ over a finite field ${\bf F}_q$.\\[2ex] 
Define $$H(({\bf F}_qQ))=\prod_{[M]}{\bf Q}\cdot[M]$$ as the direct product of one-dimensional ${\bf Q}$-vector spaces with basis elements $[M]$, indexed by the isomorphism classes of finite dimensional representations of ${\bf F}_qQ$, and with the following multiplication: $$[M]\cdot[N]=\sum_{[X]}F_{M,N}^X\cdot[X],$$
where $F_{M,N}^X$ denotes the number of ${\bf F}_qQ$-subrepresentations $U$ of $X$ such that $U\simeq N$ and $X/U\simeq M$.\\[1ex]
This defines an associative unital ${\bf N}I$-graded ${\bf Q}$-algebra by \cite{RC}.
Note furthermore that the direct product $$\prod_{[M]\in{\rm Rep}^s_{{\bf F}_q}(Q)}{\bf Q}\cdot[M]$$ over all isomorphism classes in ${\rm Rep}^s_{{\bf F}_q}(Q)$ defines a subalgebra of $H(({\bf F}_qQ))$.\\[1ex]
Define ${\bf Q}_q[[I]]$ as the direct product of the one-dimensional ${\bf Q}$-vector spaces with basis elements $t^d$ indexed by $d\in{\bf N}I$, and with multiplication $$t^d\cdot t^e=q^{-\langle d,e\rangle}t^{d+e}.$$
The map sending a basis element $[M]$ to the element $$\int[M]:=\frac{1}{|{\rm Aut}(M)|}t^{\dimv M}$$ induces a ${\bf Q}$-algebra homomorphism $$\int:H(({\bf F}_qQ))\rightarrow {\bf Q}_q[[I]]$$ by \cite{RHNS}.\\[3ex]
Let $\mathcal{C}$ be a full abelian subcategory of ${\rm Rep}_{{\bf F}_q}(Q)$, that is, $\mathcal{C}$ is closed under kernels, cokernels and extensions. If $U\subset M$ is a subrepresentation of a representation $M\in\mathcal{C}$, we therefore have a well-defined minimal subrepresentation $$\langle U\rangle_\mathcal{C}=\bigcap_{{U\subset V\subset M}\atop{V\in\mathcal{C}}}\in\mathcal{C}$$ of $M$ containing $U$, defined as the intersection of all subrepresentations $V\in\mathcal{C}$ of $M$ containing $U$.\\[1ex]
Fix an arbitrary representation $Z\in{\rm Rep}_{{\bf F}_q}(Q)$. We denote by ${\rm Hom}^0_{{{\bf F}_q}Q}(Z,M)$ the set of all ${{\bf F}_q}Q$-morphisms $f$ from $Z$ to $M$ with the following property:
\begin{center}if ${\rm Im}(f)\subset U\subset M$ for $U\in\mathcal{C}$, then $U=M$.\end{center}
We consider the following elements of $H((Q))$:
$$h_Z:=\sum_{[M]\in\mathcal{C}}|{\rm Hom}_{{{\bf F}_q}Q}^0(Z,M)|\cdot[M],$$
$$e_Z:=\sum_{[M]\in\mathcal{C}}|{\rm Hom}_{{{\bf F}_q}Q}(Z,M)|\cdot[M],$$
$$e_0:=\sum_{[M]\in\mathcal{C}}[M].$$

\begin{lemma}\label{keyidentity} We have $e_0\cdot h_Z=e_Z$ in $H((Q))$.
\end{lemma}

\proof We have
\begin{eqnarray*}e_0\cdot h_Z&=&\sum_{[M],[N]\in\mathcal{C}}|{\rm Hom}^0_{{{\bf F}_q}Q}(Z,N)|\cdot[M]\cdot[N]\\
&=&\sum_{[M],[N]\in\mathcal{C}}|{\rm Hom}^0_{{{\bf F}_q}Q}(Z,N)|\cdot(\sum_{[X]}F_{M,N}^X\cdot[X])\\
&=&\sum_{[X]\in\mathcal{C}}(\sum_{[M],[N]\in\mathcal{C}}F_{M,N}^X\cdot|{\rm Hom}^0_{{{\bf F}_q}Q}(Z,N)|)\cdot[X]\\
&=&\sum_{[X]\in\mathcal{C}}(\sum_{U\subset X, U\in\mathcal{C}}|{\rm Hom}_{{{\bf F}_q}Q}^0(Z,U)|)[X],
\end{eqnarray*}
where the reparametrization in the last equality uses the definition of $F_{M,N}^X$.
To prove that this equals $$e_Z=\sum_{[X]\in\mathcal{C}}|{\rm Hom}_{{{\bf F}_q}Q}(Z,X)|\cdot[X],$$
it thus suffices to exhibit a bijection between ${\rm Hom}_{{{\bf F}_q}Q}(Z,X)$ and the set of pairs $(U,f)$, where $U\subset X$, $U\in\mathcal{C}$ and $f\in{\rm Hom}_{{{\bf F}_q}Q}^0(Z,U)$. This is given by assigning to $f:Z\rightarrow X$ the pair $$(\langle{\rm Im}(f)\rangle_\mathcal{C},f:Z\rightarrow\langle{\rm Im}(f)\rangle_\mathcal{C}),$$
with converse map assigning to a pair $(U,f)$ the composite $$f:Z\rightarrow U\rightarrow X.$$\hb

\begin{theorem}\label{genfunbetti} For all $(Q,\Theta,n)$ as before, we have the following identity of generating functions in ${\bf Q}_q[[I]]$:
$$\sum_{d\in({\bf N}I)_s}P_{M_{d,n}^\Theta(Q)}(q)t^d=(\sum_{d\in({\bf N}I)_s}P_d(q)t^d)^{-1}\cdot\sum_{d\in({\bf N}I)_s}q^{n\cdot d}P_d(q)t^d.$$
\end{theorem}

\proof We apply the identity of Lemma \ref{keyidentity} to the abelian subcategory $$\mathcal{C}={\rm Rep}^s_{{\bf F}_q}(Q)$$ and the representation $$Z=P^n=\bigoplus_{i\in I}P_i^{n_i}$$ for $n\in{\bf N}I$, denoting the elements $h_Z$, $e_0$, $e_Z$ by $h_n$, $e_0$, $e_n$, respectively. We apply integration to these elements, noting that
$$\frac{1}{|{\rm Aut}_{{{\bf F}_q}Q}(M)({{\bf F}_q})|}=\frac{|\mathcal{O}_M({{\bf F}_q})|}{|G_d({{\bf F}_q})|}.$$
The integral of $e_n$ thus can be computed as
$$\sum_{[M]\in{\rm Rep}_{{\bf F}_q}^s (Q)}q^{n\cdot\dimv M}\frac{1}{|{\rm Aut}_{{{\bf F}_q}Q}(M)({{\bf F}_q})|}t^{\dimv M}=\sum_{d\in{\bf N}I}q^{n\cdot d}\frac{|R_d^{\rm sst}(Q)({{\bf F}_q})|}{|G_d({{\bf F}_q})|}t^d.$$
For $h_n$, we first note that, by the definitions and Lemma \ref{cosplus}, a point $(M,f)$ in $M_{d,n}^\Theta(Q)$ is given by a $\mu$-semistable representation $M$ of $Q$, together with a morphism $$f\in{\rm Hom}^0_{{{\bf F}_q}Q}(P^n,M),$$
considered up to the natural action of $G_d$ on such pairs. From this, we get:
$$\int h_n=\sum_{[M]\in{\rm Rep}_{{\bf F}_q}^s(Q)}|{\rm Hom}_{{{\bf F}_q}Q}^0(P^n,M)({{\bf F}_q})|\frac{1}{|{\rm Aut}_{{{\bf F}_q}Q}(M)({{\bf F}_q})|}t^{\dimv M}=$$
$$=\sum_{d\in{\bf N}I}(\sum_{M\in R_d^{\rm sst}(Q)}|{\rm Hom}_{{{\bf F}_q}Q}^0(P^n,M)({{\bf F}_q})|\frac{|\mathcal{O}_M({\bf F}_q)|}{|{\rm Aut}_{{{\bf F}_q}Q}(M)({{\bf F}_q})|})t^d=$$
$$=\sum_{d\in{\bf N}I}\frac{|R_{d,n}^{\Theta}(Q)({{\bf F}_q})|}{|G_d({{\bf F}_q})|}t^d.$$
By \cite[Proposition 6.6]{RHNS}, the fraction in this last sum equals $|M_{d,n}^\Theta(Q)({{\bf F}_q})|$. The identity
$$\int h_n=(\int e_0)^{-1}\cdot\int e_n$$
now gives the statement.\hb

The theorem immediately yields a recursive formula for the Poincar\'e polynomial of $M_{d,n}^\Theta(Q)$:

\begin{corollary}\label{recursion} For all $(Q,d,\Theta,n)$ as before, we have
$$P_{M_{d,n}^\Theta(Q)}(q)=(q^{n\cdot d}-1)P_d(q)-\sum_{\substack{0<e<d\\ e\in({\bf N}I)_s}}q^{-\langle d-e,e\rangle}P_{d-e}(q)P_{M_{e,n}^\Theta(Q)}(q).$$
\end{corollary}

To derive a summation formula for the cohomology of the smooth models, we analyze the formula in Definition \ref{admissible}, Theorem \ref{Betti} for the setup $( \hat{Q}, \hat{d}, \hat{\Theta})$. We work out the admissible decompositions of $ \hat{d}$ as in the theorem. These are obviously of the form 
$$(d^1,\ldots, \hat{d^r},\ldots,d^s)$$
for some decomposition $d=d^1+\ldots+d^s$, with the exception that $d^r$ is not neccessarily non-zero. Using Lemma \ref{bphs}, the stability condition in the theorem translates into
$$\mu(d^1+\ldots+d^k)\left\{\begin{array}{ccc}>\mu(d)&\mbox{if}&k<r,\\ \geq \mu(d)&\mbox{if}&k\geq r.\end{array}\right. $$
This allows us to classify such decompositions of $ \hat{d}$ as follows: we call a decomposition $d=d^1+\ldots+d^s$ into non-zero dimension vectors semi-admissible if $$\mu(d^1+\ldots+d^k)\geq\mu(d)\mbox{ for all }k\leq s.$$
For such a decomposition, we define $k_0$ as the minimal index $k$ such that $$\mu(d^1+\ldots+d^k)=\mu(d)$$ (this being well-defined since equality of slopes holds for $k=s$), and associate to it the admissible decompositions
$$(d^1,\ldots, \hat{d^r},\ldots,d^s) \mbox{ and } (d^1,\ldots,d^{r-1}, \hat{0},d^r,\ldots,d^s)$$
for all $r\leq k_0$, respectively. This gives a parametrization of all admissible decompositions of $ \hat{d}$ by (semi-)admissible decompositions of $d$, together with an index $r\leq k_0$.\\[1ex]
For the two types of sequences above, their contribution to the sum in Definition \ref{admissible} is easily worked out as
$$(-1)^{s-1}q^{-\sum_{k<l}\langle d^l,d^k\rangle +\sum_{k<r}n\cdot d^k}\prod_k\frac{|R_{d^k}({{\bf F}_q})|}{|G_{d^k}({{\bf F}_q})|}\cdot\frac{q^{n\cdot d_r}}{q-1}\mbox{ and}$$
$$(-1)^{s}q^{-\sum_{k<l}\langle d^l,d^k\rangle}\frac{|R_{d^k}({{\bf F}_q})|}{|G_{d^k}({{\bf F}_q})|}\cdot\frac{1}{q-1},$$
respectively. This results in the following sum over all semi-admissible decompositions $d^*$ of $d$:
$$\frac{1}{q-1}\cdot\sum_{d^*}\sum_{r\leq k_0}(-1)^{s-1}q^{-\sum_{k<l}\langle d^l,d^k\rangle}\prod_k\frac{|R_{d^k}({{\bf F}_q})|}{|G_{d^k}({{\bf F}_q})|}\cdot
q^{\sum_{l<r}n\cdot d^l}(q^{n\cdot d_r}-1).$$
Applying Theorem \ref{Betti}, we get:

\begin{theorem}\label{summationformula} For all $(Q,d,n,\Theta)$ as before, we have
$$P_{M_{d,n}^\Theta(Q)}(q)=\frac{1}{q-1}\sum_{d^*}(-1)^{s-1}(q^{\sum_{k\leq k_0}n\cdot d^k}-1)\cdot q^{-\sum_{k<l}\langle d^l,d^k\rangle}\prod_k\frac{|R_{d^k}({{\bf F}_q})|}{|G_{d^k}({{\bf F}_q})|},$$
the sum running over all semi-admissible decompositions of $d$.
\end{theorem}

\examples 
\begin{enumerate}

\item We consider the dimension vector $d$ given by $d_i=2$ and $d_j=2k$ in the case of Example D, the $m$-Kronecker quiver. The function $P_d(q)$ equals
$$\frac{1}{q(q-1)^2}\cdot(\frac{1}{q+1}\left[{2m}\atop{2k}\right]-\sum_{l=0}^{k-1}q^{(m-2k+l)l}\left[{m\atop l}\right]\left[m\atop{2k-l}\right])$$
by \cite[Section 7]{RHNS}, using the standard notation for $q$-binomial coefficients. The only relevant dimension vector $e$ on the right hand side of the formula in Corollary \ref{recursion} is $$e=d/2\mbox{ with }P_e(q)=\left[{m\atop {k}}\right],$$
again by \cite{RHNS}. After some computation, we arrive at the following formula for the Poincar\'e polynomial of the smooth model:
$$\frac{q^{n_1+kn_2}-1}{q(q-1)^2}\cdot(\frac{q^{n_1+kn_2}+1}{q+1}\left[{{2m}\atop{2k}}\right]-(q^{n_1+kn_2}+1)\times$$
$$\times \sum_{l=0}^{k-1}q^{(m-2k+l)l}\left[{m\atop l}\right]\left[{m\atop{2k-l}}\right]-q^{(m-k)k}\left[{m\atop k}\right]^2).$$

\item In the case of Example F with $r$ even, the function $P_d(q)$ equals
$$\frac{1}{q(q-1)^2}((q+1)^{r-1}-\sum_{l=0}^{r/2-1}{r\choose l}q^l),$$
and the Poincar\'e polynomial of the smooth model equals
$$\frac{1}{q(q-1)}((q+1)^r-(q+1)\sum_{l=0}^{r/2-1}{r\choose l}q^l-q^{r/2}{r\choose{r/2}}).$$

\end{enumerate}

\section{Hilbert schemes of path algebras}\label{hilbertpath}

From now on, we consider the special case $\Theta=0$, and set $${\rm H}_{d,n}(Q):=M_{d,n}^0(Q)$$ in this case. From the general theory, we have a projective morphism $${\rm H}_{d,n}(Q)\rightarrow M_d^{\rm ssimp}(Q)$$ to the affine moduli of semisimple representations. The points of ${\rm H}_{d,n}(Q)$ correspond to pairs $(M,f)$ consisting of an arbitrary representations $M$, together with a map $f:V\rightarrow M$ whose image generates the representation $M$. Equivalently, the points correspond to pairs consisting of a representation $M$, together with a surjection $P^{(n)}\rightarrow M$. In analogy to coherent sheaves, ${\rm H}_{d,n}(Q)$ can thus be viewed as a Hilbert scheme. In the special case $n_i=1$ for all $i\in I$, we have a surjection $kQ\rightarrow M$, thus ${\rm H}_{d,1}(Q)$ parametrizes left ideals $I$ in $kQ$ such that $\dimv kQ/I=d$.\\[2ex]
We will first derive an effective criterion for non-emptyness of ${\rm H}_{d,n}(Q)$ in terms of the Euler form of $Q$ and the support ${\rm supp}(d)$ of the dimension vector, which is defined as the full subquiver of $Q$ supported on vertices $i\in I$ such that $d_i\not=0$.

\begin{theorem}\label{criterionhilb} ${\rm H}_{d,n}(Q)$ is non-empty if and only if the following conditions hold:
\begin{enumerate}
\item $n_i-\langle d,i\rangle\geq 0$ for all $i\in I$,
\item for any $i\in I$ such that $d_i\not=0$, there exists a vertex $j$ such that $n_j\not=0$, and such that there exists a path in ${\rm supp}(d)$ from $j$ to $i$.
\end{enumerate}
\end{theorem}

\proof Assume that ${\rm H}_{d,n}(Q)$ contains a point $(M,f)$, where $f$ is viewed as a map $f:V\rightarrow M$ whose image generates the representation $M$. Then it is obviously necessary that $$M_i={\rm Im}(f_i)+\sum_{\alpha:j\rightarrow i}M_\alpha(M_j),$$
yielding the dimension estimate
$$\dim M_i\leq n_i+\sum_{\alpha:j\rightarrow i}\dim M_j,$$
or, equivalenty, $n_i\geq\langle d,i\rangle$. Now assume that $d_i\not=0$, but $n_j=0$ for any vertex $j\in I$ admitting a path from $j$ to $i$ in $Q$ (in particular, for $j=i$). Then the subrepresentation $U$ generated by the image of $f$ obviously fulfills $U_i=0$, a contradiction.\\[1ex]
To prove the converse, we want to apply the criterion Theorem \ref{criterion} for non-emptyness of moduli. Thus, we consider decompositions of $ \hat{d}$ as in Theorem \ref{criterion}, neccessarily of the form
$$ \hat{d}=d^1+\ldots+d^{k-1}+ \hat{d^k}+d^{k+1}+\ldots+d^s,$$
and of strictly descending slope with respect to the slope function $ \hat{\mu}$. But we have $$ \hat{\mu}(e)=0\mbox{ and } \hat{\mu}( \hat{e})>0\mbox{ for all }e\not=0,$$
thus the only possible such decompositions are of the form $ \hat{d}= \hat{e}+f$ for $e+f=d$ and $f\not=0$. Additionally, we have ${\rm H}_{e,n}(Q)\not=0$, thus $n_i\geq\langle e,i\rangle$ for all $i\in I$ by what is already proved, and we have
$$0=\langle  \hat{e},f\rangle=\langle e,f\rangle-n\cdot f=\sum_{i\in I}\underbrace{(\langle e,i\rangle-n_i)}_{\leq 0}f_i,$$
or, equivalently:
\begin{center}$n_i=\langle e,i\rangle$ or $f_i=0$ for all $i\in I$.\end{center}

Denote by $J\subset I$ the set of all vertices $i$ for which $e_i<d_i$, which is non-empty since $f\not=0$. For any $i\in J$, we have
$$n_i+\sum_{j\rightarrow i}e_j=e_i<d_i\leq n_i+\sum_{j\rightarrow i}d_j,$$
thus there exists some arrow $j\rightarrow i$ such that $j\in J$. Iterated application of this argument constructs (by finiteness of $Q$) a cycle $$i=i_s\rightarrow i_{n-1}\rightarrow\ldots\rightarrow i_1\rightarrow i_0=i$$ in $J$, yielding an estimate
$$e_i\geq n_i+e_{i_1}\geq n_i+n_{i_1}+e_{i_2}\geq\ldots\geq n_i+\ldots+n_{i_s}+e_i.$$
From this we can conclude $n_i=0$ and thus $e_i=\sum_{j\rightarrow i}e_j$. Using again the above cycle, we have an estimate
$$e_i=e_{i_1}+\sum_{i_1\not=j\rightarrow i}e_j\geq e_{i_2}+\sum_{i_2\not=j\rightarrow i_1}e_j+\sum_{i_1\not=j\rightarrow i}e_j\geq
\ldots\geq e_i+\sum_k\sum_{i_{k+1}\not=j\rightarrow i_k}e_j,$$
and thus $e_j=0$ for any vertex $j$ admitting an arrow $\alpha:j\rightarrow i_k$ pointing into the cycle.\\[1ex]
We have $d_i\not=0$, thus by assumption, there exists a vertex $j$ such that $n_j\not=0$ (and thus $j\not\in J$) and a path from $j$ to $i$ in the support of $d$. we conclude that there exists an arrow $\alpha:j\rightarrow l$ in this path such that $j\not\in J$ and $l\in J$. By the above conclusion, we infer $0=e_j=d_j$, a contradiction.\hb

\remark That the second condition in the theorem is necessary (in contrast to the case of quivers without oriented cycles, see \cite{RF}) can be seen for instance in the following example: assume $Q$ consists of two vertices $i,j$, a loop at $i$, and an arrow from $i$ to $j$. Define $n_i=0$ and $n_j=1$ as well as $d_i=1$ and $d_j=0$. Then the first condition of the theorem is satisfied, whereas obviously the moduli space is empty.\\[2ex]
We will derive a positive formula for the Betti numbers of ${\rm H}_{d,n}(Q)$ from Theorem \ref{genfunbetti} by a purely formal argument. The cell decomposition constructed in the next section will provide a more conceptual positive formula; the final section of this paper will show how these two positive formulas are related.\\[1ex]
Given a quiver $Q$ and a dimension vector $d$, we consider the set of multipartitions $\lambda\in\Lambda_d$, by which we mean a tuple $(\lambda^i)_{i\in I}$ of partitions $$\lambda^i:(\lambda^i_1\geq\ldots\geq\lambda^i_{d_i}\geq 0)$$ consisting of non-negative integers. We formally define $\lambda^i_0$ as some large integer (large enough for none of the finitely many conditions $\lambda^i_0<C$ appearing below to be fulfilled). We define the weight $|\lambda|$ of a multipartition by $$|\lambda|=\sum_{i\in I}\sum_{k=1}^{d_i}\lambda^i_k.$$
Furthermore, we denote by $S_{d,n}$ the subset of $\Lambda_d$ consisting of multipartitions with the following property:
\begin{center} for any $0\leq e<d$, there exists a vertex $i\in I$ such that $\lambda^i_{d_i-e_i}<n_i-\langle e,i\rangle$.
\end{center}

\begin{theorem}\label{Thm:Poincare} For any $(Q,d,n)$ as before, we have $$P_{{\rm H}_{d,n}(Q)}(q)=q^{n\cdot d-\langle d,d\rangle}\sum_{\lambda\in S_{d,n}}q^{-|\lambda|}.$$
\end{theorem}

\proof Since $\Theta=0$, any representation is semistable, thus $R_d^{\rm sst}(Q)=R_d(Q)$, and Theorem \ref{genfunbetti} reads
$$\sum_f\frac{|R_f(Q)({\bf F}_q)|}{|G_f({\bf F}_q)|}t^f\cdot\sum_eP_{{\rm H}_{e,n}(Q)}(q)t^e=\sum_d q^{n\cdot d}\frac{|R_d(Q)({\bf F}_q)|}{|G_d({\bf F}_q)|}t^d.$$
Multiplying on the left hand side and comparing coefficients of the various $t^d$, this reads
$$\sum_{e+f=d}q^{-\langle f,e\rangle}\frac{|R_f(Q)({\bf F}_q)|}{|G_f({\bf F}_q)|}P_{{\rm H}_{e,n}(Q)}(q)=q^{n\cdot d}\frac{|R_d(Q)({\bf F}_q)|}{|G_d({\bf F}_q)|}$$
for any $d\in{\bf N}I$.\\[1ex]
We will rewrite the fractions appearing on both sides as power series in the variable $q^{-1}$. To do this, we first note that
$$\frac{1}{(1-t)(1-t^2)\ldots(1-t^n)}=\sum_{\lambda}t^{|\lambda|},$$
the sum running over all partitions $\lambda:(\lambda_1\geq\ldots\geq \lambda_n\geq 0)$. This then gives us
$$\frac{|R_d(Q)({\bf F}_q)|}{|G_d({\bf F}_q)|}=\frac{q^{\sum_{\alpha:i\rightarrow j}d_id_j}}{\prod_{i\in I}|{\rm GL}_{d_i}({\bf F}_q)|}=\frac{q^{-\langle d,d\rangle}}{\prod_{i\in I}q^{-d_i^2}|{\rm GL}_{d_i}({\bf F}_q)|}$$
$$=q^{-\langle d,d\rangle}\prod_{i\in I}\frac{1}{(1-q^{-1})\cdot\ldots\cdot(1-q^{-d_i})}=q^{-\langle d,d\rangle}\sum_{\lambda\in\Lambda_d}q^{-|\lambda|}.$$
Substituting this identity in the above formula and multiplying both sides by $q^{\langle d,d\rangle-n\cdot d}$ yields
$$\sum_{e+f=d}q^{\langle e,f\rangle-n\cdot f}\sum_{\rho\in\Lambda_f}q^{-|\rho|}q^{\langle e,e\rangle-n\cdot e}P_{{\rm H}_{e,n}(Q)}(q)=\sum_{\lambda\in\Lambda_d}q^{-|\lambda|}.$$
Substituting the claimed formula for ${\rm H}_{e,n}(Q)$, we arrive at the identity
$$\sum_{e+f=d}\sum_{\rho\in\Lambda_f}\sum_{\nu\in S_{e,n}}q^{-|\rho|-(n\cdot f-\langle e,f\rangle)-|\nu|}=\sum_{\lambda\in\Lambda_d}q^{-|\lambda|}.$$

We prove this identity by induction over $d$, constructing a bijective map $$\Lambda_d\rightarrow \coprod_{e+f=d}\Lambda_f\times S_{e,n}$$
in such a way that weights of multipartitions correspond as in the claimed identity (the induction starts with $d=0$, for which $\Lambda_0=S_{0,n}$ consists of a single element of weight zero, namely an $I$-tuple of empty partitions). Suppose we are given a multipartition $\lambda\in\Lambda_d$. Consider the set of dimension vectors $g\leq d$ such that
$$\lambda_{d_i-g_i}^i\geq n_i-\langle g,i\rangle\mbox{ for all }i\in I,$$
which is non-empty since $g=d$ belongs to it by definition. Let $e$ be a minimal element in this set, and define
$f=d-e$. We then construct $\rho\in\Lambda_f$ by $$\rho^i_k=\lambda^i_k-(n_i-\langle e,i\rangle),$$
and $\nu\in\Lambda_e$ by
$$\nu^i_{k}=\lambda^i_{k-(d_i-e_i)}.$$
We claim that $\nu$ in fact belongs to $S_{e,n}$: otherwise, there exists a dimension vector $e'<e$ such that for all $i\in I$, we have
$$n_i-\langle e,i\rangle\leq \nu^i_{e_i-e_i'}=\lambda^i_{d_i-e_i'},$$
contradicting minimality of $e$.\\[1ex]
To construct a converse map, suppose we are given a pair $(\rho,\nu)$ as above. Thus $S_{e,n}\not=0$, and by the inductive assumption, the claimed formula already holds for $e$, resulting in ${\rm H}_{e,n}(Q)\not=\emptyset$. By Theorem \ref{criterionhilb}, we thus have  $n_i-\langle e,i\rangle\geq 0$ for any $i\in I$. This allows us to add $n_i-\langle e,i\rangle$ to any constituent of $\rho^i$, yielding a multipartition $ \hat{\rho}$. We define $\lambda$ as the concatenation of $ \hat{\rho}$ and $\nu$, for which we have to assure that $\nu^i_k\leq n_i-\langle e,i\rangle$ for any $i\in I$ and any $k$. To see this, we consider the dimension vector $e'=e-i$ for some vertex $i$. Working out the defining condition of $S_{e,n}$ in this special case, we get
$$\nu^i_1<n_i-\langle e,i\rangle+\langle i,i\rangle\leq n_i-\langle e,i\rangle,$$
which gives the desired estimate.\hb

\remark In the case of the $m$-loop quiver, this result was obtained in \cite[Section 5]{RNCH}.

\section{Cell decomposition of \texorpdfstring{$\Hilb_{d,n}\left(Q\right)$}{Hilb[d,n](Q)}}\label{sec:NCHilbert}
%\section{Cell decomposition of $\Hilb_{d,n}\left(Q\right)$}\label{sec:NCHilbert}
Recall the definition of the non-commutative Hilbert scheme from Section \ref{hilbertpath}.
Our aim is to give a cell decomposition of $\Hilb_{d,n}\left(Q\right)$, i.e. a filtration by closed subvarieties
\begin{equation*}
	\Hilb_{d,n}\left(Q\right)=X_0\stackrel{\text{closed}}{\supseteq} X_1\stackrel{\text{closed}}{\supseteq} X_2
	\stackrel{\text{closed}}{\supseteq}\ldots\stackrel{\text{closed}}{\supseteq}X_s=\emptyset
\end{equation*}
such that $X_i\ohne X_{i+1}$ is isomorphic to a disjoint union of affine spaces for $i=0,\ldots,s-1$. If such a decomposition exists the number of affine pieces of dimension $n$ is known to be the Betti number $\dim H^{2n}\left(\Hilb_{d,n}\left(Q\right)\right)$, and we know $\dim H^{2n+1}\left(\Hilb_{d,n}\left(Q\right)\right)=0$ for all $n$.
In a first step we will define certain affine subsets $U_{S_*}$ of $\Hilb_{d,n}\left(Q\right)$ such that $\Hilb_{d,n}\left(Q\right)=\bigcup_{S_*}U_{S_*}$. In the next step we will make this union disjoint by reducing the $U_{S_*}$ to smaller spaces $Z_{S_*}\subseteq U_{S_*}$ still covering $\Hilb_{d,n}\left(Q\right)$. The index set from which the $S_*$ are taken is a set of certain forests which we will describe later.

Let $Q$ be a quiver, $I$ its set of vertices, $d,n\in\Natural I$ dimension vectors. For each vertex $q\in I$ define a tree $T_q$ as follows: The vertices of $T_q$ are paths in $Q$, starting in $q$. Edges in $T_q$ are of the form $\left(\alpha_1\cdots\alpha_m\right)\to\left(\alpha_1\cdots\alpha_m\alpha_{m+1}\right)$, i.\,e. prolongations of paths in $Q$ by exactly one arrow. It is clear that the empty path in $q$ is the uniquely determined source of $T_q$. A subtree $S_q\subseteq T_q$ of a tree $T_q$ as above is a subset $S_q$ of $T_q$ which is closed under predecessors. We denote by a forest a tuple of trees, a subforest $\mathcal{F}\subseteq\mathcal{G}$ is a forest $\mathcal{F}$ whose trees are subtrees of the trees of $\mathcal{G}$.\\
Assign to $Q$ with dimension vector $n$ the forest $\mathcal{F}_n\left(Q\right)\coloneq\left(T_{q,1},\ldots,T_{q,n_q}\right)_{q\in I}$ consisting of $n_q$ copies of the tree $T_q$ constructed above for each $q\in I$. The elements of $\mathcal{F}_n\left(Q\right)$ can be labelled in the form $\left(q,i,w\right)$ which means the path $w$ in the $i$-th copy of $T_q$ which will be denoted by $T_{q,i}$. A vertex $\left(q,i,w\right)\in\mathcal{F}_n\left(Q\right)$ with $w=\alpha_1\cdots\alpha_m$ is called a $j$-vertex for a $j\in I$ if $\alpha_m$ points towards $j$ in $Q$; we write $t\left(q,i,w\right)=t\left(w\right)=j$.
\begin{examp}
	Take for example the quiver
	\begin{center}
%		\begin{pspicture}(6,1.5)
%			\psset{nodesep=3pt}
%			\rput(0.5,0.75){$Q:$}
%			\rput(1.5,0.75){\rnode{1}{$1$}}
%			\rput(3.5,0.75){\rnode{2}{$2$}}
%			\rput(5.5,0.75){\rnode{3}{$3$}}
%			\ncline{->}{2}{1}\naput{$\alpha$}
%			\ncarc[arcangle=20]{->}{2}{3}\naput{$\beta$}
%			\ncarc[arcangle=20]{->}{3}{2}\naput{$\gamma$}
%		\end{pspicture}
		\includegraphics{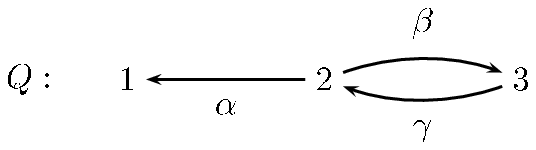}
	\end{center}
	with dimension vector $n=\left(1,1,2\right)$.
	Then we have $T_1=\emptyset$ since no arrows start in vertex 1. $T_2$ has the following form
	\begin{center}
%		\begin{pspicture}(6,5)
%			\psset{nodesep=3pt}
%			\rput(1.5,4.5){\rnode{0}{$()$}}
%			\rput(0.5,3.5){\rnode{a}{$\alpha$}}
%			\rput(2.5,3.5){\rnode{b}{$\beta$}}
%			\rput(3.5,2.5){\rnode{bc}{$\gamma\beta$}}
%			\rput(4.5,1.5){\rnode{bcb}{$\beta\gamma\beta$}}
%			\rput(5.5,0.5){\rnode{cbcb}{}}
%			\ncline{->}{0}{a}
%			\ncline{->}{0}{b}
%			\ncline{->}{b}{bc}
%			\ncline{->}{bc}{bcb}
%			\ncline[linestyle=dotted]{->}{bcb}{cbcb}
%		\end{pspicture}
		\includegraphics{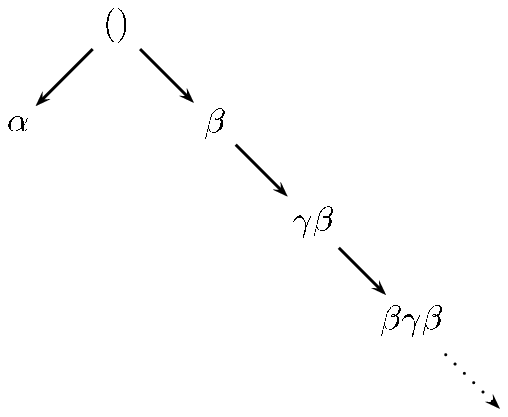}
	\end{center}
	with one infinite line, whereas $T_3$ is a infinite line with the words $()$, $\gamma$, $\beta\gamma$, $\gamma\beta\gamma$,\dots.\\
	In this case $\mathcal F_n\left(Q\right)$ consists of one copy of $T_1$ (which is the empty tree), one copy of $T_2$ and two copies of $T_3$.
\end{examp}

\begin{definition}\label{Def:Order}
	The vertices in each tree of $\mathcal{F}_n\left(Q\right)$ are totally ordered in the lexicographical order:
	Choose an arbitrary total order on $I$ and for each pair $\left(i,j\right)$ of vertices a partial order on the arrows $i\to j$ in $Q_1$, that is enumerate parallel arrows. This gives an order on $Q_1$: Given two arrows $\alpha\colon i\to j$ and $\beta\colon k\to\ell$ set $\alpha<\beta$ if one of the following conditions holds:
	\begin{enumerate}
		\item $i=k$ and $j=\ell$ and $\alpha<\beta$ in the enumeration chosen above,
		\item $i=k$ and $j<\ell$,
		\item $i<k$.
	\end{enumerate}
	Now we can construct the lexicographical order on trees as follows:
	Take two paths $w=\left(\alpha_1\cdots\alpha_r\right)$ and $w'=\left(\beta_1\cdots\beta_s\right)$ and let $k$ be minimal with $\alpha_k\not=\beta_k$; set $w<w'$ if $\alpha_k<\beta_k$. If no such $k$ exists set $w<w'$ if $r<s$.\\
	Using that we can define a total order on the vertices in $\mathcal{F}_n\left(Q\right)$: Let $\left(q,i,w\right)<\left(q',i',w'\right)$ if one of the following holds:
	\begin{enumerate}
		\item $q<q'$ in the total order of $I$,
		\item $q=q'$ and $i<i'$,
		\item $q=q'$ and $i=i'$ and $w<w'$ in the lexicographical order mentioned above.
	\end{enumerate}
	For trees $S=\left\{w_1<\ldots<w_j\right\}$ and $S'=\left\{w'_1<\ldots<w'_k\right\}$ set $S_*<S'_*$ if $\left|S_*\right|>\left|S'_*\right|$ or $w_\ell<w'_\ell$ in the order of the paths defined above where $\ell$ is minimal with $w_\ell\not=w'_\ell$. The same definition also gives an order on the forests.
\end{definition}
Since every tree $S_{q,i}$ in a subforest $F$ of $\mathcal{F}_n\left(Q\right)$ has a unique successor in $F$ as showed above denote the index of this tree by  $\succ\left(q,i\right)$.

In the following we will concentrate on subforests of $\mathcal{F}_n\left(Q\right)$.

\begin{definition}\label{Def:US}
	For $q\in I$ let $V_q$ be a $\Complex$-vector space of dimension $n_q$ with basis $\left(v_{q,i}\right)_i$. For some quiver datum $\left(Q,d,n\right)$ and a subforest $S_*$ of $\mathcal{F}_n\left(Q\right)$ let $U_{S_*}$ be the set of equivalence classes of tuples $\overline{\left(\left(M_\alpha\right)_{\alpha\in Q_1},\left(f_q\right)_{q\in I}\right)}\in\Hilb_{d,n}\left(Q\right)$ such that $\left(M_w\circ f_q\left(v_{q,i}\right)\right)_{\left(q,i,w\right)\in S_*}$ form bases of all $M_q$ for $q\in I$. Here $M_w$ stands for $M_{\alpha_\ell}\circ\cdots\circ M_{\alpha_1}$ if $w=\alpha_1\cdots\alpha_\ell$.
\end{definition}
The $U_{S_*}$ are open subsets of $\Hilb_{d,n}\left(Q\right)$ since the defining condition can be reformulated as the nonvanishing of a determinant. It is therefore easy to see that $\left(U_{S_*}\right)$ for all subforests $S_*$ of $\mathcal{F}_n\left(Q\right)$ form an open covering of $\Hilb_{d,n}\left(Q\right)$ because of the stability condition which in this case states that stable points correspond to exactly those pairs $\left(M,f\right)$ such that $M$ is generated by the image of $f$ (c.f. \cite[Corollary 3.3]{RNCH}). Note furthermore that the $U_{S_*}$ are affine analogous to \cite[Lemma 3.4]{RNCH}. We denote the set of all forests $S_*$ such that $U_{S_*}\not=\emptyset$ by $\Phi_{d,n}$.

\begin{definition}\label{Def:CS}
	For a finite subforest $S_*$ of $\mathcal{F}_n\left(Q\right)$ let $C\left(S_*\right)$ be the set of all vertices $\left(q,i,w\right)\in\mathcal{F}_n\left(Q\right)$ such that $w=\alpha_1\cdots\alpha_\ell\notin S_{q,i}$, but $\alpha_1\cdots\alpha_{\ell-1}\in S_{q,i}$. So one can think of $C\left(S_*\right)$ as a kind of corona of $S_*$.
\end{definition}

\begin{lemma}\label{Lem:Completion}
	Let $\left(M_*,f_*\right)\in\Hilb_{d,n}\left(Q\right)$ and $\bar{S}_*$ a forest such that the elements $\left(M_wf_q\left(v_{q,i}\right)\right)_{\left(q,i,w\right)\in\bar{S}_*}$ are linearly independent. Then there is a forest $S'_*\supseteq\bar{S}_*$ such that $\left(M_*,f_*\right)\in U_{S'_*}$.
\end{lemma}
\begin{proof}
	By induction on $\left|\bar{S}_*\right|$. Let $S_*\coloneq\bar{S}_*\cup C\left(\bar{S}_*\right)$ and
	\begin{align*}
		\bar{U}_j&\coloneq\left<\left(M_wf_q\left(v_{q,i}\right)\right)_{\left(q,i,w\right)\in\bar{S}_*}\right>_{t\left(w\right)=j}\\
		U_j&\coloneq\left<\left(M_wf_q\left(v_{q,i}\right)\right)_{\left(q,i,w\right)\in S_*}\right>_{t\left(w\right)=j}
	\end{align*}
	for all $j\in I$. Of course we have
	\begin{equation}\label{eq:Compdim}
		\sum_{j\in I}\dim U_j\geq\sum_{j\in I}\dim\bar{U}_j.
	\end{equation}
	If this is an equality we have $U_j=M_j$ for all $j\in I$ since by assumption $\left<\Complex\left<M_*\right>f_q\left(v_{q,i}\right)\right>_{t\left(w\right)=j}=M_j$. Otherwise there exists $\left(q,i,w\right)\in S_*\ohne\bar{S}_*=C\left(S_*\right)$ such that $M_wf_q\left(v_{q,i}\right)\notin\bar{U}_j$ for all $j\in I$. Append $\left(q,i,w\right)$ to $S_*$ and start again with $S_*$ instead of $\bar{S}_*$.\\
	The algorithm stops after finitely many steps since all vector spaces involved have finite dimension.
\end{proof}

\begin{definition}\label{Def:ZS}
	For a subforest $S_*$ of $\mathcal{F}_n\left(Q\right)$ denote by $Z_{S_*}$ the set of tuples $\left(M_*,f_*\right)\in U_{S_*}$ such that for all $\left(q,i,w\right)\in C\left(S_*\right)$ the following holds:
	\begin{equation}
		M_w\circ f_q\left(v_{q,i}\right)\in\left<M_{w'}f_{q'}\left(v_{q',i'}\right)\right>_{\substack{S_*\ni\left(q',i',w'\right)<\left(q,i,w\right)\\ t\left(w'\right)=t\left(w\right)}}.
		\label{eq:DefZ}
	\end{equation}
\end{definition}
By definition the sets $Z_{S_*}$ are also affine due to the fact that they arise from the $U_{S_*}$ by eliminating some generators. Now we can formulate the main result of this section:
\begin{theorem}\label{Th:ZU}
	We have
	\begin{equation*}
		Z_{S_*}=U_{S_*}\ohne\bigcup_{S'_*<S_*}U_{S'_*}.
	\end{equation*}
\end{theorem}
\begin{corollary}
	Let $Q$ be a quiver, $d,n\in\Natural I$. Then $\Hilb_{d,n}\left(Q\right)$ has a cell decomposition whose cells are parametrised by $\Phi_{d,n}$.
\end{corollary}
\begin{proof}
	For each forest $S_*\in\Phi_{d,n}$ define
	\begin{equation*}
		A_{S_*}\coloneq\Hilb_{d,n}\left(Q\right)\setminus\bigcup_{S'_*<S_*}Z_{S'_*}
	\end{equation*}
	which is a closed subvariety of $\Hilb_{d,n}\left(Q\right)$. The enumeration of the subforests in the chosen order now gives the required filtration.
\end{proof}
We will give the proof of the theorem by showing each inclusion seperately.
\begin{lemma}\label{Lem:ZinU}
	$Z_{S_*}\subseteq U_{S_*}\ohne\bigcup_{S'_*<S_*}U_{S'_*}$.
\end{lemma}
\begin{proof}
	The inclusion $Z_{S_*}\subseteq U_{S_*}$ is clear by definition. So let us assume there is a tuple $\left(M_*,f_*\right)\in Z_{S_*}\cap U_{S'_*}$ for a forest $S'_*<S_*$. By definition of $Z_{S_*}$ we have condition \eqref{eq:DefZ} for all $\left(q,i,w\right)\in C\left(S_*\right)$.
	Let $\left(q,i\right)$ be maximal with respect to the property $S'_{q',i'}=S_{q',i'}$ for all $\left(q',i'\right)\leq\left(q,i\right)$ which means $q'<q$ or $q=q'$ and $i'<i$. By definition of the order (see Definition \ref{Def:Order}) we have to discuss two cases:
	\begin{enumerate}
		\item $\left|S'_{\succ\left(q,i\right)}\right|>\left|S_{\succ\left(q,i\right)}\right|$ and all predecessors match. Then we have
		\begin{align*}
			\sum_{\left(r,j\right)\leq\succ\left(q,i\right)}\sum_{s\in I}\dim\left(\Complex\left<M_w\right>f\left(v_{r,j}\right)_{t\left(w\right)=s}\right)
			&\geq\sum_{\left(s,k\right)\leq\succ\left(q,i\right)}\left|S'_{s,k}\right|\\
			&>\sum_{\left(s,k\right)\leq\succ\left(q,i\right)}\left|S_{s,k}\right|\\
			=\sum_{\left(r,j\right)\leq\succ\left(q,i\right)}\sum_{s\in I}\dim\left(\Complex\left<M_w\right>f\left(v_{r,j}\right)_{t\left(w\right)=s}\right)
		\end{align*}
		and that is a contradiction.
		\item $\left|S'_{\succ\left(q,i\right)}\right|=\left|S_{\succ\left(q,i\right)}\right|$. Assume $S_{\succ\left(q,i\right)}=\left(w_j\right)_j$ and $S'_{\succ\left(q,i\right)}=\left(w'_j\right)_j$. Let $m$ be minimal with respect to the property $w'_{m+1}<w_{m+1}$. By definition of the lexicographical order of the paths we have to discuss again two cases:
		\begin{enumerate}
			\item $w'_{m+1}$ is a proper subword of $w_{m+1}$. Then we have $w'_{m+1}\in S_{\succ\left(q,i\right)}$ which is a contradiction.
			\item Assume $w_{m+1}=\left(\alpha_1\cdots\alpha_r\right)$ and $w'_{m+1}=\left(\alpha'_1\cdots\alpha'_{r'}\right)$. Let $s$ be minimal with respect to the property
			\begin{equation*}
				\alpha'_{s+1}<\alpha_{s+1}.
			\end{equation*}
			Then we have
			\begin{equation*}
				w\coloneq\alpha_1\cdots\alpha_s\alpha'_{s+1}\in S'_{\succ\left(q,i\right)}\cap C\left(S_*\right).
			\end{equation*}
			So we have in \eqref{eq:DefZ} the situation where all indices on the right side are contained in $S'_{\succ\left(q,i\right)}$, and thus a contradiction to the defining condition of $U_{S'_*}$ since by assumption $\left(M_*,f_*\right)\in U_{S'_*}$.
		\end{enumerate}
	\end{enumerate}
\end{proof}
\begin{lemma}\label{Lem:UinZ}
	$Z_{S_*}\supseteq U_{S_*}\ohne\bigcup_{S'_*<S_*}U_{S'_*}$.
\end{lemma}
\begin{proof}
	Assume there is a tuple $\left(M_*,f_*\right)\in\left(U_{S_*}\ohne\bigcup_{S'_*<S_*}U_{S'_*}\right)\ohne Z_{S_*}$. So we can choose $q,i,w$ minimal with respect to the property
	\begin{equation*}
		\left(q,i,w\right)\in C\left(S_*\right)\quad\text{and}\quad
		M_wf_q\left(v_{q,i}\right)\notin\left<M_{w'}f_{q'}\left(v_{q',i'}\right)\right>_{\substack{S_*\ni\left(q',i',w'\right)<\left(q,i,w\right),\\ t\left(w'\right)=t\left(w\right)}}.
	\end{equation*}
	Write $S_{q,i}=\left\{w_1<\ldots<w_p<w_{p+1}<\ldots\right\}$ with $w_p<w<w_{p+1}$. Define a new forest $\bar{S}_*$ containing the trees $\bar{S}_{r,j}\coloneq S_{r,j}$ for all $\left(r,j\right)<\left(q,i\right)$ and $\bar{S}_{q,i}\coloneq\left\{w_1<\ldots<w_p<w\right\}$. Then by assumption for $\left(q',i',w'\right)\in\bar{S}_*$ the elements $\left(M_{w'}f_{q'}\left(v_{q',i'}\right)\right)$ are linearly independent and by Lemma \ref{Lem:Completion} we have a forest $S'_*$ such that $\left(M_*,f_*\right)\in U_{S'_*}$.

	We will show $S'_*<S_*$ which gives us the desired contradiction. We have
	\begin{itemize}
		\item $S'_{r,j}=S_{r,j}$ for all $\left(r,j\right)<\left(q,i\right)$ since otherwise there was a $\bar{w}\in S'_{r,j}\ohne\bar{S}_{r,j}$, and we may assume $\bar{w}\in C\left(S_{r,j}\right)$. But $\left(q,i\right)$ was chosen to be minimal, so
		\begin{equation*}
			M_{\bar{w}}f_r\left(v_{r,j}\right)\in\left<M_{w''}f_{q''}\left(v_{q'',i''}\right)\right>_{%
			\substack{S_*\ni\left(q'',i'',w''\right)<\left(r,j,\bar{w}\right)\\ t\left(w''\right)=t\left(\bar{w}\right)}},
		\end{equation*}
		and hence $\left(M_*,f_*\right)\notin U_{S_*}$ which contradicts our former conclusion.
		\item In the same manner we can show that the first $p$ words of $S_{q,i}$ and $S'_{q,i}$ coincide. Since $w<w_{p+1}$ by assumption we have $S'_{q,i}<S_{q,i}$.
		\item If there is no such $w_{p+1}$ which means $S_{q,i}=\left\{w_1<\ldots<w_p\right\}$, we have $\left|S'_{q,i}\right|>\left|S_{q,i}\right|$ and hence $S'_*<S_*$.
	\end{itemize}
\end{proof}

\begin{corollary}\label{Cor:Poincare}
	From the above we derive a formula for the Poincaré polynomial:
	\begin{equation*}
		P_{\Hilb_{d,n}\left(Q\right)}\left(q\right)=\sum_{\substack{S_*\\ Z_{S_*}\not=\emptyset}}q^{\dim Z_{S_*}}.
	\end{equation*}
\end{corollary}

\section{Relating combinatorics of forests and multipartitions}\label{sec:MP}
We know from Theorem \ref{Thm:Poincare} that we can obtain the Betti numbers of $\Hilb_{d,n}\left(Q\right)$ by a weighted counting of multipartitions which satisfy certain conditions. In detail we have for a quiver $Q$ and dimension vectors $d,n\in\Natural I$ the multipartitions $\left(\lambda^i_1,\ldots,\lambda^i_{d_i}\right)_{i\in I}$ such that the following holds:
\begin{equation}\label{eq:MPcond}
	\text{for all}\ 0\leq e<d\ \text{there exists an}\ i\in I\ \text{such that}\ \lambda^i_{d_i-e_i}<n_i-\left<e,i\right>.
\end{equation}

We are now in the situation where two combinatorial formulas for the Poincaré polynomials are available, namely Theorem \ref{Thm:Poincare} and Corollary \ref{Cor:Poincare}. It is therefore natural to expect the underlying combinatorial objects to be in natural bijection. We will now construct such a bijection $\varphi$, generalizing \cite[Proposition 6.2.1]{Stanley}.

We construct $\varphi$ as follows: Let $S_*\in\Phi_{d,n}$. Define $\lambda^j_m$ as the number of $j$-vertices $\left(q,i,w\right)\in C\left(S_*\right)$ such that there are at least $m$ $j$-vertices $\left(q',i',w'\right)\in S_*$ with $\left(q,i,w\right)<\left(q',i',w'\right)$.

\begin{lemma}\label{Lem:MPwd}
	The map $\varphi\colon\Phi_{d,n}\to S_{d,n}$ is well-defined, so the multipartitions we obtain satisfy \eqref{eq:MPcond}.
\end{lemma}
\begin{proof}
	Assume there is a dimension vector $0\leq e<d$ such that for all vertices $i\in I$ the condition
	\begin{equation*}
		\lambda^i_{d_i-e_i}\geq n_i-\left<e,i\right>
	\end{equation*}
	holds. This is equivalent to the following: There is a dimension vector $0\leq e<d$ such that for all vertices $i\in I$ there are at least $n_i-\left<e,i\right>$ $i$-vertices in $C\left(S_*\right)$ smaller than the $e_i$-th $i$-vertex in $S_*$. Let $i_0\in I$ be the uniquely determined minimal vertex from the $e_i$-th $i$-vertices in $S_*$. Then we have an upper bound for the number of $i_0$-vertices in $C\left(S_*\right)$ which are smaller then the $e_{i_0}$-th $i_0$-vertex in $S_*$ as the sum of the following:
	\begin{itemize}
		\item the number of possible subforests with root $i_0$: this is of course less than $n_{i_0}$,
		\item the number of $j$-vertices in $S_*$ which are smaller than the $e_{i_0}$-th $i_0$-point. Because of the minimality of $i_0$ this number is strictly less than $\sum_{j\to i_0}e_j$.
		\item $-e_{i_0}$, since $e_{i_0}$ vertices of the above lie in $S_*$.
	\end{itemize}
	Summation gives us that the desired number is strictly less than $n_{i_0}-e_{i_0}+\sum_{j\to i_0}e_j=n_{i_0}-\left<e,i_0\right>$ which is in contradiction to our assumption from above.
\end{proof}
\begin{theorem}\label{Th:MPbij}
	The map $\varphi$ defined above is bijective.
\end{theorem}
\begin{proof}
	It suffices to show surjectivity since $\Phi_{d,n}$ and $S_{d,n}$ are finite and of the same cardinality by Theorem \ref{Thm:Poincare} and Corollary \ref{Cor:Poincare}.\\
	Let $\left(\lambda^i_m\right)_{i,m}$ be a multipartition that satisfies \eqref{eq:MPcond}. Define
	\begin{align*}
		\mu^i_m&\coloneq\lambda^i_{d_i-m}-\lambda^i_{d_i-m+1}\qquad\left(m=1,\ldots,d_i-1,\ i\in I\right)\\
		\mu^i_0&\coloneq\lambda^i_{d_i}\qquad\left(i\in I\right).
	\end{align*}
	We construct a forest as follows: Denote the vertices in $I$ by $i_0<\ldots<i_{p-1}$ and start with the empty forest $S'_*\in\mathcal{F}_n\left(Q\right)$.
	\begin{enumerate}[1{. step:}]
		\item Append the path $\left(i_0,\lambda^{i_0}_{d_{i_0}},()\right)$ to $S'_*$.
		\item Append the $\mu^{i_1}_0$-th $i_1$-vertex in $C\left(S'_*\right)$ to $S'_*$.
		\item[$p$-th step:] Append the $\mu^{i_{p-1}}_0$-th $i_{p-1}$-vertex to $S'_*$.
		\item[$\left(p+1\right)$-th step:] Append the $\mu^{i_0}_{1}$-th $i_0$-vertex in $C\left(S'_*\right)$ to $S'_*$.
	\end{enumerate}
	Proceed inductively. The translation in the proof of Lemma \ref{Lem:MPwd} shows that the obtained forest is a preimage of $\lambda$ with respect to $\varphi$.
\end{proof}
\begin{examp}\label{ex:sec8}
	As an example let us take the quiver
	\begin{center}
%		\begin{pspicture}(4,1.5)
%			\psset{nodesep=3pt}
%			\rput(0.5,0.75){$Q:$}
%			\rput(1.5,0.75){\rnode{a}{$a$}}
%			\rput(3.5,0.75){\rnode{b}{$b$}}
%			\ncarc[arcangle=20]{->}{a}{b}\naput{$\alpha$}
%			\ncarc[arcangle=20]{->}{b}{a}\naput{$\beta$}
%		\end{pspicture}
		\includegraphics{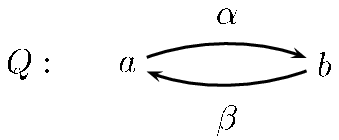}
	\end{center}
	together with dimension vectors $d=n=\left(2,2\right)$. Then $T_a$ is an infinite line, and so is $T_b$. Therefore $\mathcal F_n\left(Q\right)$ consists of two copies of each of them.
	The subforests of dimension type $d$ which parametrize the cells are listed in Table \ref{tab:mps} together with the corresponding multipartitions.
	\begin{longtable}{l|l|p{4.5cm}|l}
		\caption{List of trees/ multipartitions for example \ref{ex:sec8}}\\
		& $S_*$ & Conditions for $Z_{S_*}$ & Multipartition\\
		\hline
		\endfirsthead
		\caption[]{List of trees/ multipartitions for example \ref{ex:sec8} (cont.)}\\
		& $S_*$ & Conditions for $Z_{S_*}$ & Multipartition\\
		\hline\endhead
		1 & $\left(\left(\alpha,\alpha\beta,\alpha\beta\alpha\right),\emptyset,\emptyset,\emptyset\right)$ & & $\left(0,0\mid 0,0\right)$\\
		2 & $\left(\left(\alpha,\alpha\beta\right),\emptyset,(),\emptyset\right)$ & $\left( a,1,\alpha\beta\alpha\right)\in\left<\left( a,1,\alpha\right)\right>$ & $\left(0,0\mid 1,0\right)$\\
		3 & $\left(\left(\alpha,\alpha\beta\right),\emptyset,\emptyset,()\right)$ & $\left( b,1,()\right)\in\left<\left( a,1,\alpha\right)\right>$, $\left( a,1,\alpha\beta\alpha\right)\in\left<\left( a,1,\alpha\right)\right>$ & $\left(0,0\mid 2,0\right)$\\
		4 & $\left(\left(\alpha\right),\left(\alpha\right),\emptyset,\emptyset\right)$ & $\left( a,1,\alpha\beta\right)\in\left<\left( a,1,()\right)\right>$ & $\left(1,0\mid 0,0\right)$\\
		% B\alphaum 5
		5 & $\left(\left(\alpha\right),(),(),\emptyset\right)$ & $\left( a,2,\alpha\right)\in\left<\left( a,1,\alpha\right)\right>$, $\left( a,1,\alpha\beta\right)\in\left<\left( a,1,()\right)\right>$ & $\left(1,0\mid 1,0\right)$\\
		6 & $\left(\left(\alpha\right),(),\emptyset,()\right)$ & $\left( b,1,()\right)\in\left<\left( a,1,\alpha\right)\right>$, $\left( a,2,\alpha\right)\in\left<\left( a,1,\alpha\right)\right>$, $\left( a,1,\alpha\beta\right)\in\left<\left( a,1,()\right)\right>$ & $\left(1,0\mid 2,0\right)$\\
		7 & $\left(\left(\alpha\right),\emptyset,\left(\beta\right),\emptyset\right)$ & $\left( a,2,()\right)\in\left<\left( a,1,()\right)\right>$, $\left( a,1,\alpha\beta\right)\in\left<\left( a,1,()\right)\right>$ & $\left(2,0\mid 0,0\right)$\\
		8 & $\left(\left(\alpha\right),\emptyset,\emptyset,\left(\beta\right)\right)$ & $\left( a,2,()\right)\in\left<\left( a,1,()\right)\right>$, $\left( a,1,\alpha\beta\right)\in\left<\left( a,1,()\right)\right>$, $\left( b,1,()\right)\in\left<\left( a,1,\alpha\right)\right>$ & $\left(2,0\mid 1,0\right)$\\
		9 & $\left((),\left(\alpha\right),(),\emptyset\right)$ & $\left( a,1,\alpha\right)\in\left<\right>=0$ & $\left(0,0\mid 1,1\right)$\\
		10 & $\left((),\left(\alpha\right),\emptyset,()\right)$ & $\left( b,1,()\right)\in\left<\left( a,2,\alpha\right)\right>$, $\left( a,1,\alpha\right)=0$ & $\left(0,0\mid 2,1\right)$\\
		11 & $\left((),(),(),()\right)$ & $\left( a,1,\alpha\right)=0$, $\left( a,2,\alpha\right)\in\left<\right>=0$ & $\left(0,0\mid 2,2\right)$\\
		12 & $\left((),\emptyset,\left(\beta,\beta\alpha\right),\emptyset\right)$ & $\left( a,1,\alpha\right)=0$, $\left( a,2,()\right)\in\left<\left( a,1,()\right)\right>$ & $\left(1,0\mid 1,1\right)$\\
		13 & $\left((),\emptyset,\left(\beta\right),()\right)$ & $\left( a,1,\alpha\right)=0$, $\left( a,2,()\right)\in\left<\left( a,1,()\right)\right>$, $\left( b,1,\beta\alpha\right)\in\left<\left( b,1,()\right)\right>$ & $\left(1,0\mid 2,1\right)$\\
		14 & $\left((),\emptyset,(),\left(\beta\right)\right)$ & $\left( a,1,\alpha\right)=0$, $\left( a,2,()\right)\in\left<\left( a,1,()\right)\right>$, $\left( b,1,\beta\right)\in\left<\left( a,1,()\right)\right>$ & $\left(2,0\mid 1,1\right)$\\
		15 & $\left((),\emptyset,\emptyset,\left(\beta,\beta\alpha\right)\right)$ & $\left( a,1,\alpha\right)=0$, $\left( a,2,()\right)\in\left<\left( a,1,()\right)\right>$, $\left( b,1,()\right)\in\left<\right>=0$ & $\left(1,0\mid 2,2\right)$\\
		16 & $\left(\emptyset,\left(\alpha,\alpha\beta,\alpha\beta\alpha\right),\emptyset,\emptyset\right)$ & $\left( a,1,()\right)\in\left<\right>=0$ & $\left(1,1\mid 0,0\right)$\\
		17 & $\left(\emptyset,\left(\alpha,\alpha\beta\right),(),\emptyset\right)$ & $\left( a,1,()\right)=0$, $\left( a,2,\alpha\beta\alpha\right)\in\left<\left( a,2,\alpha\right)\right>$ & $\left(1,1\mid 1,0\right)$\\
		18 & $\left(\emptyset,\left(\alpha,\alpha\beta\right),(),\emptyset\right)$ & $\left( a,1,()\right)=0$, $\left( a,2,\alpha\beta\alpha\right)\in\left<\left( a,2,\alpha\right)\right>$, $\left( b,1,()\right)\in\left<\left( a,2,\alpha\right)\right>$ & $\left(1,1\mid 2,0\right)$\\
		19 & $\left(\emptyset,\left(\alpha\right),\left(\beta\right),\emptyset\right)$ & $\left( a,1,()\right)=0$, $\left( a,2,\alpha\beta\right)\in\left<\left( a,2,()\right)\right>$ & $\left(2,1\mid 0,0\right)$\\
		20 & $\left(\emptyset,\left(\alpha\right),\emptyset,\left(\beta\right)\right)$ & $\left( a,1,()\right)=0$, $\left( a,2,\alpha\beta\right)\in\left<\left( a,2,()\right)\right>$, $\left( b,1,()\right)\in\left<\left( a,2,\alpha\right)\right>$ & $\left(2,1\mid 1,0\right)$\\
		21 & $\left(\emptyset,(),\left(\beta,\beta\alpha\right),\emptyset\right)$ & $\left( a,1,()\right)=0$, $\left( a,2,\alpha\right)\in\left<\right>=0$ & $\left(1,1\mid 1,1\right)$\\
		22 & $\left(\emptyset,(),\left(\beta\right),()\right)$ & $\left( a,1,()\right)=0$, $\left( a,2\alpha\right)=0$, $\left( b,1,\beta\alpha\right)\in\left<\left( a,2,()\right)\right>$ & $\left(1,1\mid 2,1\right)$\\
		23 & $\left(\emptyset,(),(),\left(\beta\right)\right)$ & $\left( a,1,()\right)=0$, $\left( a,2,\alpha\right)=0$, $\left( b,1,\beta\right)\in\left<\left( a,2,()\right)\right>$ & $\left(2,1\mid 1,1\right)$\\
		24 & $\left(\emptyset,(),\emptyset\left(\beta,\beta\alpha\right)\right)$ & $\left( a,1,()\right)=0$, $\left( a,2,\alpha\right)=0$, $\left( b,1,()\right)\in\left<\right>=0$ & $\left(1,1\mid 2,2\right)$\\
		25 & $\left(\emptyset,\emptyset,\left(\beta,\beta\alpha,\beta\alpha\beta\right),\emptyset\right)$ & $\left( a,1,()\right)=0$, $\left( a,2,()\right)\in\left<\right>=0$ & $\left(2,2\mid 0,0\right)$\\
		26 & $\left(\emptyset,\emptyset,\left(\beta\right),\left(\beta\right)\right)$ & $\left( a,1,()\right)=0$, $\left( a,2,()\right)=0$, $\left( b,1,\beta\alpha\right)\in\left<\left( b,1,()\right)\right>$ & $\left(2,2\mid 1,0\right)$\\
		27 & $\left(\emptyset,\emptyset,\emptyset,\left(\beta,\beta\alpha,\beta\alpha\beta\right)\right)$ & $\left( a,1,()\right)=0$, $\left( a,2,()\right)=0$, $\left( b,1,()\right)\in\left<\right>=0$ & $\left(2,2\mid 1,1\right)$
		\label{tab:mps}
	\end{longtable}
\end{examp}

\end{document}